\documentclass[a4paper,10pt,twoside]{article} 

\usepackage[pdftex]{graphicx}
\DeclareGraphicsExtensions{.pdf,.PDF,.eps,.EPS,.png,.PNG,.jpg,.JPG,.jpeg,.JPEG}
\usepackage{lastpage}
\usepackage{latexsym}

\usepackage{amsmath,amsfonts,amssymb,amsthm,bbm,mathrsfs,exscale}
\usepackage{mathtools,color}
\usepackage{url} 
\usepackage{tikz}
\usetikzlibrary{arrows}

\usepackage{geometry}
\usepackage{bera} 

\usepackage[pdftex,pagebackref=false]{hyperref}
\hypersetup{pdfborder=0 0 0}
\hypersetup{pdfstartview={FitH}}
\usepackage[normalem]{ulem}
\usepackage[british]{babel}
\usepackage[rightcaption]{sidecap}
\usepackage{enumerate}

\makeatletter%
\newcommand{\@TITLE}{FIXME!}
\newcommand{\@SHORTTITLE}{FIXME!}
\newcommand{\@KEYWORDS}{FIXME!}
\newcommand{\@AMSSUBJ}{FIXME!}
\newcommand{\@ABSTRACT}{FIXME!}
\newcommand{\@VOLUME}{0}
\newcommand{\@PAPERNUM}{0}
\newcommand{\@YEAR}{2012}
\newcommand{\@PAGESTART}{1}
\newcommand{\@PAGEEND}{\pageref{LastPage}}  
\newcommand{\@SUBMITTED}{FIXME!}
\newcommand{\@ACCEPTED}{FIXME!}
\newcommand{\TITLE}[1]{\renewcommand{\@TITLE}{#1}}
\newcommand{\SHORTTITLE}[1]{\renewcommand{\@SHORTTITLE}{#1}}
\newcommand{\DEDICATORY}[1]{\gdef\@DEDICATORY{#1}}
\newcommand{\AUTHORS}[1]{\author{#1}}
\newcommand{\KEYWORDS}[1]{\renewcommand{\@KEYWORDS}{#1}}
\newcommand{\AMSSUBJ}[1]{\renewcommand{\@AMSSUBJ}{#1}}
\newcommand{\AMSSUBJSECONDARY}[1]{\gdef\@AMSSUBJSECONDARY{#1}}
\newcommand{\ABSTRACT}[1]{\renewcommand{\@ABSTRACT}{#1}}
\newcommand{\VOLUME}[1]{\renewcommand{\@VOLUME}{#1}}
\newcommand{\PAPERNUM}[1]{\renewcommand{\@PAPERNUM}{#1}}
\newcommand{\YEAR}[1]{\renewcommand{\@YEAR}{#1}}
\newcommand{\PAGESTART}[1]{\renewcommand{\@PAGESTART}{#1}} 
\newcommand{\PAGEEND}[1]{\renewcommand{\@PAGEEND}{#1}} 
\newcommand{\SUBMITTED}[1]{\renewcommand{\@SUBMITTED}{#1}}
\newcommand{\ACCEPTED}[1]{\renewcommand{\@ACCEPTED}{#1}}
\newcommand{\DOI}[1]{\gdef\@DOI{10.1214/\@JOURNAL.#1}}
\newcommand{\ARXIVID}[1]{\gdef\@ARXIVID{#1}}
\newcommand{\HALID}[1]{\gdef\@HALID{#1}}
\newcommand{\ARXIVPASSWORD}[1]{}

\newcommand{\EMAIL}[1]{E-mail:~\texttt{\href{mailto:#1}{#1}}}
\makeatother

\makeatletter
\newcommand{\FIRSTPAGE}{%
  \setcounter{page}{\@PAGESTART}%
  \title{\Large\bfseries\@TITLE}
 \date{%
   \ifx\@DEDICATORY\undefined%
   \else%
   \noindent%
   \emph{\small\sffamily\@DEDICATORY}
   \fi%
}%
 \maketitle\thispagestyle{empty}%
 \begin{abstract}%
   \noindent%
   \@ABSTRACT\\[1em]%
   {\footnotesize%
     \textbf{Keywords: }\@KEYWORDS.\par%
     \noindent\textbf{AMS MSC 2010: }%
     \ifx\@AMSSUBJSECONDARY\undefined%
     \noindent%
     \@AMSSUBJ.\par%
     \else%
     \noindent%
     Primary \@AMSSUBJ, Secondary \@AMSSUBJSECONDARY.\par%
     \fi%
   }
 \end{abstract}

 \smallskip

}
\makeatother

\makeatletter
\def\@MRExtract#1 #2!{#1}     
\newcommand{\MR}[1]{
  \xdef\@MRSTRIP{\@MRExtract#1 !}%
  \href{http://www.ams.org/mathscinet-getitem?mr=\@MRSTRIP}{MR-\@MRSTRIP}}
\makeatother
\newcommand{\ARXIV}[1]{\href{http://arXiv.org/abs/#1}{arXiv:#1}}
\makeatletter
\renewenvironment{thebibliography}[1]{%
  \section*{\refname
    \@mkboth{\MakeUppercase\refname}{\MakeUppercase\refname}}%
  \phantomsection%
  \addcontentsline{toc}{section}{\refname}%
  \list{\@biblabel{\@arabic\c@enumiv}}%
  {\settowidth\labelwidth{\@biblabel{#1}}%
    \small%
    \setlength{\labelsep}{0.4em}%
    \setlength{\leftmargin}{\labelwidth}%
    \addtolength{\leftmargin}{\labelsep}%
    \setlength{\itemsep}{-.25em}%
    \@openbib@code
    \usecounter{enumiv}%
    \let\p@enumiv\@empty
    \renewcommand\theenumiv{\@arabic\c@enumiv}}%
  \sloppy\clubpenalty4000\@clubpenalty\clubpenalty\widowpenalty4000%
  \sfcode`\.\@m}{%
  \def\@noitemerr{%
    \@latex@warning{Empty `thebibliography' environment}}%
  \endlist}
\makeatother
\makeatletter
 \newtheoremstyle{ejpecpbodyit}
 {3pt}
 {3pt}
 {\itshape}
 {}
 {\bfseries\sffamily}
 {.}
 { }
 {}
 \newtheoremstyle{ejpecpbodyrm}
 {3pt}
 {3pt}
 {}
 {}
 {\bfseries\sffamily}
 {.}
 { }
 {}
\makeatother

\AtBeginDocument{\FIRSTPAGE}

\AtEndDocument{\vfill}

\SHORTTITLE{Random walks and ancestry under local population
  regulation}

\TITLE{Random walks in dynamic random environments and ancestry under
  local population regulation\thanks{M.B.\ and A.D.\ gratefully
    acknowledge support by DFG priority programme SPP~1590
    \emph{Probabilistic structures in evolution} through grants BI\
    1058/3-1 respectively Pf\ 672/6-1.}}

\AUTHORS{%
  Matthias Birkner\footnote{Johannes Gutenberg University Mainz,
    Germany. \EMAIL{birkner@mathematik.uni-mainz.de}}
  \and
  Ji\v r\'\i\ \v Cern\' y\footnote{University of Vienna, Austria.
    \EMAIL{jiri.cerny@univie.ac.at}}
 \and
 Andrej Depperschmidt\footnote{University of Freiburg, Germany.
    \EMAIL{depperschmidt@stochastik.uni-freiburg.de}}
}

\KEYWORDS{Random walk; dynamical random environment; oriented
  percolation; supercritical cluster; central limit theorem in random
  environment; logistic branching random walk} 

\AMSSUBJ{60K37; 60J10; 82B43; 60K35} 

\VOLUME{21}
\YEAR{2016}
\PAPERNUM{38}
\DOI{10.1214/16-EJP4666}

\ABSTRACT{We consider random walks in dynamic random environments,
  with an environment generated by the time-reversal of a Markov
  process from the oriented percolation universality class. If the
  influence of the random medium on the walk is small in space-time
  regions where the medium is typical, we obtain a law of large
  numbers and an averaged central limit theorem for the walk via a
  regeneration construction under suitable coarse-graining.

  Such random walks occur naturally as spatial embeddings of ancestral
  lineages in spatial population models with local regulation. We
  verify that our assumptions hold for logistic branching random walks
  when the population density is sufficiently high.}

\renewcommand{\Pr}{\mathbb{P}}
\newcommand{\N}{\mathbb{N}}
\newcommand{\Z}{\mathbb{Z}}
\newcommand{\R}{\mathbb{R}}
\newcommand{\E}{\mathbb{E}}

\newcommand{\FF}{\mathcal{F}}

\newcommand{\e}{\varepsilon}

\newcommand\unnumberedfootnote[1]{ %
  \let\temp=\thefootnote %
  \renewcommand{\thefootnote}{}%
  \footnote{#1}%
  \let\thefootnote=\temp%
  \addtocounter{footnote}{-1}}

\newcommand{\abs}[1]{\lvert#1\rvert} 

\newcommand{\norm}[1]{\lVert#1\rVert} 

\newcommand{\ind}[1]{\mathbbm{1}_{\{#1\}}} 
\newcommand{\indset}[1]{\mathbbm{1}_{#1}} 
\newcommand\restr[2]{\ensuremath{\left.#1\right|_{#2}}} 

\newcommand\cs{{\mathsf{cs}}}
\newcommand\cone{{\mathsf{cone}}}
\newcommand\inn{{\mathrm{inn}}}
\newcommand\out{{\mathrm{out}}}
\newcommand\rf{{\mathrm{ref}}}
\newcommand\loc{{\mathrm{loc}}}
\newcommand\wt{\widetilde}

\theoremstyle{ejpecpbodyrm}%
\newtheorem{theorem}{Theorem}[section]%
\newtheorem{assumption}[theorem]{Assumption}%
\newtheorem{claim}[theorem]{Claim}%
\newtheorem{corollary}[theorem]{Corollary}%
\newtheorem{lemma}[theorem]{Lemma}%
\newtheorem{proposition}[theorem]{Proposition}%
\newtheorem{remark}[theorem]{Remark}%

\numberwithin{equation}{section}

\begin{document}



\section{Introduction}
\label{sec:introduction}

Let $\eta_n(x)$ be a random number of particles located at position $x
\in \Z^d$ at time $n$, where $\eta \coloneqq (\eta_n)_{n\in\Z}
\coloneqq (\eta_n(x) : x \in \Z^d)_{n\in\Z}$ is a stationary (discrete
time) Markovian particle system whose evolution can be described by
`local rules'. We assume that $\eta$ is in its unique non-trivial
ergodic equilibrium. Prototypical examples are the super-critical
discrete-time contact process, see \eqref{eq:CP} below, or systems of
logistic branching random walks, see \eqref{eq:flowdef} in
Section~\ref{sec:ancestr-line-locally}. We consider a random walk $X =
(X_k)_{k =0,1,\dots}$ that moves `backwards' through the medium
generated by $\eta$, i.e.\ given $\eta$, $X$ is a Markov chain and
given $X_k=x$, the law of the next increment is a function of $\eta$
in a finite window around the space-time point $(x,-k)$.

Our main result, see Theorem~\ref{thm:abstr-regen} in
Section~\ref{sect:abstr-setup}, provides a law of large numbers (LLN)
and an averaged central limit theorem (CLT) for $X$. Very broadly
speaking we require that the law of an increment of $X$ is close to a
fixed symmetric finite-range random walk kernel whenever the walk is
in a `good region' and that such good regions are sufficiently
frequent in a typical realisation of $\eta$. In particular we assume
that on suitably coarse-grained space-time scales, the occurrence of
good regions can be compared to super-critical oriented percolation.
The explicit assumptions are rather technical and we refer to
Sections~\ref{subsect:abstr-assumpt}--\ref{subsect:abstr-assumpt-walk}
for details.

The reversal of the natural time directions of $X$ and $\eta$
results from and is consistent with the interpretation of $X_k$ as
the position of the $k$-th ancestor of a particle picked from position
$X_0$ at time $0$. In fact, the spatial embeddings of ancestral
lineages in models with fluctuating population sizes and local
regulation are fairly complicated random walks in space-time dependent
random environments given by the time reversals of the local
population size processes. In biological applications, they are often
replaced by ordinary random walks without random environments via an
ad-hoc assumption of constant local population size; see for example
the discussion and references in \cite{BDE02} and Section~6.4 in
\cite{Etheridge2011}.

We verify that in a prototypical discrete spatial population model
with local regulation, namely logistic branching random walks with
Poisson offspring distributions, the assumptions of
Theorem~\ref{thm:abstr-regen} are satisfied if the population density
in equilibrium is sufficiently large. This allows to formulate in
Theorem~\ref{thm:lln-clt-rm} a LLN and a CLT for the ancestral lineage
of an individual sampled from such an equilibrium. Thus, we provide at
least a partial justification for the aforementioned ad-hoc
assumptions from biology in the sense that here, an ancestral lineage
will indeed behave like a random walk when viewed over large
space-time scales. This partly answers the question posed in
\cite[Chapter~4]{Dep08} in the affirmative.

\smallskip

As often in the study of random walks in random environments, the main
technical tool behind our results is a regeneration construction. The
details are somewhat involved; in principle, the medium $\eta$ can
have arbitrary dependence range and in general its time reversal
cannot be explicitly constructed using local rules. A similar problem
was faced in \cite{BCDG13} in the study of a directed random walk on
the backbone of an oriented percolation cluster. There, the particular
structure of oriented percolation allowed to jointly construct the
medium and the walk under the annealed law using suitable space-time
local operations (cf.\ \cite[Sect.~2.1]{BCDG13}) and therefrom deduce
the regeneration structure. This construction was extended
in \cite{Miller2016} to random walks on weighted oriented percolation
clusters with weights satisfying certain mixing conditions.

Here, we must use a different approach. Again very broadly speaking,
regeneration occurs after $T$ steps if the medium $\eta_{-T}$ in a
large window around $X_T$ is `good' and also the `local driving
randomness' of $\eta$ in a (large) neighbourhood of the space-time
path $\{(X_m,-m) : 0 \le m \le T \}$ has `good' properties. This
essentially enforces that the information about $\eta$ that the random
walk path has explored so far is a function of that local driving
randomness. Such a time allows to decouple the past and the future of
$X$ conditional on the position $X_T$ and $\eta_{-T}$ in a finite
window around that position. A difficulty arises from the fact that if
a regeneration fails at a given time $k$, then we have potentially
gained a lot of undesirable information about the behaviour of
$\eta_n$ at times $n < -k$ which might render successful regeneration
at a later time $\ell > k$ much less likely. We address this problem
by covering the path and the medium around it by a carefully chosen
sequence of eventually nested cones, see Figure~\ref{fig:cones1}. We
finally express $X$ as an additive functional of a Markov chain which
keeps track of the increments between regeneration times and local
configurations of $\eta$ at the regeneration points.

\smallskip

Note that random walks in dynamic random environments generated by
various interacting particle systems, in particular also by the
contact process in continuous time, have received considerable
attention recently; see for example
\cite{AvenaDosSantosVoellering:2013,AvenaJaraVoellering:2014,BethuelsenHeydenreich2015,HKS14,MV15,
RedigVoellering:2013}.
A fundamental difference to the present set-up lies in the time
directions. Traditionally, both the walker and the dynamic environment
have the same `natural' forwards time direction whereas here, forwards
time for the walk is backwards time for the medium. We also refer to
the more detailed discussion and references in \cite[Remark~1.7]{BCDG13}.

\smallskip

The rest of this manuscript is organised as follows. We first
introduce and study in Section~\ref{sect:model1} a class of random
walks which travel through the time-reversal of the discrete time
contact process, i.e., $\eta$ is literally a super-critical contact
process. We note that unlike the set-up in \cite{BCDG13}, here the
walk is also allowed to step on zeros of $\eta$. We use this simple
model to develop and explain our regeneration construction and obtain
a LLN and an annealed CLT in the `$p$ close to $1$' regime, see
Theorem~\ref{thm:LLNuCLTmodel1}. In Section~\ref{sect:abstr-setup} we
develop abstract conditions for spatial models and random walks in
dynamic random environments governed by the time-reversals of the
spatial models. Under these conditions on a coarse-grained space-time
grid we implement a regeneration construction similar to the one from
Section~\ref{sect:model1} and then obtain a LLN and an annealed CLT in
Theorem~\ref{thm:abstr-regen}. In Section~\ref{sect:model.real} we
introduce logistic branching random walks, the class of stochastic
spatial population models mentioned above. An ancestral lineage in
such a model is a particular random walk in a dynamic random
environment, see \eqref{eq:anclinquedyn}. We show that this class
provides a family of examples where the abstract conditions from
Section~\ref{sect:abstr-setup} can be verified. We believe that there
are several further classes of (population) models that satisfy the
abstract conditions from Section~\ref{sect:abstr-setup} in suitable
parameter regions. In Section~\ref{rem:furthermodels} we list and
discuss such models.

\smallskip

Finally, we note that a natural next step will be to extend our
regeneration construction to two random walks on the same realisation
of $\eta$ and to then also deduce a quenched CLT, analogous to
\cite{BCDG13}. We defer this to future work.

\paragraph{Acknowledgements:} We would like to thank Nina Gantert for
many interesting discussions on this topic and for her constant
interest and encouragement during the preparation of this work. We
also thank Stein Bethuelsen for carefully reading a preprint version
of the manuscript and his helpful comments. Finally, we are grateful
to an anonymous referee for her or his suggestions that made the
presentation more complete.

\section{An auxiliary model}
\label{sect:model1}

In this section we prove a law of large numbers and an annealed
(averaged) central limit theorem for a particular type of random walks
in dynamic random environments. The model is the simplest and the most
transparent among the models that we consider in this paper. The
proofs here contain already the main ideas and difficulties that we
will face in the following sections. It will also become clear later
how the dynamics of ancestral lineages in spatial stochastic
population models is related to this particular random walk.

\subsection{Definition of the model and results}
\label{sec:defin-model-results}

We define first the model that generates the dynamic random
environment of the random walk. Let
$\omega \coloneqq \{\omega(x,n): (x,n) \in \Z^d \times \Z\}$ be a
family of i.i.d.\ Bernoulli random variables with parameter $p>0$. We
call a site $(x,n)$ \emph{open} if $\omega(x,n)=1$ and \emph{closed}
if $\omega(x,n)=0$. Throughout the paper $\norm{\, \cdot\,}$ denotes
sup-norm. For $m \le n$, we say that there is an \emph{open path} from
$(x,m)$ to $(y,n)$ if there is a sequence $x_m,\dots, x_n\in \Z^d$
such that $x_m=x$, $x_n=y$, $\norm{x_k-x_{k-1}} \le 1$ for
$k=m+1, \dots, n$ and $\omega(x_k,k)=1$ for all $k=m,\dots,n$. In this
case we write $(x,m) \to^\omega (y,n)$, and in the complementary case
$(x,m) \not\to^\omega (y,n)$. For sets $A ,B \subseteq \Z^d$ and
$m \le n$ we write $A\times \{m\} \to^\omega B \times\{n\}$, if there
exist $x \in A$ and $y\in B$ so that $(x,m) \to^\omega (y,n)$. Here,
slightly abusing the notation, we use the convention that
$\omega(x,m)=\indset{A}(x)$ while for $k>m$ the $\omega(x,k)$ are
i.i.d.\ Bernoulli random variables as above. With this convention for
$A \subset \Z^d$, $m \in \Z$ we define the \emph{discrete time contact
  process} $\eta^A \coloneqq (\eta_n^A)_{n=m,m+1,\dots}$ driven by
$\omega$ as
\begin{align}
  \label{eq:CPA}
  \eta_{m}^A = \indset{A} \quad \text{and} \quad  \eta_n^A(x)
  \coloneqq \ind{A \times \{m\} \to^\omega (x,n)},\; \; n > m.
\end{align}
Alternatively $\eta^A = (\eta_n^A)_{n=m,m+1,\dots}$ can be viewed as a
Markov chain with $\eta^A_m = \indset{A}$ and the following local
dynamics:
\begin{align}
  \label{eq:DCP-dyn}
  \eta_{n+1}^A(x)
  & =
    \begin{cases}
      1 & \text{if $\omega(x,n+1)=1$ and $\eta_n^A(y)=1$ for some
        $y \in \Z^d$
        with $\norm{x-y} \le 1$}, \\
      0 & \text{otherwise}.
    \end{cases}
\end{align}
For a distribution $\mu$ on $\{0,1\}^{\Z^d}$ we write
$\eta^\mu=(\eta^\mu_n)_{n=m,m+1,\dots}$ for the discrete time contact
process with initial (random) configuration $\eta_m^\mu$ distributed
according to $\mu$.

\smallskip

The contact process $\eta^A=(\eta_n^A)_{n = m,m+1,\dots}$ is closely
related to \emph{oriented percolation}. In this context $A$ is the set
of `wet' sites at time $m$ and
$\{x \in \Z^d: \eta_n^A(x)=1\} \times \{n\}$ is the $n$-th time-slice
of the cluster of wet sites. Obviously for any $p<1$ the Dirac measure
on the configuration $\underline 0 \in \{0,1\}^{\Z^d}$ is a trivial
invariant distribution of the discrete time contact process. It is
well known that there is a critical percolation probability
$p_c\in (0,1)$ such that for $p>p_c$ and any non-empty
$A \subset \Z^d$ the process $\eta^A$ survives with positive
probability. Furthermore, in this case there is a unique non-trivial
extremal invariant measure $\nu$, referred to as the \emph{upper
  invariant measure}, such that, starting at any time $m \in \Z$ the
distribution of $\eta^{\Z^d}_n$ converges to $\nu$ as $n\to \infty$.

We assume $p>p_c$ throughout this section. Given a configuration
$\omega \in \{0,1\}^{\Z^d \times \Z}$, we define the \emph{stationary
  discrete time contact process} driven by $\omega$ as
\begin{align}
  \label{eq:CP}
  \eta \coloneqq (\eta_n)_{n \in \Z} \coloneqq \{\eta_n(x): x\in
  \Z^d, n\in \Z \}
  \quad \text{with} \quad \eta_n(x) \coloneqq \ind{\Z^d \times
  \{-\infty\} \to^\omega (x,n)}.
\end{align}
The event on the right hand side should be understood as
$\cap_{m\le n} \big\{\Z^d \times \{m\} \to^\omega (x,n)\big\}$. In the
above notation we have $\eta=\eta^{\Z^d}=(\eta_{n}^{\Z^d})_{n \in \Z}$.
Furthermore, since $\eta$ is a stationary Markov process, by abstract
arguments its time reversal is also a stationary Markov process with
the same invariant distribution; cf.\ Remark~\ref{rem:time-rev}.
Unless stated otherwise, throughout the paper $\eta$ denotes the
stationary discrete time contact process. Many other versions of
contact processes that will be needed in proofs will be labelled by
some superscripts similarly to the definition in~\eqref{eq:CPA}.

\medskip

To define a random walk in the random environment generated by
$\eta$, or more precisely by its time-reversal, let
\begin{align}
  \label{eq:Def.kappa}
  \kappa \coloneqq \bigl\{\kappa_n(x,y) : {n \in \Z, \,x,y \in
  \Z^d}\bigr\}
\end{align}
be a family of random transition kernels defined on the same
probability space as $\eta$, in particular
$\kappa_n(x, \,\cdot\, ) \ge 0$ and
$\sum_{y \in \Z^d} \kappa_n(x,y)=1$ holds for all $n\in \Z$ and
$x \in \Z^d$. Given $\kappa$, we consider a $\Z^d$-valued random
walk $X\coloneqq (X_n)_{n =0,1,\dots }$ with $X_0=0$ and transition
probabilities given by
\begin{align}
  \label{eq:defX}
  \Pr\big(X_{n+1}=y \, \big| \, X_n=x, \kappa \big) = \kappa_n(x,y),
\end{align}
that is, the random walk at time $n$ takes a step according to the
kernel $\kappa_n(x,\,\cdot\,)$ if $x$ is its position at time $n$. We make
the following four assumptions on the distribution of $\kappa$; see
Remark~\ref{rem:Xlocalconstr} for an interpretation.

\begin{assumption}[Locality]
  The \label{ass:local} transition kernels in the family $\kappa $
  depend locally on the time-reversal of $\eta$, that is for some
  fixed $R_\loc \in \N$
  \begin{align}
    \label{eq:ass:local}
    \kappa_n(x,\cdot) \; \text{ depends only on } \; \big\{ \omega(y,-n),
    \eta_{-n}(y) : \norm{x-y} \leq R_\loc\big\}.
  \end{align}
\end{assumption}

\begin{assumption}[Closeness to a symmetric reference measure on
  $\eta_{-n}(x)=1$]
  There \label{ass:appr-sym} is a deterministic symmetric probability
  measure $\kappa_\rf$ on $\Z^d$ with finite range $R_\rf \in \N$,
  that is $\kappa_\rf(x)=0$ if $\norm{x} > R_\rf$, and a suitably small
  $\varepsilon_\rf >0$ such that
  \begin{equation}
    \label{eq:ass:appr-sym}
    \norm{ \kappa_n(x,x+\,\cdot\,) -
      \kappa_\rf(\,\cdot\,) }_{\mathrm{TV}} <
    \varepsilon_\rf   \quad \text{whenever} \quad \eta_{-n}(x)=1.
  \end{equation}
  Here $\norm{\, \cdot \, }_{\mathrm{TV}}$ denotes the total variation
  norm.
\end{assumption}

\begin{assumption}[Space-time shift invariance and spatial point
  reflection invariance]
  The \label{ass:distr-sym} kernels in the family $\kappa$ are
  shift-invariant on $\Z^d \times \Z$, that is, using notation
  \begin{align*}
    \theta^{z,m} \omega (\,\cdot\,,\,\cdot\,)=\omega
    (z+\,\cdot\,,m+\,\cdot\,),
  \end{align*}
 we have
  \begin{align}
    \label{eq:sp-time-stat}
    \kappa_n(x,y)(\omega ) = \kappa_{n+m}
    (x+z,y+z)(\theta^{z,m}\omega).
  \end{align}
  Moreover, if $\varrho$ is the spatial point reflection operator
  acting on $\omega$, i.e., $\varrho \omega (x,n) = \omega(-x,n)$ for
  any $n \in \Z$ and $x \in \Z^d$, then
  \begin{align}
    \label{eq:point-refl-inv}
   \kappa_n(0,y)(\omega )=\kappa_n(0,-y)(\varrho \omega).
  \end{align}
\end{assumption}

\begin{assumption}[Finite range]
  There \label{ass:finite-range} is $R_\kappa \in \N$ such that a.s.\
  \begin{align}
    \label{eq:ass:kappa2}
    \kappa_n(x,y) = 0 \quad \text{whenever \;} \norm{y-x}>R_\kappa.
  \end{align}
\end{assumption}

\begin{remark}[Interpretation of the assumptions]
  The \label{rem:Xlocalconstr}
  Assumptions~\ref{ass:local}--\ref{ass:finite-range} are natural as
  we want to interpret the random walk as the spatial embedding of an
  ancestral lineage in a spatial population model, in which roughly
  speaking, children (if any present at a site) choose their parents
  at random from a finite neighbourhood in the previous generation.
  See also Section~\ref{sect:model.real} and in particular the
  discussion around \eqref{eq:anclinquedyn}.

  By \eqref{eq:defX} and \eqref{eq:ass:local}, we can, and often shall
  think of creating the walk from $\eta$ and $\omega$ in a local
  window around the current position and additional auxiliary
  randomness.
\end{remark}

The main result of this section is the following theorem. Its proof is
given in Section~\ref{sec:proof-theor-refthm:l}.

\begin{theorem}[LLN and annealed CLT]
  One \label{thm:LLNuCLTmodel1} can choose $0 < \varepsilon_\rf$
  sufficiently small and $p$ sufficiently close to $1$, so that if
  $\kappa$ satisfies
  Assumptions~\ref{ass:local}--\ref{ass:finite-range} then $X$
  satisfies the strong law of large numbers with speed $0$ and an
  annealed (i.e.\ when averaging over both $\omega$ and the walk)
  central limit theorem with non-trivial covariance matrix. A
  corresponding functional central limit theorem holds as well.
\end{theorem}

\begin{remark}[Time-reversal of $\eta$, oriented percolation
  interpretation]
  In \cite{BCDG13}, the \label{rem:time-rev} stationary process $\eta$
  was equivalently parametrised via its time reversal
  \begin{align*}
    \xi \coloneqq \{ \xi(x,n): x \in \Z^d, n \in \Z\} \quad\text{with }
   \xi(x,n) \coloneqq \eta_{-n}(x).
  \end{align*}
  Then $\xi$ is the indicator of the \emph{backbone} of the oriented
  percolation cluster and it was notationally and conceptually
  convenient to use in \cite{BCDG13}, not least because then the
  \emph{medium} $\xi$ and the walk $X$ had the same positive time
  direction.

  Here, we keep $\eta$ as our basic datum because we wish to emphasise
  and in fact later use in Section~\ref{sect:abstr-setup} the
  interplay between the medium $\eta$, interpreted as describing the
  dynamics of a population, and the walk $X$, cf.\ \eqref{eq:defX}
  above, describing the embedding of an ancestral lineage.
  Furthermore, in the more general population models, as the one
  studied in Section~\ref{sect:model.real} for instance, there will be
  no natural parametrization of the time-reversal of $\eta$.
\end{remark}

Note that the assertions of Theorem~\ref{thm:LLNuCLTmodel1} are in a
sense conceptual rather than practical because the proofs of
preliminary results in Section~\ref{subs:regen-model1} require $1-p$
to be very small. Situations with $p>p_c$ but also $1-p$ appreciably
large require an additional coarse-graining step so that the arguments
from Section~\ref{sect:abstr-setup} can be applied.

In order to prove Theorem~\ref{thm:LLNuCLTmodel1} we will construct
suitable regeneration times and show that the increments of these
regeneration times as well as the corresponding spatial increments of
the walk have finite moments of order $b$ for some $b>2$. This
construction is rather intricate. The main source of complications
stems from the fact that in order to construct the random walk $X$ one
should know $\omega$ and $\eta$ in the vicinity of its trajectory;
cf.\ Remark~\ref{rem:Xlocalconstr} above. While it is easy to deal
with the knowledge of $\omega $'s, because they are i.i.d., the
knowledge of $\eta$ leads to problems. Due to definition \eqref{eq:CP}
of $\eta$ and Assumption~\ref{ass:local} on $\kappa$, this knowledge
provides non-trivial information about the \emph{past behaviour} of
$\eta$ and therefore also about the \emph{future behaviour} of $X$.
Both is not desirable at regeneration times.

More precisely, we need to deal with two types of information on
$\eta$. The first type, the \emph{negative} information, that is
knowing that $\eta_n(x)=0$ for some $n$ and $x$ is dealt with
similarly as in \cite{BCDG13}. The key observation is that such
information is essentially local: To discover that $\eta_n(x)=0$ one
should check that $\Z^d\times \{-\infty\} \not\to^\omega (x,n)$ which
requires observing $\omega $'s in a layer $\Z^d \times\{n-T,\dots,n\}$
where $T$ is a random variable with exponentially decaying tails. The
second type of information, the \emph{positive} one, that is knowing
that $\eta_n(x)=1$, is removed by making use of strong coupling
properties of the forwards-in-time dynamics of $\eta$. When at time
$-t$ we have $\eta_{-t} \ge \indset{x+\{-L,\dots,L\}^d}$ pointwise,
then there is a substantial chance that every infection, i.e., every
`$1$' of $\eta$, inside a growing space-forwards-time cone with base
point $(x,-t)$ can be traced back to
$(x+\{-L,\dots,L\}^d) \times \{-t\}$. Furthermore, whether this event
has occurred can be checked by observing the restriction of $\eta_{-t}$
to $x+\{-L,\dots,L\}^d$ and the $\omega$'s inside a suitably fattened
shell of the cone in question, in particular without looking at any
$\eta_m(y)$ for $m<-t$; see \eqref{eq:coneshell} and
Lemma~\ref{lem:CPconeconv} below. We will construct suitable random
times $T$ at which this event occurs for $\eta$ at the current
space-time position $(X_T,-T)$ of the walker and in addition the
space-time path of the walk up to $T$, $\{(X_k,-k) : 0 \le k \le T\}$,
is completely covered by a suitable cone. Such a time $T$ allows to
regenerate.

\smallskip

For the proof of Theorem~\ref{thm:LLNuCLTmodel1} we first collect some
results on the high density discrete time contact process in
Section~\ref{ss:contactprocess}. We then rigorously implement the
regeneration construction sketched above in
Section~\ref{subs:regen-model1}. Finally we prove
Theorem~\ref{thm:LLNuCLTmodel1} in
Section~\ref{sec:proof-theor-refthm:l}.

\subsection{Some results about the contact process}
\label{ss:contactprocess}

This section contains several estimates for the discrete-time contact
process $\eta$ that will be crucial for the regeneration
construction. The main results of this section are the estimates given
in Lemma~\ref{lem:drysitesbd} and Lemma~\ref{lem:CPconeconv}. We start
by recalling two well known results.

\begin{lemma} For \label{lem:exp-bound-perc} $p>p_c$ there exist
  $C(p),c(p) \in (0,\infty)$ such that
  \begin{align}
    \label{est:exp-bound-perc}
    \Pr \bigl(\Z^d\times\{-n\} \to^\omega (0,0) \; \text{ and }
    \; \Z^d\times\{-\infty\} \not\to^\omega (0,0)  \bigr) \leq
    C(p)e^{-c(p) n}, \quad n \in \N.
  \end{align}
  Moreover, we have $\limsup_{p\nearrow 1} C(p) <\infty$ and
  $\lim_{p\nearrow 1}c(p)=\infty$.
\end{lemma}
\begin{proof}
  Due to self-duality of the contact process this is a reformulation
  of the fact that for $p > p_c$ and
  $\eta^{\{0\}}=(\eta^{\{0\}}_n)_{n = 0,1,\dots}$ there exist
  $C(p), c(p) \in (0,\infty)$ such that
\begin{align}
    \label{eq:exp-bound-perc:classic}
  \Pr\bigl(\eta^{\{0\}}_n \not\equiv 0 \, \text{ and } \,
  \eta^{\{0\}}\;\text{eventually dies out}  \bigr) \le C(p) e^{-c(p) n},
  \quad n \in \N.
\end{align}
For a proof of the latter assertion we refer to e.g.\
\cite{Dur84,GH02} or Lemma~A.1 in \cite{BCDG13}.
\end{proof}

\begin{lemma}
  \label{lem:DG82coupl}
  Let $\eta^{\{0\}}=(\eta^{\{0\}}_n)_{n = 0,1,\dots}$ and let
  $\eta^\nu=(\eta_n^\nu)_{n = 0,1,\dots}$, where $\nu$ is the upper
  invariant measure. For $p$ sufficiently close to $1$ there exist
  $s_{\mathrm{coupl}}>0$, $C<\infty$, $c>0$ such that
  \begin{align}
    \label{eq:DG82coupl}
    \Pr\bigl( \eta^{\{0\}}_n(x) = \eta^\nu_n(x) \; \text{ for all } \;
    \norm{x} \le s_{\mathrm{coupl}} n  \, | \, \eta^{\{0\}}_n \not
    \equiv 0 \bigr) \ge 1 - C e^{-c n}, \quad n \in \N.
  \end{align}
\end{lemma}
\begin{proof}
  For the contact process in continuous time, this is proved in
  \cite{DG82}, see in particular (33) and (34) in Proposition~6 there.
  Although literally, \cite[Eq.~(34)]{DG82} refers to conditioning on
  $\{\eta^{\{0\}}\;\text{survives}\}$ the result follows in view of
  \eqref{eq:exp-bound-perc:classic}.
\end{proof}
\begin{remark}
  In \cite{FvZ03} it is shown (literally, for the contact process in
  continuous time) that for any $p>p_c$ and $a>0$, there is a
  $C<\infty$ such that the probability on the left hand side of
  \eqref{eq:DG82coupl} is bounded below by $1 - C n^{-a}$.

  More recently, in \cite{GaretMarchand2014} large deviations for the
  continuous time contact process in a random environment were
  studied. Among other results it is shown that exponential decay as
  in \eqref{eq:DG82coupl} holds in the supercritical case; see
  Theorem~1, Eq.~(2) there.
\end{remark}

The first main result of this section is the following lemma on
controlling the probabilities of certain \emph{negative} events; cf.\
Lemma~7 in Section~3 of \cite{Dur92} for a related result.

\begin{lemma}
  For \label{lem:drysitesbd} $p$ large enough there exists
  $\varepsilon(p) \in (0,1]$ satisfying
  $\lim_{p\nearrow 1}\varepsilon(p)= 0$ such that for any
  $V = \{ (x_i, t_i) : 1 \le i \le k \} \subset \Z^d \times \Z$ with
  $t_1 > t_2 > \cdots > t_k$, we have
  \begin{align}
    \label{eq:drysitesbd}
    \Pr\bigl( \eta_{t}(x)=0 \; \text{for all}\; (x,t) \in V \bigr) \leq
    \varepsilon(p)^k.
  \end{align}
\end{lemma}

\begin{remark} In our proof of \eqref{eq:drysitesbd} it is essential
  that all $t_i$'s are distinct. For a general set
  $V \subset \Z^d \times \Z$ space-time boundary effects can play a
  role so that the decay will only be stretched exponential in $|V|$.
  For a concrete example in the case $d=1$ consider the set
  \begin{align*}
    V & =\{ (x,n) \in \Z \times \Z : \abs{x} \le h, 0 \le n \le h\} \\
    \intertext{and the event}
    B & = \{ \omega(y,0) = 0 , \abs{y} \le h+1\} \cap
        \{ \omega(\pm(h+1),k) = 0 , k =  0,1,\dots, h \}.
  \end{align*}
  Then, obviously we have
  $B \subset \{\eta_n(x) =0 \, \text{for all} \, (x,n) \in V\}$ and
  $\Pr(B) = (1-p)^{4h+3}$.

  Note however that in Corollary~4.1. in \cite{LiggettSteif2006} it
  is shown that the upper invariant measure of the continuous time
  contact process on $\Z^d$ dominates stochastically a product measure
  on $\{0,1\}^{\Z^d}$. Therefore a bound analogous to
  \eqref{eq:drysitesbd} does hold in the situation when all $t_i$'s
  are equal. Lemma~\ref{lem:drysitesbd} can be seen as a space-time
  extension of that result in the discrete time case and in the `$p$
  large enough' regime.
\end{remark}

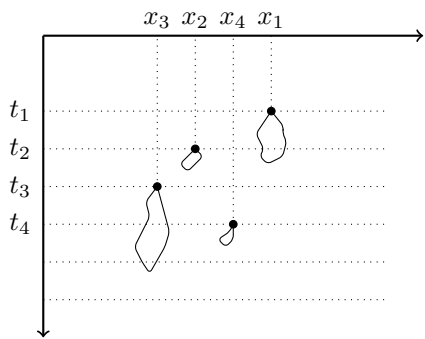
\begin{figure}[t]
  \centering
  \begin{tikzpicture}
    \draw [<-, thick] (3,1) -- (3,5); \draw [->, thick] (3,5) -- (8,5);

    \filldraw[black] (6,4) circle (0.05); \draw[dotted] (3,4) -- (7.5,4);
    \draw[dotted] (6,5)--(6,4); \draw (2.7,4) node {$t_1$};
    \draw (6,5.2) node {$x_1$};

    \draw[rounded corners =1.5] (6,4) -- (5.8,3.7) --
    (5.9,3.5)--(5.85,3.4)--(5.95,3.3) -- (6.05,3.35) -- (6.15,3.4) --
    (6.2,3.6) -- (6.15,3.7) -- (6.16,3.8) -- (6,4);

    \draw[dotted] (3,3.5) -- (7.5,3.5);
    \filldraw[black] (5,3.5) circle (0.05);
    \draw (2.7,3.5) node {$t_2$};
    \draw [dotted] (5,3.5) -- (5,5);
    \draw (5,5.2) node {$x_2$};
    \draw[rounded corners =1.5] (5,3.5) -- (4.8,3.3)-- (4.9,3.2) --
    (5.1, 3.4) -- (5,3.5);

    \filldraw[black] (4.5,3) circle (0.05);
    \draw[dotted] (3,3) -- (7.5,3);
    \draw[dotted] (4.5,3)--(4.5,5);
    \draw (2.7,3) node {$t_3$};
    \draw (4.5,5.2) node {$x_3$};

    \draw[rounded corners =1.5] (4.5,3) -- (4.35,2.8) -- (4.4,2.6) --
    (4.3,2.4) -- (4.2,2.2) -- (4.3, 2) -- (4.4,1.85) -- (4.6,2.2) --
    (4.66,2.4) -- (4.5,3);

    \draw[dotted] (3,2.5) -- (7.5,2.5);
    \draw[dotted] (5.5,2.5)--(5.5,5);
    \filldraw[black] (5.5,2.5) circle (0.05);
    \draw[rounded corners =1.5] (5.5,2.5) -- (5.45,2.4) -- (5.3,2.3)
    -- (5.4,2.2) -- (5.5,2.3) -- (5.5,2.5);
    \draw (2.7,2.5) node {$t_4$}; \draw (5.5,5.2) node {$x_4$};

    \draw[dotted] (3,2) -- (7.5,2);
    \draw[dotted] (3,1.5) -- (7.5,1.5);
  \end{tikzpicture}
  \caption{Possibly overlapping finite clusters starting at
    $V=\{(x_i,t_i)\}$ that appear in Lemma~\ref{lem:drysitesbd}.
    Here $k=4$, $D_1=2, D_2=1, D_3=3, D_4=1$, hence $M=3$,
    $S_1=1, S_2=3, S_3=6$.}
  \label{fig:caricature}
\end{figure}

\begin{proof}[Proof of Lemma~\ref{lem:drysitesbd}]
  An immediate consequence of Lemma~\ref{lem:exp-bound-perc} is that
  for every $p>p_c$
  \begin{align}
    \label{est:exp-bound-perc-hd}
    \Pr \bigl(\Z^d\times\{-n\} \to^\omega (0,0)\,  \text{ and } \,
    \Z^d\times\{-\infty\} \not\to^\omega (0,0) \bigr)
    \leq e^{-c_1(p) (n+1)}, \quad n = 0,1,2,\dots,
  \end{align}
  with some $c_1=c_1(p)>0$ satisfying $\lim_{p\nearrow 1}c_1(p) = \infty$.
  To prove \eqref{eq:drysitesbd}, we extend the finite sequence
  $\{t_1,\dots,t_k\}$ to an infinite sequence via
  $t_{k+j} \coloneqq t_k - j$, $j=1,2,\dots$, and put
  \begin{align}
    \label{eq:Di}
    D_i \coloneqq \min \left\{ \ell \in \N :  \Z^d
    \times \{t_{i+\ell}\} \not\to^\omega (x_i, t_i) \right\}.
  \end{align}
  Note that the random variables $D_i$ are upper bounds on the heights
  of the backwards-clusters of open sites attached to $(x_i,t_i)$
  given by (see Figure~\ref{fig:caricature})
  \begin{align*}
    \{(y,m) \in \Z^d \times \Z: m \le t_i, (y,m) \to^\omega
    (x_i,t_i)\}.
  \end{align*}
  For each $(x_i,t_i) \in V$ we have $\eta_{t_i}(x_i) = 0$ if and only
  if $D_i<\infty$. Thus the left-hand side of \eqref{eq:drysitesbd}
  satisfies
\begin{align}
  \Pr\big( \eta_{t}(x)=0 \; \text{for all}\; (x,t) \in V
  \big) = \Pr\Big(\bigcap_{i=1}^k \{D_i<\infty\}\Big).
\end{align}
On the event $\cap_{i=1}^k \{ D_i < \infty \}$ we further define
\begin{align*}
  S_1=1, S_2=S_1+D_{S_1}, \dots, S_{i+1}=S_i+D_{S_i}\quad \text{as long
    as}\; S_i \le k,
\end{align*}
and let $M$ be such that $S_{M-1} \le k < S_M$; see
Figure~\ref{fig:caricature}. For $i=1,\dots,M$ we set
$\widehat{D}_i \coloneqq D_{S_i}$ and $\widehat{D}_i \coloneqq \infty$
for $i>M$. Finally we set
\begin{align}
  \label{eq:Imk}
  \mathcal{I}(m,k)=\{(d_1,\dots,d_m) \in \N^m : d_1+\cdots+d_{m-1} \le
  k < d_1+\cdots+d_m\},
\end{align}
and for $(d_1,\dots,d_m) \in \mathcal{I}(m,k)$ we write
\begin{align*}
u(1)=1, \, u(2)=u(1)+d_1, \, \dots, \, u(m)=u(m-1)+d_{m-1}.
\end{align*}
Then we have
\begin{align}
  \label{eq:di-incl}
  \bigcap_{i=1}^k \{ D_i < \infty \} \subset
  \bigcup_{m=1}^k \bigcup_{ (d_1,\dots,d_m) \in \mathcal{I}(m,k)}
  \left\{ \widehat{D}_1=d_1,\dots, \widehat{D}_m=d_m \right\}.
\end{align}
Note that
  \begin{multline}
    \left\{ \widehat{D}_1=d_1,\dots, \widehat{D}_m=d_m \right\} \\ =
    \bigcap_{j=1}^m \bigl\{\Z^d \times \{t_{u(j)+d_j-1} \}\to^\omega
      (x_{u(j)},t_{u(j)}) \bigr\} \cap \bigl\{\Z^d \times
      \{t_{u(j)+d_j}\}\not\to^\omega
      (x_{u(j)},t_{u(j)}) \bigr\}.
  \end{multline}
  The events in the intersection on the right-hand side depend on
  $\omega $ restricted to disjoint sets and are thus independent.
  Furthermore we observe that the event
  \begin{equation*}
    \bigl\{\Z^d \times \{t_{u(j)+d_j-1} \}\to^\omega
    (x_{u(j)},t_{u(j)})\bigr\} \cap \bigl\{\Z^d \times \{t_{u(j)+d_j}\}\not\to^\omega
    (x_{u(j)},t_{u(j)}) \bigr\}
  \end{equation*}
  enforces that $(x_{u(j)},t_{u(j)})$ is the starting point of a
  finite (backwards) cluster of height at least
  $t_{u(j)+d_j-1}-t_{u(j)} \ge d_j-1$ (when $d_j=1$ this means that
  $\omega(x_{u(j)},t_{u(j)})$ is closed, which also gives a factor
  $1-p < 1$). Hence, using \eqref{est:exp-bound-perc-hd} we obtain
\begin{align}
\label{eq:hatDest}
\begin{split}
\Pr\bigl( & \widehat D_1 =d_1, \dots,\widehat D_m =d_m \bigr) \\
& = \prod_{j=1}^m \Pr\left(\bigl\{\Z^d \times \{t_{u(j)+d_j-1}\}\to^\omega (x_{u(j)},t_{u(j)}), \Z^d \times\{t_{u(j)+d_j}\}\not\to^\omega (x_{u(j)},t_{u(j)}) \bigr\} \right) \\
& \le \prod_{j=1}^m e^{-c_1(p) d_j} = e^{-c_1(p) \sum_{j=1}^m d_j}.
\end{split}
\end{align}
  Now \eqref{eq:hatDest} with \eqref{eq:di-incl} imply
  \begin{multline}
    \label{eq:fin-cl-estim}
    \Pr\big( \eta_{t}(x)=0 \; \text{for all}\; (x,t) \in V \big)
    \le \sum_{m=1}^k \sum_{(d_1,\dots,d_m) \in \mathcal{I}(m,k)} 
    e^{-c_1(p) \sum_{j=1}^m d_j} \\
    = \sum_{m=1}^k \sum_{s=k+1}^\infty 
    e^{-c_1(p) s} \cdot \# \{ (d_1,\dots,d_m) \in \mathcal{I}(m,k) :
    d_1+\cdots+d_m=s \}.
  \end{multline}
  By definition of $\mathcal{I}(m,k)$ for $s \ge k+1$ we have
  \begin{multline}
    \# \{ (d_1,\dots,d_m) \in \mathcal{I}(m,k) : d_1+\cdots+d_m=s \} \\
    = \# \{ (d_1,\dots,d_m) \in \mathcal{I}(m,k) : d_1+\cdots+d_m=k+1
    \} = \binom{k}{m-1} \le 2^k.
  \end{multline}
  Thus, the right hand side of \eqref{eq:fin-cl-estim} can be bounded
  by
  \begin{align}
    2^k \sum_{m=1}^k \sum_{s=k+1}^\infty e^{-c_1(p) s}
    = k 2^k \frac{e^{-c_1(p)(k+1)}}{1-e^{-c_1(p)}},
  \end{align}
  yielding the claim of the lemma.
\end{proof}

\medskip

The second main result of this section, Lemma~\ref{lem:CPconeconv}
below, is the crucial tool in the construction of a certain coupling
which will be useful to forget the positive information about $\eta$
in the regeneration construction. To state this lemma we need to
introduce more notation. For $A\subset \Z^d$ let
$\eta^A=(\eta^A_n)_{n=0,1,\dots}$ be the discrete time contact process
as defined in \eqref{eq:CPA}. For positive $b,s,h$ we write (denoting
by $\Z_+$ the set non-negative integers and by $\norm{\cdot}_2$ the
$\ell^2$-norm)
\begin{equation}
  \label{def:cone}
  \cone(b,s,h) \coloneqq \big\{ (x,n) \in \Z^d \times \Z_+ \, \colon
  \, \norm{x}_2 \leq b + s n, 0 \le n \le h \big\}.
\end{equation}
for a (truncated upside-down) \emph{cone} with base radius $b$, slope
$s$, height $h$ and base point $(0,0)$. Furthermore for
\begin{align}
  \label{eq:bs-inn-out}
  b_\inn \le b_\out \quad \text{and} \quad s_\inn < s_\out,
\end{align}
we define the \emph{conical shell} with inner base radius $b_\inn$,
inner slope $s_\inn$, outer base radius $b_\out$, outer slope
$s_\out$, and height $h \in \N \cup \{\infty\}$ by
\begin{equation}
  \label{eq:coneshell}
  \cs(b_\inn , b_\out , s_\inn , s_\out ,h)
  \coloneqq \big\{ (x,n) \in \Z^d \times \Z  :
  b_\inn  + s_\inn n \le \norm{x}_2 \le
  b_\out  + s_\out n, 0 < n \le h \big\}.
\end{equation}
The conical shell can be thought of as a difference of the \emph{outer
  cone} $\cone(b_\out,s_\out,h)$ and the \emph{inner cone}
$\cone(b_\inn,s_\inn,h)$ with all boundaries except the bottom
boundary of that difference included; see Figure~\ref{fig:cone-cs}.

\begin{figure}
  \centering
  \begin{tikzpicture}[xscale=0.8,yscale=0.8]
   \draw[->,thick] (-8,0) -- (-8,4.55);
   \draw[->,thick] (-11,0) -- (-5,0);
   \draw (-4.8,0.35) node {\footnotesize $\Z^d$};
   \draw (-8.3,4.5)  node {\footnotesize $\Z_+$};
   \draw (-7.75,4.2) node {\footnotesize $h$};

    \draw[thick] (-6.5,0) -- (-5.6,4) -- (-10.4,4) -- (-9.5,0) ;

    \draw (-6.275,1) --(-5.825,1) -- (-5.825,3);
    \draw (-6.05,0.8) node {\footnotesize $s$};
    \draw (-5.7, 1.9) node {\footnotesize $1$};

    \draw (-6.5,0)--(-6.5,-0.1)
          (-9.5,0)--(-9.5,-0.1);

    \draw (-6.5,-0.28) node {\footnotesize $b$};
    \draw (-9.55,-0.28) node {\footnotesize -$b$};


    \draw[thick] (1,0) -- (1.6,4)
                 (1.5,0) -- (2.4,4)
                 (-1,0) -- (-1.6,4)
                 (-1.5,0) -- (-2.4,4);

    \draw[thick] (1,0)--(1.5,0)
                 (-1,0) -- (-1.5,0);

    \draw[thick, dashed] (-1,0) -- (1,0);

   \begin{scope}
     \path [fill=black,opacity=0.2] (1,0) -- (1.6,4) -- (2.4,4) --
     (1.5,0) -- cycle;
     \path [fill=black,opacity=0.2] (-1,0) -- (-1.6,4) -- (-2.4,4) --
     (-1.5,0) -- cycle;
   \end{scope}

   \draw[->,thick] (0,0) -- (0,4.55);
   \draw[->,thick] (-3,0) -- (3,0);
   \draw (3.2,0.35) node {\footnotesize $\Z^d$};
   \draw (-0.3,4.5)  node {\footnotesize $\Z_+$};

  \draw[dashed, thick] (-2.4,4) -- (2.4,4);
  \draw (0.25,4.2) node {\footnotesize $h$};

 \end{tikzpicture}

 \caption{The left figure shows $\cone(b,s,h)$. The grey region in the
   figure on the right (without the bottom line) shows
   $\cs(b_\inn , b_\out , s_\inn , s_\out ,h)$ in $\Z^d \times \Z$.}
  \label{fig:cone-cs}
\end{figure}
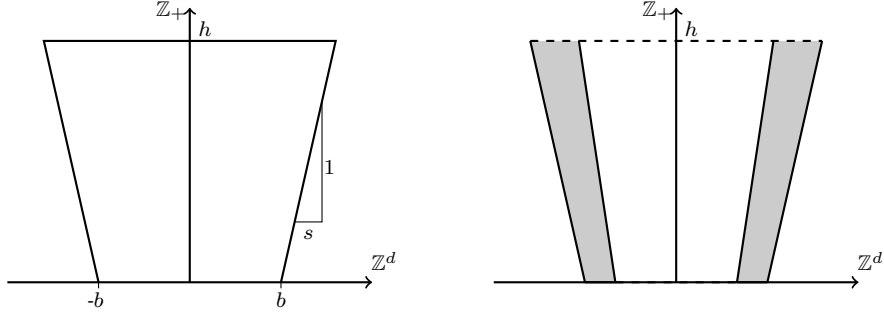

Let $\eta^{\cs} \coloneqq (\eta^\cs_{n})_{n=0,1,\dots}$ be the contact
process in $\cs(b_\inn , b_\out , s_\inn , s_\out , \infty)$ with
initial condition
$\eta^\cs_0(x) = \ind{b_\inn \le \norm{x}_2 \le b_\out}$ and
\begin{align*}
  \eta^\cs _{n+1}(x)
  & =
    \begin{cases}
      1 & \parbox[t]{28em}{if $(x,n+1)\in \mathsf{cs}(b_\inn,
        b_\out, s_\inn, s_\out, \infty)$, $\omega(x,n+1)=1$ and
        $\eta^\cs_n(y)=1$ for some $y \in \Z^d$ with
        $\norm{x-y} \le 1$,\smallskip} \\
      0 & \text{otherwise}.
    \end{cases}
\end{align*}
We think of $\eta^\cs $ as a version of the contact process where all
$\omega$'s outside the conical shell
$\cs (b_\inn , b_\out , s_\inn , s_\out , \infty)$ have been set to
$0$. We say that $\eta^\cs$ \emph{survives (in all parts of the
  conical shell)} if for all $n\in \Z_+$ there is $x\in \Z^d$ with
$\eta^\cs_n(x)=1$. In the case $d=1$ we require additionally that for
all $n \in \Z_+$ there is $x \in \Z_+$ and $y \in \Z_-$ with
$\eta^\cs_n(x) = \eta_n^\cs(y)=1$. (Here $\Z_-$ denotes the
non-positive integers.) For a directed path
\begin{align}
\label{eq:crossingpath}
\gamma = \big( (x_m,m), (x_{m+1},m+1),\dots, (x_{n},n)\big), \quad
  m \le n, \; x_i\in \Z^d \; \text{ with } \; \norm{x_{i-1}-x_i} \le
  1
\end{align}
we say that $\gamma$ \emph{crosses} the conical shell
$\cs(b_\inn , b_\out , s_\inn , s_\out , \infty)$ from the outside to
the inside if the following three conditions are fulfilled:
\begin{enumerate}[(i)]
\item the starting point lies outside the outer cone, i.e.,
  $\norm{x_m}_2 > b_\out + m s_\out$,
\item the terminal point lies inside the inner cone, i.e.,
  $\norm{x_{n}}_2< b_\inn + n s_\inn$,
\item all remaining points lie inside the shell, i.e.,
  $(x_{i},i) \in \cs(b_\inn, b_\out, s_\inn, s_\out, \infty)$ for
  $i=m+1,\dots,n-1$.
\end{enumerate}
We say that $\gamma$ \emph{intersects} $\eta^\cs $ if there exists
$i \in \{m+1,\dots,n-1\}$ with $\eta^\cs _{i}(x_i)=1$. Finally we say
that $\gamma$ is \emph{open in
  $\cs(b_\inn , b_\out , s_\inn , s_\out , \infty)$} if
$\omega(x_i,i)=1$ for all $i=m+1,\dots,n-1$.

Note that if in the case $d=1$ the process $\eta^\cs$ survives in a
ray of a conical shell and $\gamma$ is a path that is open in this ray
then by geometric properties of directed paths and clusters $\gamma$
necessarily intersects $\eta^\cs$.  In
the case $d>1$ however, even if $\eta^\cs$ survives, it is in
principle possible that an open path crosses the conical shell
$\cs(b_\inn , b_\out , s_\inn , s_\out , \infty)$ without intersecting
$\eta^\cs$. The next lemma states that the probability of that can be
made arbitrarily small.

\begin{lemma}
  Assume \label{lem:CPconeconv} that the relations in
  \eqref{eq:bs-inn-out} hold and consider the events
  \begin{align*}
    G_1 & \coloneqq \{ \eta^\cs  \; \text{survives} \}, \\
    G_2 & \coloneqq \{ \text{every open path $\gamma$ that crosses
          $\mathsf{cs}(b_\inn , b_\out ,
          s_\inn , s_\out , \infty)$ intersects $\eta^\cs $} \}.
  \end{align*}
  For any $\varepsilon >0$ and $0 \le  s_\inn < s_\out < 1$
  one can choose $p$ sufficiently close to
  $1$ and $b_\inn < b_\out$ sufficiently large so that
  \begin{align*}
    \Pr(G_1 \cap G_2) \ge 1-\varepsilon.
  \end{align*}
\end{lemma}

\begin{remark}[Observations concerning $G_1$ and $G_2$]
  The \label{rem:obsG1G2} meaning of the event $G_1$ is clear. Let us
  just note that it is essential that the relations in
  \eqref{eq:bs-inn-out} hold. In particular, in the case $s_\inn =
  s_\out$ survival of $\eta^\cs$ is only possible in the trivial case
  $p=1$.

  To understand the importance of the event $G_2$, observe that if a
  path $\gamma$ as in \eqref{eq:crossingpath} crosses and is open in
  $\cs(b_\inn, b_\out, s_\inn, s_\out, \infty)$, and also intersects
  $\eta^\cs$ then necessarily $\eta^\cs_n(x_n) =\omega(x_n,n)$ for the
  terminal point $(x_n,n)$ of the path. Thus, on $G_1 \cap G_2$ the
  values of the contact process inside the inner cone, that is for
  $(x,n)$ with $\norm{x}_2< b_\inn + n s_\inn$, are independent of
  what happens outside of the shell; cf.\ \eqref{def:goodconshells}
  and Lemma~\ref{lem:goodconshells}.
\end{remark}

\begin{proof}[Proof of Lemma~\ref{lem:CPconeconv}]
  The proof consists of two steps. In the first step we prove the
  assertion for the case $d=1$. The second step then uses the
  assertion for $d=1$ to give a proof for $d \ge 2$.

  Throughout the proof of this lemma for $r >0$ and $x \in \Z^d$ we
  denote by $B_r(x)$ the closed $\ell^2$ ball of radius $r$ around
  $x$, i.e., $B_r(x) = \{y \in \Z^d : \norm{x - y}_2\le r \}$.

  \medskip

  \noindent
  \textit{Step~1.} Consider the case $d=1$. We first check that the
  discrete time contact process survives with high probability in any
  \emph{oblique cone} when $p$ is large enough. To this end  for
  $0 < s_1 < s_2 < 1$ and $b \in \N$ we set
  \begin{align*}
    \mathsf{C}_{b,s_1,s_2} \coloneqq \big\{ (x,n) : x \in \Z,\; n \in
    \Z_+, \; s_1 n \le x \le s_2n + b \big\}.
  \end{align*}
  Furthermore we let $\bar\eta \coloneqq (\bar\eta_n)_{n =0,1,\dots}$
  be the discrete time contact process in $\mathsf{C}_{b,s_1,s_2}$
  starting from $\bar\eta_0 = \indset{[0,b] \cap \Z}$ and with all
  $\omega$'s outside $\mathsf{C}_{b,s_1,s_2}$ set to $0$. Finally, set
  $x_n=\lfloor b/2+n (s_1+s_2)/2 \rfloor$, $r_n=n(s_2-s_1)/4$.

  \begin{claim}
    (i) For every $0 < s_1 < s_2 < 1$ and $\varepsilon>0$ there is $b$
    large and $p_0 < 1$ such that for all $p\ge p_0$, $\bar\eta$
    survives with probability at least $1-\varepsilon$.

    (ii) Moreover, there exist $c, C \in (0,\infty)$ so that on the event
    $\{\bar \eta \text{ survives}\}$, with probability at least
    $1-Ce^{-cn}$, $\bar\eta_n$ restricted to the ball $B_{r_n}(x_n)$
    can be coupled with the unrestricted process
    $\eta^\nu=(\eta_n^\nu)_{n=0,1,\dots}$ started from the upper
    invariant measure $\nu$, that is $\bar \eta_n(x)=\eta_n^\nu (x)$
    for all $x\in B_{r_n}(x_n)$.
  \end{claim}

  The claim (i) follows using the same arguments as in the proof of
  Theorems~1 and 2 of \cite{CoxMaricSchinazi2010}, where an analogous
  statement is proved for continuous time contact process in a wedge.
  Moreover, \cite{CoxMaricSchinazi2010} use for the proofs a
  coarse-graining construction and comparison with oriented
  percolation. That links the problem in the wedge with a suitable
  shift of the contact process $\eta^{\{0\}}$ used in
  Lemma~\ref{lem:DG82coupl}. Combining that lemma with coarse-graining
  then yields the claim (ii).

  As noted in the paragraph above Lemma~\ref{lem:CPconeconv} we have
  $G_1 \subset G_2$ in the case $d=1$. Thus, by using the above
  argument for the oblique cone twice we see that the assertion of the
  lemma holds in the case $d=1$.

  \medskip

  \noindent
  \textit{Step~2}. Consider now $d>1$. Since the probability of $G_1$
  is increasing in dimension, it can be bounded from below by the same
  reasoning as in $d=1$. It remains to show that the probability of
  $G_2^c$ can be made small by choosing $b_\inn$, $b_\out$ and $p$
  appropriately.

  To this end for $n\ge 0$ we set
  $d(n)=(b_\inn+b_\out)/2 + n (s_\inn+s_\out)/2$ and define
  \begin{align*}
    \mathcal M_n \coloneqq \bigl\{x\in \mathbb Z^d: \norm{x}_2\in [d(n),d(n)+1]\bigr\}.
  \end{align*}
  Furthermore for $x\in \mathcal M_n$ we consider the event
  \begin{align*}
    \mathcal B_n(x) \coloneqq  \bigl\{\eta^\cs_n(x)=0 \; \text{ and }
    (x,n) \text{ contained in an open path } \gamma \text{ crossing
    the conical shell} \bigr\}.
  \end{align*}
  Finally, we define the backward cluster of $(x,n)$ by
  \begin{align*}
    \mathsf{bC}(x,n) & \coloneqq \big\{ (y,m) \in \Z^d \times \Z : m
                       \le n, (y,m) \to^\omega (x,n) \big\},
    \\ \intertext{and for $m \le n$  we set}
    \mathsf{bC}_m(x,n) & \coloneqq \{ y \in \Z^d : (y,m) \in \mathsf{bC}(x,n)\}.
  \end{align*}
  Assume that $\mathcal B_n(x)$ occurs. Then, by trivial geometrical
  arguments there is a small constant $\rho$ depending on the
  parameters of the shell so that for
  $M=\lfloor (1-\rho)n\rfloor$ the set $\mathsf{bC}_M(x,n)$ is
  not empty. Moreover, self-duality of the contact process and
  Lemma~\ref{lem:DG82coupl} imply that there exist
  $s_{\mathrm{cpl}}>0$ (here we think of
  $s_{\mathrm{cpl}}\approx s_{\mathrm{coupl}} \rho$ for
  $s_{\mathrm{coupl}}$ from Lemma~\ref{lem:DG82coupl}) and $c>0$ such
  that
  $B_{s_{\mathrm{cpl}}n}(x)\times \{M\} \subset
  \cs(b_\inn,b_\out,s_\inn,s_\out,\infty)$ and
  \begin{equation}
    \label{e:clustercoupling}
    \parbox{0.8\textwidth}{with probability bounded below by $1-e^{-c n}$, the
      indicator function of the set $\mathsf{bC}_M(x,n)$
      can  be coupled inside $B_{s_{\mathrm{cpl}} n}(x)$ with the set of
      $1$'s under the upper invariant measure $\nu$ of the
      (full) contact process.}
  \end{equation}
  Fix $p$ large enough so that the density of $1$'s under $\nu$ is
  strictly larger than $1/2$; this is not a restriction in the
  parameter region that we consider. On the one hand this requirement
  means heuristically that the set
  $\mathsf{bC}_M(x,n) \cap (B_{s_{\mathrm{cpl}}n}(x)\times\{M\})$ is
  large with high probability. Thus we must have $\eta^{\cs}_M(z)=0$
  for all $(z,M)$ in this set. On the other hand, using
  $d=1$-arguments we will show that this is not possible.

  To this end, depending on the previous parameters, we fix $\delta>0$
  small, and unit vectors $v^i\in \mathbb R^d$, $1\le i\le N$, with
  $N$ sufficiently large so that for every $x\in \mathcal M_n$, there
  is an $i\le N$ such that
  \begin{equation}
    \label{eq:longinter}
    \parbox{0.8\textwidth}
    {the length of the intersection of the half-line $\{t v^i:t\ge 0\}$ with the
      (real) ball $\{y\in \mathbb R^d:\norm{x-y}_2\le s_{\mathrm{coupl}}n\}$
      has length at least $\delta n$.}
  \end{equation}
  Observe that $N$ and $v^i$'s can be chosen independently of $n$. For
  $i\le N$,  let $\alpha^i = (\alpha^i(j))_{j\in\N}$ be a self-avoiding
  nearest neighbour path in $\Z^d$ approximating the half-line
  $\{tv^i:t\ge 0\}$ given by
  \begin{enumerate}[(a)]
  \item $\alpha^i_0 = 0$,
  \item $\alpha^i$ makes steps only in direction of $v^i$, that is for
    every coordinate $k=1,\dots, d$ and $j\ge 0$ one has
    $v^i_k (\alpha^i_k(j+1)-\alpha^i_k(j))>0$,
  \item $\alpha^i$ stays close to $tv^i$, that is
    $\{ \alpha^i_j : j \in \N \} \subset \{ t v^i + z : t \ge 0, z \in
    \R^d, \norm{z} \le 2 \}$.
  \end{enumerate}

  Using \eqref{e:clustercoupling}, \eqref{eq:longinter} and large
  deviation estimates for the density of $1$'s under $\nu^{(1)}$, see
  \cite[Thm.~1]{DS88} (literally, proved there for the continuous-time
    contact process) we see that there is $c>0$ such that that for
    $i\le N$ satisfying \eqref{eq:longinter},
  \begin{equation}
    \label{eq:DS88Thm1}
    \begin{split}
      \Pr\bigg(\frac{|\alpha^i\cap B_{s_{\mathrm{cpl}}n}(x) \cap \mathsf{bC}_M(x,n)|}
      {|\alpha^i\cap B_{s_{\mathrm{cpl}}n}(x) |} > 1/2
      \,\bigg| \,\mathcal B_n(x)\bigg) \ge 1- ce^{-cn}.
    \end{split}
  \end{equation}

  We now use the result of Step~1 and the last claim to bound the
  probability of $\mathcal B_n(x)$. To this end we define contact
  process $\eta^{(i)}$ as the contact process restricted to the set
  $\mathcal W^i \coloneqq (\alpha^i \times \Z_+)\cap
  \cs(b_\inn,b_\out,s_\inn,\infty)$
  started from $\indset{\mathcal W^i\cap (\mathbb Z^d \times \{0\})}$.
  Observe that $\mathcal W^i$ contains an isomorphic image of
  $\mathsf{C}_{b,s_1,s_2}$, for some $b, s_1,s_2$. Thus, the contact
  process $\eta^{(i)}$ dominates a corresponding image of a contact
  process in $\mathsf{C}_{b,s_1,s_2}$.

  Let $\mathcal S$ defined be the event
  $ \{\eta^{(i)} \text{ survives for every } i\le N \}$. Given
  $\varepsilon >0$, choose $\varepsilon'$ so that
  $(1-\varepsilon ')^N \ge 1- \varepsilon /2$. Using the result of
  Step~1 with $\varepsilon'$ replacing $\varepsilon$, we see that
  \begin{equation}
    \mathbb P(\mathcal S)\ge 1- \varepsilon /2.
  \end{equation}
  Moreover, on $\mathcal S$, for suitable $x_M^{(i)}$ with probability
  at least $1-e^{-cM}$, $\eta^{(i)}$ can be coupled in
  $B_{r_M}(x_M^{(i)})\times \{M\}$ to a stationary contact process.
  Moreover, the parameters can be chosen so that for every
  $x\in \mathcal M_n$, there is $x_M^{(i)}$ for which
  $B_{s_{\mathrm{cpl}}n}(x)\cap \alpha^i \subset B_{r_M}(x_M^{(i)})$.
  Hence, using again large deviation arguments for the
  density of $1$'s we obtain
  \begin{equation}
    \begin{split}
      \Pr\bigg( \frac
        {|\{ y \in \alpha^i\cap B_{s_{\mathrm{coupl}}n}(x)
            : \eta^{(i)}_M(y)=1\}|}
        {| \alpha^i\cap B_{s_{\mathrm{coupl}}n}(x)|}
        > 1/2
        \,\bigg|
        \,\mathcal S\bigg) \ge 1- ce^{-cn}.
    \end{split}
  \end{equation}
  Comparing \eqref{eq:DS88Thm1} with the last display, we see that  for
  every $x\in \mathcal M_n$,
  \begin{equation}
    \mathbb P\bigl(\mathcal B_{n}(x) | \mathcal S\bigr)\le e^{-cn}.
  \end{equation}
  Assume that $G_2^c$ occurs. Since there are at most polynomially
  many $x\in \mathcal M_n$, there must be
  $(x,n) \in \cup_{ \ell \ge 1} (\mathcal M_\ell \times\{\ell\})$ so
  that $\mathcal B_n(x)$ occurs. It follows that for $p$ sufficiently
  large we have
  \begin{equation}
    \mathbb P(G_2|\mathcal S)\ge 1-\varepsilon /2
  \end{equation}
  and therefore $\mathbb P(G_1\cap G_2)\ge 1-\varepsilon$, as required.
\end{proof}

\subsection{Regeneration construction}
\label{subs:regen-model1}

In Theorem~\ref{thm:LLNuCLTmodel1} we claim that the speed of the
random walk $X$ is $0$. As an intermediate result we will show that
the speed is bounded by a small constant. This will be needed for the
regeneration construction.

\begin{lemma}[A priori bound on the speed of the random walk]
  If the \label{lem:aprioribd} the family of kernels $\kappa$
  satisfies Assumptions~\ref{ass:local}--\ref{ass:finite-range} then
  there are positive finite constants $s_{\mathrm{max}}$, $c$ and $C$
  so that
  \begin{align}
    \label{eq:apr-bound}
    \Pr\big( \norm{X_n} > s_{\mathrm{max}} n \big) \leq C e^{-c n} ,
    \quad  n\in \N,
  \end{align}
  in particular
  $\limsup_{n\to\infty} \norm{X_n}/n \le s_{\mathrm{max}}$ almost
  surely. The bound $s_{\mathrm{max}}$ can be chosen arbitrarily small
  by taking $\varepsilon_\rf \ll 1$ (where $\varepsilon_\rf$ is from
  Assumption~\ref{ass:appr-sym}) and $1-p \ll 1$.
\end{lemma}
\begin{proof}
  With the percolation interpretation in mind, we say that a
  space-time site $(x,k)$ is \emph{wet} if $\eta_k(x)=1$, and
  \emph{dry} if $\eta_k(x)=0$. Let $\Gamma_n$ be the set of all
  $n$-step paths $\gamma$ on $\Z^d$ starting from $\gamma_0=0$ with
  the restriction $\norm{\gamma_{i}-\gamma_{i-1}} \le R_\kappa$,
  $i=1,\dots,n$, where $R_\kappa$ is the range of the kernels
  $\kappa_n$ from Assumption~\ref{ass:finite-range}. For $\gamma \in
  \Gamma_n$ and $0 \le i_1 < i_2 \dots < i_k \le n$ we define
  \begin{align*}
    D^\gamma_{i_1,\dots,i_k}
    & \coloneqq \{ \eta_{-\ell} (\gamma_{\ell}) =0 \text{ for all } \ell
      \in\{ i_1,\dots,i_k\}\},\\
    W^\gamma_{i_1,\dots,i_k}
    & \coloneqq \{ \eta_{-\ell} (\gamma_{\ell}) =1 \text{ for all } \ell \in
      \{1,\dots,n\} \setminus \{i_1,\dots,i_k\}\}.
  \end{align*}
  Let $H_n \coloneqq \#\{ 0 \le i \le n : \eta_{-i}(X_i)=0\}$ be the
  number of dry sites the walker visits up to time $n$ and set $K
  \coloneqq \max_{x \in \Z^d}\{\kappa_\rf(x)\}
  +\varepsilon_\rf$. For $k \in \{1,\dots,n\}$ by
  Lemma~\ref{lem:drysitesbd} we have
  \begin{align}
    \label{eq:Hnk0}
    \begin{split}
      \Pr(H_n = k)
      & = \sum_{0 \le i_1 < \dots < i_k \le n} \sum_{\gamma \in
        \Gamma_n} \Pr\bigl( (X_0,\dots,X_n) =\gamma,
      W^\gamma_{i_1,\dots,i_k},  D^\gamma_{i_1,\dots,i_k} \bigr) \\
      & \le \sum_{0 \le i_1 < \dots < i_k \le n} \sum_{\gamma \in
        \Gamma_n} K^{n-k} \Pr\bigl(D^\gamma_{i_1,\dots,i_k} \bigr) \\
      & \le \sum_{i_1 < i_2 < \dots < i_k \le n} R_\kappa^{d n}
      K^{n-k} \varepsilon(p)^k = \binom n  k R_\kappa^{d n} K^{n-k}
      \varepsilon(p)^k.
    \end{split}
  \end{align}
  It follows
  \begin{align}
    \label{eq:Hnk}
    \begin{split}
      \Pr(H_n \ge \delta n) & \le \sum_{k=\lfloor n \delta \rfloor}^n
      \binom{n}{k} R_\kappa^{d n} K^{n-k} \varepsilon(p)^k \\
      & \le \left( 2 R_\kappa^d K \right)^{n} \sum_{k=\lfloor n \delta
        \rfloor}^\infty \left(\varepsilon(p)/K\right)^k =
      \left(2 R_\kappa^d K \right)^{n}
      \frac{\left(\varepsilon(p)/K\right)^{\delta
          n}}{1-\varepsilon(p)/K} \le c_1 e^{-c_2 n}
    \end{split}
  \end{align}
  with $c_1, c_2 \in (0, \infty)$, when $\delta>0$ is sufficiently
  small and $p \ge p_0=p_0(\delta,\varepsilon_\rf)$.

 \smallskip

  Writing $X_n=(X_{n,1},\dots,X_{n,d})$ we can couple the first
  coordinate $(X_{n,1})_{n=0,1,\dots}$ of the random walk $X$ with
  a one-dimensional random walk
  $\wt{X} = (\wt{X}_n)_{n=0,1,\dots}$ with transition
  probabilities given by
  \begin{align*}
  \Pr\left(\wt{X}_n-\wt{X}_{n-1} = x \right) =
  (1-\varepsilon_\rf) \sum_{(x_2,\dots,x_d) \in \Z^{d-1}}
    \kappa_\rf\big(0, (x,x_2,\dots,x_d) \big)
  + \varepsilon_\rf \delta_{R_\kappa}(x), \quad x \in \Z
  \end{align*}
  (i.e., $\wt{X}$ takes with probability $1-\varepsilon_\rf$ a
  step according to the projection of $\kappa_\rf$ on the first
  coordinate and with probability $\varepsilon_\rf$ simply a step of
  size $R_\kappa$ to the right) such that for all $n\in\N$
  \begin{align*}
    X_{n,1} \le \wt{X}_{n-H_n} +R_\kappa H_n.
  \end{align*}
  Then, we have
  \begin{align}
    \label{eq:Xn1bars}
    \begin{split}
      \Pr\left(X_{n,1} > \bar{s} n \right) & \le \Pr\left(H_n \ge
        \delta n\right) + \sum_{k=0}^{\lfloor n
        \delta \rfloor} \Pr\left(X_{n,1} > \bar{s} n, H_n = k\right) \\
      & \le \Pr\left(H_n \ge \delta n\right) + \sum_{k=0}^{\lfloor n
        \delta \rfloor} \Pr\left( \wt{X}_{n-k} > \bar{s} n - k
        R_\kappa\right)\\
      & \le \Pr\left(H_n \ge \delta n\right) + \delta n
      \Pr\left(\wt{X}_n > (\bar{s}- \delta R_\kappa) n \right).
  \end{split}
  \end{align}
  The estimates \eqref{eq:Hnk}, \eqref{eq:Xn1bars} and standard large
  deviations bounds for $\wt X$ show that
  \begin{align}
    \Pr\left(X_{n,1} > \bar{s} n \right) \le c_3 e^{-c_4 n}
    \quad \text{holds for all} \; n \in \N
  \end{align}
  with $c_3, c_4 \in (0,\infty)$ when
  $\bar{s}- \delta R_\kappa
  > \E[\wt{X}_1] = \varepsilon_\rf R_\kappa$. By symmetry, we
  have an analogous bound for $\Pr\left(X_{n,1} < -\bar{s} n \right)$.
  The same reasoning applies to the coordinates
  $X_{n,2},\dots,X_{n,d}$. Thus, we have
  \begin{align}
    \Pr\left(\norm{X_n} > \bar{s} n \right) \le \sum_{i=1}^d
    \Pr\left(\abs{X_{n,i}} > \bar{s} n \right) \le 2d c_3 e^{-c_4 n}.
  \end{align}
  In particular we have
  $\limsup_{n\to\infty} \norm{X_n}/n \le \bar{s}$ a.s.\ by the
  Borel-Cantelli lemma.
\end{proof}

Denote the $R_\loc$-tube around the first $n$ steps of the path by
\begin{align}
  \label{eq:tuben}
  \mathsf{tube}_n \coloneqq \{ (y,-k) : 0 \le k \le n, \norm{y-X_k}
  \le R_\loc \}.
\end{align}
For $(x,n) \in \Z^d\times \Z$ let $\ell(x,n)$ be the length of the
longest (backwards in time) directed open path starting in $(x,n)$
with the convention $\ell(x,n)=-1$ if $\omega(x,n)=0$ and
$\ell(x,n)=\infty$ if $\eta_n(x)=1$. For each $(x,n)$ we define its
\emph{determining triangle} by
\begin{align}
  \label{eq:Dxn}
  D(x,n) \coloneqq \begin{cases} \emptyset, & \text{if} \; \eta_n(x) = 1, \\
    \{ (y,m) : \norm{y-x} \le (n-m), \, n-\ell(x,n)-1 \leq m \leq n \},
    & \text{if} \; \eta_n(x) = 0.
  \end{cases}
\end{align}
The idea is that if $\eta_n(x) = 0$, i.e.\ $(x,n)$ is not connected to
$\Z^d\times\{-\infty\}$, then this information can be deduced by
inspecting the $\omega$'s in $D(x,n)$. By definition of $\ell(x,n)$,
in this case there must be a \emph{closed contour} contained in
$D(x,n)$ which separates $(x,n)$ from $\Z^d\times\{-\infty\}$; see
Figure~\ref{fig:det-triangle}. Note in particular that
$D(x,n) = \{(x,n)\}$ if $\omega(x,n)=0$.

\begin{figure}
  \centering
  \begin{tikzpicture}
    \draw[thick] (0,0) -- (1.74,-3) -- (-1.74,-3) -- cycle;
    \draw (0.05,0.3) node {\footnotesize $(x,n)$};

    \draw[very thick] plot [smooth] coordinates {(-1.444,-2.5)
      (-1.3,-2.8) (-1,-2.99) (-0.5,-2.9) (0,-2.5) (0.3,-2.6) (0.6,-2.3)
      (0.9,-1.9) (1.154,-2)};

    \draw[black,opacity=0.5] plot [smooth] coordinates {(0,-0.01)
      (-0.15,-0.5) (-0.05,-1) (-0.3,-1.5) (-0.8,-2) (-0.6,-2.5) (-1,-2.99) };

    \draw[black,opacity=0.5] plot [smooth] coordinates {(-0.6,-2.5)
      (-0.5,-2.7)  (-0.5,-2.9)};

    \draw[black,opacity=0.5] plot [smooth] coordinates {(-0.15,-0.5)
      (0.05,-1) (-0.05,-1.5) (0.35,-2) (0.3,-2.6) };

     \draw[black,opacity=0.5] plot [smooth] coordinates {(0.05,-1)
       (0.3,-1.5) (0.4,-2) (0.6,-2.3)};
  \end{tikzpicture}
  \caption{A caricature of the determining triangle $D(x,n)$ with a
    closed contour. The height of the triangle is $\ell(x,n)+1$.}
  \label{fig:det-triangle}
\end{figure}
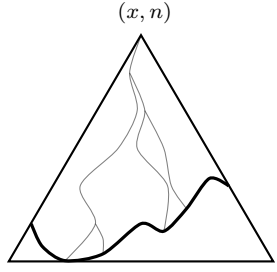

\smallskip

When constructing the walk $X$ for $n$ steps we must inspect $\omega$
and $\eta$ in $\mathsf{tube}_n$ (cf.\ Remark~\ref{rem:Xlocalconstr}).
By the nature of $\eta$, this in principle yields information on the
configurations $\eta_{-k}$, $k > n$ that the walk will find in its
future. \emph{Positive information} of the form $\eta_m(y)=1$ for
certain $m$ and $y$ is at this stage harmless because $\eta$ has
positive correlations and in view of Assumption~\ref{ass:appr-sym}
this suggests a well-behaved path in the future. On the other hand,
\emph{negative information} of the form $\eta_m(y)=0$ for certain $m$
and $y$ is problematic because this increases the chances to find more
$0$'s of $\eta$ in the walk's future. In this case
Assumption~\ref{ass:appr-sym} is useless. In order to `localise'
this negative information we `decorate' the tube around the path
with the determining triangles for all sites in $\mathsf{tube}_n$
(obviously, only zeros of $\eta$ matter)
\begin{equation}
  \label{eq:dtuben}
  \mathsf{dtube}_n = \bigcup_{(y,k) \in \mathsf{tube}_n} D(y,k) .
\end{equation}
Define
\begin{align}
  \label{eq:Dn}
  D_n \coloneqq n + \max\big\{ \ell(y,-n)+2 : \norm{X_n-y} \leq R_\loc,\,
  \ell(y,-n) < \infty \big\}.
\end{align}
Note that $D_n$ is precisely the time (for the walk) at which the
reasons for $\eta_{-n}(y)=0$ for all $y$ from the
$R_\loc$-neighbourhood of $X_n$ are explored by inspecting all the
determining triangles with base points in
$B_{R_\loc}(X_n) \times \{-n\}$. The information $\eta_{-n}(y)=0$ does
not affect the law of the random walk after time $D_n$. Note that the
`height' of a non-empty triangle $D(y,-n)$ is $\ell(y,-n)+1$. This
is why $\ell(y,-n)+2$ appears in definition~\eqref{eq:Dn}.

\smallskip

Now, between time $n$ and $D_n$ the random walk might have explored
more negative information which in general will be decided after time
$D_n$ and will affect the law of the random walk thereafter. To deal
with this \emph{cumulative negative future information} we define
recursively a sequence
\begin{align}
  \label{eq:sigmas}
  \sigma_0 \coloneqq 0, \quad \sigma_{i} \coloneqq
  \min\Big\{m>\sigma_{i-1} : \max_{\sigma_{i-1} \leq n \leq m} D_n
  \leq m \Big\}, \; i \ge 1.
\end{align}
In words, $\sigma_i$ is the first time $m$ after $\sigma_{i-1}$ when
the reasons for all the negative information that the random walk
explores in the time interval $\sigma_{i-1},\dots,m$ are decided
`locally' and thus the law of the random walk after time $\sigma_i$
does not depend on that negative information. The $\sigma_i$ are
stopping times with respect to the filtration
$\mathcal{F}=(\mathcal{F}_n)_{n =0,1,2,\dots}$, where
\begin{align}
  \label{eq:defFn}
  \mathcal{F}_n \coloneqq \sigma\big( X_j : 0 \leq j \leq n \big)
  \vee \sigma\big(\eta_{j}(y), \omega(y,j)
  : (y,j) \in \mathsf{tube}_n \big)
  \vee \sigma\big( \omega(y,j)
  : (y,j) \in \mathsf{dtube}_n \big) .
\end{align}
Note that by construction we have $\eta_{-\sigma_i}(y)=1$ for all
$y \in B_{R_\loc}(X_{\sigma_i})$.
\begin{lemma}
  \label{lem:regen1}
  When $p$ is sufficiently close to $1$ there exist finite positive
  constants $c$ and $C$ so that
  \begin{align}
    \label{eq:sigma-inc}
    \Pr\big( \sigma_{i+1}-\sigma_i > n \, \big| \,
    \mathcal{F}_{\sigma_i} \big) \leq C e^{-c n} \quad \text{for all
    $n =1,2,\dots$, $i =0,1,\dots$ a.s.},
  \end{align}
  in particular, all $\sigma_i$ are a.s.\ finite. Furthermore, we have
  \begin{align}
    \label{eq:omegatauidomination}
    \mathscr{L}\big( (\omega(\cdot, -j-\sigma_i)_{j =0,1,\dots} \, \big| \,
    \mathcal{F}_{\sigma_i}
    \big) \succcurlyeq \mathscr{L}\big( (\omega(\cdot, -j)_{j
    =0,1,\dots} \big) \quad \text{for all  $i =0,1,\dots$  a.s.},
  \end{align}
  where `$\succcurlyeq$' denotes stochastic domination.
\end{lemma}
\begin{proof}
  Throughout the proof we write
  $\widehat {R}_\kappa \coloneqq (2 R_\kappa+1)^d$ and
  $\widehat {R}_\loc \coloneqq (2 R_\loc+1)^d$ for the number of
  elements in $B_{R_\kappa}(0)$ respectively in
  $B_{R_\loc}(0)$.

  Consider first the case $i=0$ in \eqref{eq:sigma-inc}. The event
  $\{\sigma_1>n\}$ enforces that in the $R_\loc$-vicinity of the path
  there are space-time points $(y_j,-j)$ with $\eta_{-j}(y_j)=0$ for
  $j=0,1,\dots,n$. For a fixed choice of the $y_j$'s by
  Lemma~\ref{lem:drysitesbd} the probability of that event is bounded
  by $\varepsilon(p)^n$. We use a relatively crude estimate to bound
  the number of relevant vectors
  $(y_0,y_1,\dots,y_n) \in (\Z^d)^{n+1}$, as follows. There are
  $\widehat R_\kappa^{n}$ possible $n$-step paths for the walk. Assume
  there are exactly $k$ time points along the path, say
  $0 \le m_1 <\cdots<m_k\le n$, when a point
  $(y_{m_i},-m_i) \in B_{R_\loc}(X_{m_i}) \times \{-m_i\}$ with
  $\eta_{-i}(y_{m_i})=0$ is encountered and hence the corresponding
  `determining' triangle $D(y_{m_i},-m_i)$ is not empty (when $n>1$,
  we necessarily have $m_1=0$ or $m_1=1$).

  For consistency of notation we write $m_{k+1} =n$. Then the height
  of $D(y_{m_i},-m_i)$ is bounded below by $ m_{i+1}-m_i$. For a fixed
  $n$-step path of $X$ and fixed $m_1 <\cdots<m_k$, there are at most
  $\widehat R_\loc^k$ many choices for the $y_{m_i}$, $i=1,\dots,k$,
  and inside $D(y_{m_i},-m_i)$ we have at most
  $\widehat R_\kappa^{m_{i+1}-m_i-1}$ choices to pick
  $y_{m_i+1},y_{m_i+2},\dots, y_{m_{i+1}-1}$ (start with $y_{m_i}$,
  then follow a longest open path which is not connected to
  $\Z^d \times \{-\infty\}$, these sites are necessarily zeros of
  $\eta$). Thus, there are at most
  \begin{align*}
    \widehat R_\kappa^n \sum_{k=1}^n \sum_{m_1<\cdots<m_k \le m_{k+1}
      = n} \widehat R_\loc^k
    \prod_{i=1}^k \widehat R_\kappa^{m_{i+1}-m_i-1}
    = \widehat R_\kappa^n \sum_{k=1}^n \binom{n}{k} \widehat R_\loc^k
    \widehat R_\kappa^{n-k} \le \widehat R_\kappa^n \big( \widehat R_\loc  +
    \widehat R_\kappa \big)^n
  \end{align*}
  possible choices of $(y_0,y_1,\dots,y_n)$ and hence we have
  \begin{align*}
    \Pr(\sigma_1>n) \le \bigl(\widehat  R_\kappa (\widehat R_\loc  +
    \widehat R_\kappa ) \varepsilon(p)\bigr)^n.
  \end{align*}
  The right hand side decays exponentially when $p$ is close to $1$ so
  that $\varepsilon(p)$ is small enough. For general $i>0$
  \eqref{eq:sigma-inc} follows by induction, employing
  \eqref{eq:omegatauidomination} and the argument for $i=0$.

  \smallskip

  In order to verify \eqref{eq:omegatauidomination} note that the
  stopping times $\sigma_i$ are special in the sense that on the one
  hand, at a time $\sigma_i$ the `negative information' in
  $\mathcal{F}_{\sigma_i}$, that is the knowledge of some zeros of
  $\eta$ in the $R_\loc$-neighbourhood of the path, has been
  `erased' because the reasons for that are decided by local
  information contained in $\mathcal{F}_{\sigma_i}$. On the other
  hand, the `positive information', that is the knowledge of ones of
  $\eta$, enforces the existence of certain open paths for the
  $\omega$'s. And this information is possibly retained. Thus,
  \eqref{eq:omegatauidomination} follows from the FKG inequality for
  the $\omega$'s.
\end{proof}

\begin{corollary}[Reformulation of Lemma~\ref{lem:drysitesbd}]
  For \label{cor:drysitesbd} any
  $V = \big\{ (x_1,t_1),\dots,(x_k,t_k) \big\} \subset \Z^d \times \N$
  with $t_1 < t_2 < \cdots < t_k$ and $\varepsilon(p)$ as in
  Lemma~\ref{lem:drysitesbd} we have
  \begin{equation}
    \label{eq:drysitesbd.cond}
    \Pr\left( \eta_{-t-\sigma_i}(x+X_{\sigma_i})=0 \; \text{for all}\; (x,t) \in V
      \, \big| \, \mathcal{F}_{\sigma_i} \right) \leq
    \varepsilon(p)^k.
  \end{equation}
\end{corollary}
\begin{proof}
  The assertion is an easy consequence of \eqref{eq:omegatauidomination} and
  Lemma~\ref{lem:drysitesbd}.
\end{proof}

For $t \in \N$ we define
$R_t \coloneqq \inf\{ i \in \Z_+ : \sigma_i \ge t\}$ and for
$m=1,2,\dots$ we put
\begin{align}
  \label{eq:tautildes}
  \wt{\tau}^{(t)}_m \coloneqq \begin{cases}
    \sigma_{R_t-m+1} - \sigma_{R_t-m}, & m \le R_t, \\
    0, & \text{else}.
  \end{cases}
\end{align}
In words, $\wt{\tau}^{(t)}_1$ is the length of the time interval
$(\sigma_{i-1},\sigma_i]$ which contains $t$ and $\wt{\tau}^{(t)}_m$
is the length of the ($m-1$)-th interval before it.

\begin{lemma}
  \label{lem:tautildetails}
  When $p$ is sufficiently close to $1$ there exist finite positive
  constants $c$ and $C$ so that for all $i, n =0,1,\dots$
  \begin{align}
    \label{eq:tautilde1tail}
    \Pr\big( \wt{\tau}^{(t)}_1 \ge n \, \big| \, \mathcal{F}_{\sigma_i} \big)
    & \leq C e^{-c n} \quad \text{a.s.\ on $\{\sigma_i < t \}$}, \\
    \intertext{and generally}
    \label{eq:tautildemtail}
    \Pr\big( R_t \ge i+m, \wt{\tau}^{(t)}_m \ge n \, \big| \, \mathcal{F}_{\sigma_i} \big)
    & \leq C m^2 e^{-c n} \quad \text{for $m =1,2,\dots$ a.s.\ on $\{\sigma_i < t \}$}.
  \end{align}
\end{lemma}
\begin{proof}
  For \eqref{eq:tautilde1tail}, we have
  \begin{align*}
    & \Pr\big( \wt{\tau}^{(t)}_1 \ge n \, \big| \, \mathcal{F}_{\sigma_i} \big) \\
    & = \Pr\big( \sigma_{i+1} \ge t \vee (n + \sigma_i) \, \big| \,
      \mathcal{F}_{\sigma_i} \big)
           + \sum_{j>i} \sum_{\ell=\sigma_i + 1}^{t-1}
           \Pr\big( \sigma_j=\ell, \sigma_{j+1} \ge t \vee (\ell + n) \, \big| \,
           \mathcal{F}_{\sigma_i} \big) \\
    & \le C e^{-c n} + \sum_{\ell=\sigma_i + 1}^{t-1} C e^{-c((t-\ell) \vee n)}
      \Pr\big( \exists\, j > i \, : \, \sigma_j = \ell \, \big| \,
      \mathcal{F}_{\sigma_i} \big) \\
    & \le C e^{-c n} + \ind{\sigma_i \le t-n-2} \sum_{\ell=\sigma_i + 1}^{t-n-1} C e^{-c(t-\ell)}
      + \ind{n+1 \le t} \sum_{\ell=t-n}^{t-1} C e^{-c n}
      \le C \big( 1 + \tfrac{e^{-c}}{1-e^{-c}} + n \big) e^{-c n}
  \end{align*}
  where we used Lemma~\ref{lem:regen1} and
  \begin{align*}
    \Pr\big( \sigma_j=\ell, \sigma_{j+1} \ge t \vee (\ell + n) \, \big| \,
    \mathcal{F}_{\sigma_i} \big)
    = \E\big[ \ind{\sigma_j=\ell} \Pr( \sigma_{j+1}- \sigma_j \ge (t-\ell) \vee n \,|\,
    \mathcal{F}_{\sigma_j}) \big| \, \mathcal{F}_{\sigma_i} \big]
  \end{align*}
  in the first inequality.

  Similarly, for $m \ge 2$ (we assume implicitly that
  $\sigma_i \le t-n-m-1$ for otherwise the conditional probability
  appearing on the right-hand side of \eqref{eq:tautildemtail} equals
  $0$)
  \begin{align*}
    \Pr\big( \wt{\tau}^{(t)}_m \ge \; n \, \big| \, \mathcal{F}_{\sigma_i} \big)
    & = \sum_{j>i} \sum_{k=\sigma_i+1}^{t-m-n} \sum_{\ell=k+n}^{t-m+1}
      \Pr\big( \sigma_j=k, \sigma_{j+1}=\ell, \sigma_{j+m-1}<t, \sigma_{j+m} \ge t
      \, \big| \, \mathcal{F}_{\sigma_i} \big) \\
    & \le \sum_{j>i} \sum_{k=\sigma_i+1}^{t-m-n} \sum_{\ell=k+n}^{t-m+1}
      \Pr\big( \sigma_j=k, \sigma_{j+1}=\ell \, \big| \, \mathcal{F}_{\sigma_i} \big)
      \times (m-1) C e^{-c(t-\ell)/(m-1)} \\
    & \le C (m-1) \sum_{k=\sigma_i+1}^{t-m-n} \sum_{\ell=k+n}^{t-m+1} e^{-c(t-\ell)/(m-1)}
      \sum_{j>i} \Pr\big( \sigma_j=k \, \big| \,
      \mathcal{F}_{\sigma_i} \big) \times C e^{-c(\ell-k)} \\
    & \le C^2 (m-1) \sum_{k=\sigma_i+1}^{t-m-n} e^{ck - ct/(m-1)}
      \sum_{\ell=k+n}^{t-m+1} \exp\big(-c \tfrac{m-2}{m-1} \ell \big)
  \end{align*}
  where we used in the first inequality that
  \begin{align*}
    \big\{ \sigma_{j+1}=\ell, \sigma_{j+m-1}<t, \sigma_{j+m} \ge t
    \big\} \subset \bigcup_{r=j+2}^{j+m} \big\{ \sigma_r - \sigma_{r-1}
    \ge \tfrac{t-\ell}{m-1} \big\}
  \end{align*}
  together with Lemma~\ref{lem:regen1} and then argued analogously to
  the proof of \eqref{eq:tautilde1tail} for the second inequality. For
  $m=2$, the chain of inequalities above yields the bound
  \begin{align*}
    \Pr\big( \wt{\tau}^{(t)}_m \ge n \, \big| \, \mathcal{F}_{\sigma_i} \big)
    \le C^2 \sum_{k=\sigma_i+1}^{t-n} (t-k-n) e^{-c(t-k)}
    \le C^2 e^{-c n} \sum_{\ell=0}^\infty \ell e^{-c \ell}
  \end{align*}
  whereas for $m>2$ we obtain
  \begin{align*}
     \Pr\big( \wt{\tau}^{(t)}_m \ge n \, \big| \, \mathcal{F}_{\sigma_i} \big)
    & \le C^2 (m-1) \sum_{k=\sigma_i+1}^{t-m-n} e^{ck - ct/(m-1)}
      \frac{\exp\big( - c \tfrac{m-2}{m-1} (k+n) \big)}{1-e^{c \frac{m-2}{m-1}}}  \\
    & = \frac{C^2 (m-1)}{1-e^{c \frac{m-2}{m-1}}}
      \exp\big( - c \tfrac{m-2}{m-1} n - c \tfrac{t}{m-1} \big)
      \sum_{k=\sigma_i+1}^{t-m-n} e^{ck/(m-1)} \\
    & \le \frac{C^2 (m-1)}{1-e^{c \frac{m-2}{m-1}}}
      \exp\big( - c \tfrac{m-2}{m-1} n - c \tfrac{t}{m-1} \big)
      \frac{e^{c(t-n)/(m-1)}}{e^{c/(m-1)}-1}
    \\&= \frac{C^2 (m-1)}{1-e^{c \frac{m-2}{m-1}}} \frac{e^{-c n}}{e^{c/(m-1)}-1}
        \le C' (m-1)^2 e^{-c n}.
  \end{align*}
  Thus, \eqref{eq:tautildemtail} holds (with suitable adaptation of
  the value of the prefactor).
\end{proof}

As a result of \eqref{eq:omegatauidomination} and
Assumption~\ref{ass:appr-sym}, the walk is well-behaved at least along
the sequence of stopping times $\sigma_i$, we formalise this in the
following result.
\begin{lemma}
  \label{lem:xincrbd1} When $p$ is sufficiently close to $1$ there
  exist finite positive constants $c$ and $C$ so that for all finite
  $\mathcal{F}$-stopping times $T$ with $T \in \{\sigma_i : i \in
  \N\}$ a.s.\
  and all $k \in \N$
  \begin{align}
    \label{eq:xincrbd1a}
    \Pr\big( \norm{X_k-X_T} > s_{\mathrm{max}}(k-T) \, \big| \,
    \mathcal{F}_T \big)
    & \leq C e^{-c(k-T)} \quad \text{a.s.\ on $\{T<k\}$} \\
    \intertext{and for $j<k$}
    \label{eq:xincrbd1b}
    \Pr\big( \norm{X_k-X_j} > (1+\epsilon) s_{\mathrm{max}}(k-j) \,
    \big| \, \mathcal{F}_T \big)
    & \leq C e^{-c(k-j)} \quad \text{a.s.\ on $\{T \le j\}$}
  \end{align}
  with $s_{\mathrm{max}}$ as in Lemma~\ref{lem:aprioribd}.
\end{lemma}
\begin{proof} 
  Note that by Lemma~\ref{lem:equalstoppingtimes}, we may assume that
  $T=\sigma_\ell$ for some $\ell\in\Z_+$, the general case follows by
  writing $1=\sum_{\ell=0}^\infty \ind{T=\sigma_\ell}$ (a.s.).

  For \eqref{eq:xincrbd1a}, we combine \eqref{eq:omegatauidomination}
  and in particular Corollary~~\ref{cor:drysitesbd} with the proof of
  Lemma~\ref{lem:aprioribd}. For \eqref{eq:xincrbd1b}, let
  $T'\coloneqq\inf \big(\{\sigma_i : i \in \N\} \cap [j,\infty) \big)$
  be the time of the next $\sigma_i$ after time $j$.
  Inequality~\eqref{eq:tautilde1tail} from
  Lemma~\ref{lem:tautildetails} shows that
  $\Pr\big( T'-j > \epsilon(k-j) \, \big| \, \mathcal{F}_T)$ is
  exponentially small in $k-j$. On $\{ T'-j \le \epsilon(k-j)\}$ we
  use \eqref{eq:xincrbd1a} starting from time $T'$ and simply use the
  fact that increments are bounded for the initial piece between time
  $j$ and time $T'$.
\end{proof}

For $m <n$ we say that $n$ is a \emph{$(b,s)$-cone time point for the
  decorated path beyond $m$} if
\begin{align}
  \label{eq:decconetime}
  &\big( \mathsf{tube}_n \cup \mathsf{dtube}_n \big) \cap
    \big( \Z^d \times \{-n,-n+1,\dots,-m\} \big) \notag \\
  & \hspace{6em}
    \subset \big\{ (x,-j) : m \le j \le n, \norm{x-X_n} \le b + s(n-j)
    \big\}.
\end{align}
In words (see also Figure~\ref{fig:cone-time}), $n$ is a cone time
point for the decorated path beyond $m$ if the space-time path
$(X_j,-j)_{j=m,\dots,n}$ together with its $R_\loc$-tube and
decorations by determining triangles is contained in $\cone(b,s,n-m)$
shifted to the base point $(X_n,-n)$; recall the definition of
$\cone(b,s,h)$ in \eqref{def:cone}. Note that \eqref{eq:decconetime}
in particular implies
\begin{align}
  \label{eq:conetime}
  \norm{X_n-X_j} \leq b + s(n-j) \quad \text{for}\; j=m,\dots,n-1.
\end{align}

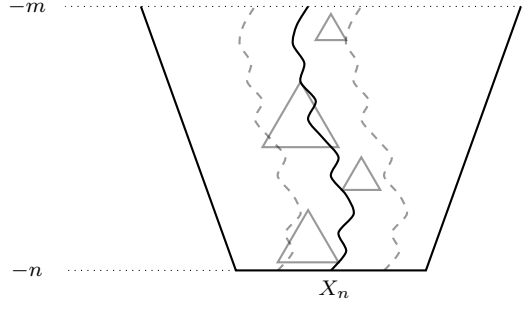
\begin{figure}
  \centering
  \begin{tikzpicture}
    \draw[thick, black] plot [smooth] coordinates {(0,0)
      (0.2,0.25) (0.1,0.5) (0.3,0.75) (0.2,1) (0,1.25) (0.1,1.5)
      (-0.1,1.75) (-0.3,2) (-0.2,2.25) (-0.4,2.5) (-0.35,2.75)
      (-0.5,3) (-0.45,3.25) (-0.3,3.5)};

    \draw (0.05,-0.25) node {\footnotesize $X_n$};

    \draw[xshift=7mm,thick, dashed, black,opacity=0.4] plot [smooth]
    coordinates {(0,0) (0.2,0.25) (0.1,0.5) (0.3,0.75) (0.2,1)
      (0,1.25) (0.1,1.5) (-0.1,1.75) (-0.3,2) (-0.2,2.25) (-0.4,2.5)
      (-0.35,2.75) (-0.5,3) (-0.45,3.25) (-0.3,3.5)};

    \draw[xshift=-7mm,thick, dashed, black,opacity=0.4] plot [smooth]
    coordinates {(0,0) (0.2,0.25) (0.1,0.5) (0.3,0.75) (0.2,1)
      (0,1.25) (0.1,1.5) (-0.1,1.75) (-0.3,2) (-0.2,2.25) (-0.4,2.5)
      (-0.35,2.75) (-0.5,3) (-0.45,3.25) (-0.3,3.5)};

    \draw[thick] (-2.5,3.5) -- (-1.25,0) -- (1.25,0) -- (2.5,3.5) ;

    \draw[dotted] (2.5,3.5) -- (-3.5,3.5) (-1.25,0) -- (-3.5,0) ;
    \draw (-4,3.5) node {\footnotesize $-m$};
    \draw (-4,0) node {\footnotesize $-n$};


    \draw[xshift=-3mm, yshift=0.8cm, scale=0.8,
    thick,black,opacity=0.4] (0,0) -- (0.5,-0.866) -- (-0.5,-0.866) -- cycle;

    \draw[xshift=4mm, yshift=1.5cm, scale=0.5,
    thick,black,opacity=0.4] (0,0) -- (0.5,-0.866) -- (-0.5,-0.866) -- cycle;

    \draw[xshift=-4mm, yshift=2.5cm, scale=1, thick,black,opacity=0.4]
    (0,0) -- (0.5,-0.866) -- (-0.5,-0.866) -- cycle;

    \draw[xshift=0mm, yshift=3.4cm, scale=0.4, thick,black,opacity=0.4]
    (0,0) -- (0.5,-0.866) -- (-0.5,-0.866) -- cycle;
  \end{tikzpicture}
  \caption{Cone time point $n$ for the decorated path beyond $m$.
    Dashed lines indicate the $R_\loc$ tube around the path and the
    triangles are the determining triangles. The cone is given by
    $\cone(b,s,n-m)$ shifted to the base point $(X_n,-n)$.}
  \label{fig:cone-time}
\end{figure}

\begin{lemma}
  \label{lem:xconetimes}
  For $\varepsilon > 0$, when $p$ is sufficiently close to $1$,
  there exist $b>0$ and $s>s_{\mathrm{max}}$ such that for all finite
  $\mathcal{F}$-stopping times $T$ with
  $T \in \{\sigma_i : i \in \N\}$ a.s.\ (i.e., $T=\sigma_J$ for a
  suitable random index $J$) and all $k \in \N$, with
  $T'\coloneqq\inf\{\sigma_i : \sigma_i \ge k\}$
  \begin{align}
    \label{eq:xconetimeprob}
    \Pr\big( \text{$T'$ is a $(b,s)$-cone time point for the decorated path beyond $T$}
    \, \big| \, \mathcal{F}_T \big) \geq 1-\varepsilon
  \end{align}
  a.s.\ on $\{T<k\}$. Furthermore $0<s-s_{\mathrm{max}} \ll 1$ can be
  chosen small.
\end{lemma}
\begin{proof}
  Denote the event in \eqref{eq:conetime} (with $m=T$, $n=k$ and $b$
  replaced by $b - M$, where $M>0$ will be tuned later) by $B_{T,k}$.
  We have
  \begin{multline*}
    1 - \Pr\big( B_{T,k} \, \big| \, \mathcal{F}_T \big) \le
    \sum_{j=T}^{k-1} \Pr\big( \norm{X_k-X_j} >
    b - M + s (k-j) \, \big| \, \mathcal{F}_T \big) \\
    \le \sum_{j=k-m}^{k-1} \Pr\big( \norm{X_k-X_j} > b - M \, \big| \,
    \mathcal{F}_T \big) + \sum_{j=T}^{k-m-1} \Pr\big( \norm{X_k-X_j} >
    s (k-j) \, \big| \, \mathcal{F}_T \big).
  \end{multline*}
  Using \eqref{eq:xincrbd1b} from Lemma~\ref{lem:xincrbd1} we can make
  the second sum small by choosing $m$ sufficiently large and
  $s>s_{\mathrm{max}}$. Then we can make the first sum small (or even
  vanish) by picking $b-M$ sufficiently large.

  Recall that $R_\kappa$ is the range of the random walk $X$.
  Inequality~\eqref{eq:tautilde1tail} from
  Lemma~\ref{lem:tautildetails} implies that
  $\Pr\big( T'- k \ge (M-R_\loc)/R_\kappa \, \big| \, \mathcal{F}_T
  \big)$
  can be made arbitrarily small by choosing $M$ sufficiently large. On
  $B_{T,k} \cap \{ T'- k < (M-R_\loc)/R_\kappa \}$, which has high
  probability under $\Pr(\cdot \, | \, \mathcal{F}_T)$, we have by
  construction that
  \begin{align*}
    & \mathsf{tube}_{T'} \cap
      \big( \Z^d \times \{-T',-T'+1,\dots,-T\} \big) \\
    & \hspace{6em}
      \subset \big\{ (x,-j) : T \le j \le T', \norm{x-X_{T'}} \le b + s(T'-j) \big\},
  \end{align*}
  i.e., the path together with its $R_\loc$-tube is covered by a suitably
  shifted cone with base point $(X_{T'},-T')$.

  \smallskip

  It remains to verify that under $\Pr(\cdot \, | \, \mathcal{F}_T)$
  with high probability also the decorations (recall \eqref{eq:Dxn},
  \eqref{eq:dtuben}) are covered by the same cone. To show this we may
  assume $T=\sigma_i$ for notational simplicity; this is justified by
  Lemma~\ref{lem:equalstoppingtimes}. Let $R_k$,
  $\wt{\tau}^{(k)}_1, \wt{\tau}^{(k)}_2, \dots$ be as
  defined 
  in and around \eqref{eq:tautildes} with $t=k$.

  Note that $\mathsf{dtube}_{T'}$ is contained in a union of
  space-time rectangles with heights $\wt{\tau}^{(k)}_m$, side lengths
  $2\wt{\tau}^{(k)}_m (R_\loc \vee R_\kappa)$ and base points
  $\big(X_{T'-\wt{\tau}^{(k)}_1-\cdots-\wt{\tau}^{(k)}_m},
  -(T'-\wt{\tau}^{(k)}_1-\cdots-\wt{\tau}^{(k)}_m)\big)$.
  A geometric argument shows that on the event
  \begin{align*}
  B_{T,k} \cap \bigcap_{m=1}^\infty \Big( \big\{ R_k \ge i+m,
  \wt{\tau}^{(t)}_m < M_1 + \epsilon_1 m \big\} \cup \big\{ R_k
  < i+m \big\} \Big),
  \end{align*}
  for $\epsilon_1$, $M_1$ chosen suitably in relation to $b$ and $s$,
  we also have
  \begin{align*}
    & \mathsf{dtube}_{T'} \cap
      \big( \Z^d \times \{-T',-T'+1,\dots,-T\} \big)  \\
    & \hspace{6em}
      \subset \big\{ (x,-j) : T \le j \le T', \norm{x-X_{T'}} \le b +
      s(T'-j) \big\}.
  \end{align*}
  Using inequality~\eqref{eq:tautildemtail} from
  Lemma~\ref{lem:tautildetails}, we see that
  \begin{align*}
    \sum_{m=1}^\infty \Pr\big( R_k \ge i+m, \wt{\tau}^{(k)}_m
    \ge M_1 + \epsilon_1 m \, \big| \,
    \mathcal{F}_{T} \big)
    \le \sum_{m=1}^\infty C m^2 e^{-c(M_1 + \epsilon_1 m)} \ \text{ a.s.\ on $\{T<k\}$}
  \end{align*}
  which can be made arbitrarily small when $M_1$ and $\epsilon_1$ are
  suitably tuned. This completes the proof of \eqref{eq:xconetimeprob}.
\end{proof}


Note that the $\sigma_i$ defined in \eqref{eq:sigmas} are themselves
not regeneration times since \eqref{eq:omegatauidomination} is in
general not an equality of laws. We use another layer in the
construction with suitably nested cones to forget remaining positive
information.

\smallskip

Recall the definition of cones and cone shells from \eqref{def:cone},
\eqref{eq:coneshell} and Figure~\ref{fig:cone-cs}. The following sets
of `good' $\omega$-configurations in conical shells will play a key
role in the regeneration construction. Let
$G(b_\inn , b_\out , s_\inn , s_\out ,h) \subset
\{0,1\}^{\mathsf{cs}(b_\inn , b_\out , s_\inn , s_\out ,h)}$
be the set of all $\omega$-configurations
with the property
\begin{align}
  \label{def:goodconshells}
  \begin{split}
  & \forall \,
    \eta_0, \eta_0' \in \{0,1\}^{\Z^d} \; \text{with} \;
    \eta_0|_{\displaystyle B_{b_\out }(0)} = \eta_0'|_{\displaystyle
    B_{b_\out }(0)}
    \equiv 1 \quad \text{and} \\
  & \hspace{1em} \omega \in \{0,1\}^{\Z^d \times \{1,\dots,h\}} \; \text{with} \;
    \omega|_{\mathsf{cs}(b_\inn , b_\out , s_\inn , s_\out ,h)}
    \in G(b_\inn , b_\out , s_\inn , s_\out ,h) \, : \\[1ex]
  & \hspace{3em} \eta_n(x) = \eta'_n(x) \; \text{for all} \;
    (x,n) \in \mathsf{cone}(b_\inn , s_\inn , h),
  \end{split}
\end{align}
where $\eta$ and $\eta'$ are both constructed from \eqref{eq:DCP-dyn}
with the same $\omega$'s. In words, when there are $1$'s at the bottom
of the outer cone, a configuration from
$G(b_\inn , b_\out , s_\inn , s_\out ,h)$ guarantees successful
coupling inside the inner cone irrespective of what happens outside
the outer cone.
\begin{lemma}
  \label{lem:goodconshells}
  For parameters $p$, $b_\inn$, $b_\out$, $s_\inn$ and $s_\out$ as in
  Lemma~\ref{lem:CPconeconv},
  \begin{equation}
    \Pr\big( \omega|_{\mathsf{cs}(b_\inn , b_\out ,
      s_\inn , s_\out ,h)}
    \in G(b_\inn , b_\out , s_\inn , s_\out ,h)\big)
    \ge 1-\varepsilon
  \end{equation}
  uniformly in $h \in \N$.
\end{lemma}
\begin{proof}
  The assertion follows from Lemma~\ref{lem:CPconeconv} because if the
  event $G_1 \cap G_2$ defined there occurs, then
  $\omega|_{\mathsf{cs}(b_\inn , b_\out , s_\inn , s_\out ,h)} \in
  G(b_\inn , b_\out , s_\inn , s_\out ,h)$ holds (recall
  Remark~\ref{rem:obsG1G2}).
\end{proof}

\begin{figure}
  \centering
  \begin{tikzpicture}[yscale=0.65]
    \draw[thick] (-0.7,7)--(-0.5,6) -- (0.5,6) -- (0.7,7);

    \draw[dashed] (-0.6,0) -- (0,7) ; 

    \draw[thick] (-1.97,7) -- (-0.9,0)  -- (-0.3,0) -- (0.77,7);

    \draw[dotted] (-3,7) -- (0.77,7) ;
    \draw (-4,7) node {\footnotesize $0$};

    \draw[dotted] (-3,6) -- (0.77,6) ;
    \draw (-4,6) node {\footnotesize $-t_\ell$};

    \draw[dotted] (-3,0) -- (0.77,0) ;
    \draw (-4,0) node {\footnotesize $-t_{\ell+1}$};

    \draw[dotted] (0,7) -- (0,-0.5);
    \draw (0,-0.8) node{\footnotesize $0$};

    \draw[->] (-3,-0.5) -- (-3,7.5);
    \draw[->] (-3,-0.5) -- (1.5,-0.5);

    \draw[dotted] (0.77,-0.5) -- (0.77,7);
    \draw (0.77,-0.8) node {\footnotesize $a_2$};

    \draw[dotted] (-1.97,-0.5) -- (-1.97,7);
    \draw (-1.97,-0.8) node {\footnotesize $-a_2$};

    \draw[dotted] (-0.7,7) -- (-0.7,-0.5) ;
    \draw (-0.7,-0.8) node {\footnotesize $-a_1$};
  \end{tikzpicture}
  \caption{Growth condition for the sequence $(t_\ell)$: The small
    inner cone is
    $\cone(t_\ell s_{\mathrm{max}}+b_\out ,s_\out ,t_\ell)$ shifted to
    the base point $(0,-t_\ell)$. The big outer cone is
    $\cone(b_\inn, s_\inn, t_{\ell+1})$ shifted to the base point
    $(-t_{\ell+1}s_{\mathrm{max}},-t_{\ell+1})$. The slope of the
    dashed line is $s_{\mathrm{max}}$. The sequence $(t_\ell)$ must
    satisfy $a_1 < a_2$ for
    $ a_1=s_\out t_{\ell}+b_\out +s_{\mathrm{max}} t_{\ell}$ and
    $a_2=s_\inn t_{\ell+1}+b_\inn - s_{\mathrm{max}} t_{\ell+1}$. }
  \label{fig:t-ell-cond}
\end{figure}
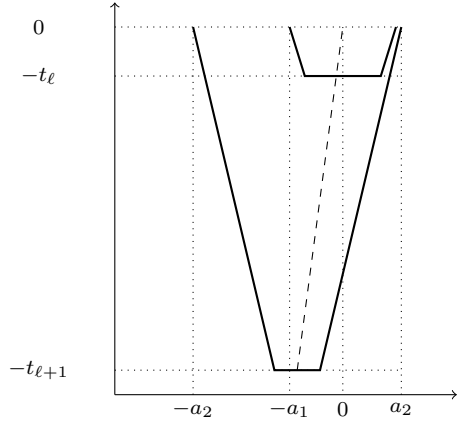

Let us denote the space-time shifts on $\Z^d \times \Z$ by
$\Theta^{(x,n)}$, i.e.,
\begin{equation}
  \label{eq:def-shift}
  \Theta^{(x,n)}(A) = \{ (x+y,m+n) : (y,m) \in A\}
  \quad \text{for}\;\: A \subset \Z^d \times \Z.
\end{equation}
An elementary geometric consideration reveals that one can choose a
deterministic sequence $t_\ell \nearrow \infty$ with the property that
for $\ell \in \N$ and $\norm{x} \le s_{\mathrm{max}} t_{\ell+1}$
\begin{align}
  \label{eq:conenesting2}
  \Theta^{(0,-t_\ell)}\big(\mathsf{cone}(t_\ell
  s_{\mathrm{max}}+b_\out ,s_\out ,t_\ell) \big) \subset
  \Theta^{(x,-t_{\ell+1})}\big(\mathsf{cone}(b_\inn , s_\inn ,
  t_{\ell+1})\big).
\end{align}
Note that this essentially enforces $t_\ell \approx \rho^\ell$ for a
suitable $\rho>1$. Indeed, a worst case picture (see
Figure~\ref{fig:t-ell-cond}) shows that we need
\begin{align*}
  t_{\ell+1} s_\inn + b_\inn - t_{\ell+1} s_{\mathrm{max}} & > t_\ell
  s_{\mathrm{max}} + b_\out + t_\ell s_\out \\ \intertext{which is
  equivalent to} t_{\ell+1} & >
  \frac{t_\ell(s_\out + s_{\mathrm{max}}) + b_\out - b_\inn }{s_\inn
                              -s_{\mathrm{max}}}.
\end{align*}
Thus, we can use (for $\ell$ sufficiently large)
\begin{align}
  \label{eq:tlc}
  t_{\ell} = \lceil \rho^\ell \rceil \; \text{ for any }  \; \rho >
  \frac{s_\out + s_{\mathrm{max}}}{s_\inn -s_{\mathrm{max}}}.
\end{align}
Furthermore note that using Lemma~\ref{lem:aprioribd} we obtain
\begin{align}
  \Pr\big(\exists \, n \le t_\ell \, : \, \norm{X_n} >
  s_{\mathrm{max}} t_\ell \big) \le \sum_{n= \lceil t_\ell
    s_{\mathrm{max}} \rceil}^{t_\ell} \Pr\big( \norm{X_n} >
  s_{\mathrm{max}} n \big) \le C' e^{-c' t_\ell}.
\end{align}
Since $t_\ell$ grows exponentially in $\ell$, the right hand side is
summable in $\ell$. Thus, from some random $\ell_0$ on, we have
$\sup_{n \le t_\ell} \norm{X_n} \le s_{\mathrm{max}} t_\ell$ for all
$\ell \ge \ell_0$, and $\ell_0$ has very short tails.

\subsection{Proof of Theorem~\ref{thm:LLNuCLTmodel1}}
\label{sec:proof-theor-refthm:l}

The proof of Theorem~\ref{thm:LLNuCLTmodel1} relies on a regeneration
construction and moment estimates for the increments between
regeneration times (recall the discussion after Remark~\ref{rem:time-rev}).
We now have prepared all the ingredients to carry out the argument,
which we prefer to give in a more verbal, descriptive style.
While all the concepts and properties discussed below can be
easily expressed in mathematical formulas we believe that the resulting
increase in length and in notational heaviness would burden the text
unnecessarily without improving neither readability nor
understandability.

That said, the regeneration construction goes as follows (see
also Figure~\ref{fig:cones1}):
  \begin{enumerate}
  \item Go to the first $\sigma_i$ after $t_1$, check if $\eta$ in the
    $b_\out $-neighbourhood of $(X_{\sigma_i},-\sigma_i)$ is
    $\equiv 1$, the path (together with its tube and decorations) has
    stayed inside the interior of the corresponding conical shell
    based at the current space-time position and the $\omega$'s in
    that conical shell are in the good set as defined in
    \eqref{def:goodconshells}. This has positive (in fact, very high)
    probability (cf.~Lemma~\ref{lem:xconetimes}) and if it occurs, we
    have found the first regeneration time $T_1$.

  \item If the event fails, we must try again. We successively check
    at times $t_2$, $t_3$, etc.: If not previously successful, at the
    $\ell$-th step let $\tilde{\sigma}_\ell$ be the first $\sigma_i$
    after $t_\ell$, check if $\tilde{\sigma}_\ell$ is a cone point for
    the decorated path beyond $t_{\ell-1}$ with
    $\norm{X_{\tilde{\sigma}_\ell}} \le s_{\mathrm{max}}
    \tilde{\sigma}_\ell$,
    the $\eta$'s in the $b_\out $-neighbourhood of
    $(X_{\tilde{\sigma}_\ell},-\tilde{\sigma}_\ell)$ are $\equiv 1$,
    $\omega$'s in the corresponding conical shell are in the good set
    as defined in \eqref{def:goodconshells} and the path (with tube
    and decorations) up to time $t_{\ell-1}$ is contained in the box
    of diameter $s_\out t_{\ell-1} + b_\out $ and height $t_{\ell-1}$.
    If this all holds, we have found the first regeneration time $T_1$.

    (We may assume that $\tilde{\sigma}_{\ell-1}$ is suitably close to
    $t_{\ell-1}$, this has very high probability by
    Lemma~\ref{lem:tautildetails}.)

  \item The path containment property holds from some finite $\ell_0$
    on. Given the construction and all the information obtained from
    it up to the $(\ell-1)$-th step, the probability that the other
    requirements occur is uniformly high (for the cone time property
    use Lemma~\ref{lem:xconetimes} with $k=t_\ell$; use
    \eqref{eq:omegatauidomination} to verify that the probability to
    see $\eta \equiv 1$ in a box around
    $(X_{\tilde{\sigma}_\ell},-\tilde{\sigma}_\ell)$ is high; use
    Lemma~\ref{lem:goodconshells} to check that conditional on the
    construction so far the probability that the $\omega$'s in the
    corresponding conical shell are in the good set is high, note that
    these $\omega$'s have not yet been looked at by the construction so far).

  \item \label{step:tlgeom} We will thus at most require a geometric
    number of $t_\ell$'s to construct the regeneration time $T_1$.
    Then we shift the space-time origin to $(X_{T_1},-T_1)$ and start
    afresh, noting that by construction, the law of
    $(\eta_{-k-T_1}(x+X_{T_1}))_{x \in \Z^d, k \in \Z}$ given all the
    information obtained in the construction so far equals the law of
    $(\eta_{-k}(x))_{x \in \Z^d, k \in \Z}$ conditioned on seeing the
    configuration $\eta_0 \equiv 1$ in the $b_\out $-box around $0$.

    The sequence $t_\ell$ grows exponentially in $\ell$ with rate
    $\rho$ (see \eqref{eq:tlc}) and we need to go to at most a random
    $\ell$ with geometric distribution with a success parameter
    $1-\delta$ very close to $1$. We thus can enforce a finite very
    high moment of the regeneration time:
    \begin{align}
      \label{eq:reg-tales}
      \begin{split}
        \Pr(\text{regeneration after time $n$})
        & \le \Pr(\text{more than $\log n/\log \rho$ steps needed}) \\
        & \le \delta^{\log n/\log \rho} = n^{-a},
      \end{split}
    \end{align}
    where $a = \log (1/\delta)/\log \rho$ can be made large by
    choosing $\delta$ small and $\rho$ close to $1$. Both is achieved
    by choosing $p$ close to $1$.
  \end{enumerate}

  We obtain a sequence of random times $T_1 < T_2 < \cdots$ such that
  $(X_{T_i}-X_{T_{i-1}}, T_i-T_{i-1})_{i=2,3,\dots}$ are i.i.d.\ and
  $\E[T_1^b], \E[(T_2-T_1)^b] < \infty$ and hence also
  $\E[\norm{X_{T_1}}^b], \E[\norm{X_{T_2}-X_{T_1}}^b] < \infty$ for
  some $b > 2$. The existence of such regeneration times implies
  Theorem~\ref{thm:LLNuCLTmodel1} by standard arguments, see e.g.\ the
  proof of Corollary~1 in \cite{Kuczek:1989} (it is easy to see from
  the construction that $X_{T_i}-X_{T_{i-1}}$ is not a deterministic
  multiple of $T_i-T_{i-1}$) and the proof of Theorem~4.1.\ in
  \cite{Sznitman:2000} for the functional CLT. Note that the speed
  must be $0$ by the assumed symmetry; see
  Assumption~\ref{ass:distr-sym}.

  \begin{remark}
    In the general case without the Assumption~\ref{ass:distr-sym} the
    above argument yields that there must be a limit speed, its value
    is given only implicitly as $\E[X_{T_2}-X_{T_1}]/\E[T_2-T_1]$.

    If in Assumption~\ref{ass:distr-sym} we would additionally require
    symmetries with respect to coordinate permutations and with
    respect to reflections along coordinate hyperplanes then the
    limiting law $\Phi$ would be a (non-trivial) centred isotropic
    $d$-dimensional normal law, cf.\ the proof of Theorem~1.1 in
    \cite{BCDG13}.
  \end{remark}

\begin{figure}
  \centering
  \begin{tikzpicture}[scale=0.7] 
    \draw [<-, thick] (-3,0) -- (-3,11.5); \draw (-4.2,3.3)
    node[rotate=90] {\small positive time (for the walk)};

    \filldraw[black] (0.7,11) circle (0.07cm); \draw[dotted] (0.7,11)
    -- (-3,11); \draw (-3.5,11) node {\scriptsize $t_0$};
    \filldraw[black] (1.5,10) circle (0.07cm); \draw[dotted] (1,10) --
    (-3,10); \draw (-3.5,10) node {\scriptsize $t_1$}; \draw[very
    thick] (0.5,10) -- (2.5,10); \draw (0.5,10) --
    ++(105:1); 
    \draw (2.5,10) -- ++(75:1);

    \filldraw[black] (0.3,6.5) circle (0.07cm); \draw[dotted]
    (0.3,6.5) -- (-3,6.5); \draw (-3.5,6.5) node {\scriptsize $t_2$};

    \draw[very thick] (-0.7,6.5) -- (1.3,6.5); \draw (-0.7,6.5) --
    ++(105:4.65); \draw (1.3,6.5) -- ++(75:4.65);

    \filldraw[black] (1.7,0.5) circle (0.07cm); \draw[dotted]
    (1.7,0.5) -- (-3,0.5); \draw (-3.5,0.5) node {\scriptsize $t_3$};

    \draw[very thick] (0.7,0.5) -- (2.7,0.5); \draw (0.7,0.5) --
    ++(105:10.85); \draw (2.7,0.5) -- ++(75:10.85);

    \draw [rounded corners =3pt] (0.7,11) .. controls (-0.5,10.4) and
    (0.9,10.7)..(1.5,10) .. controls (1.5,9.5) and (0.3,9) ..
    (0.7,8.5) .. controls (1.4,8) and (1,7.5).. (0.8,7.2) .. controls
    (0.65,7) and (0.4,6.8) .. (0.3,6.5) .. controls (0,5.9) and
    (0.9,5.3) .. (1.2,4.7) .. controls (1.6,4.1) and (1.4,3.5) ..
    (1,2.9) .. controls (0.6,2.2) and (0.9,1.5) .. (1.2,1) .. controls
    (1.4,0.7) .. (1.7,0.5);
  \end{tikzpicture}
  \caption{A schematic example: The walk passing through a sequence of
    cones in an attempt to regenerate. Here, $\tau_1=\tau_0+t_3$.}
  \label{fig:cones1}
\end{figure}
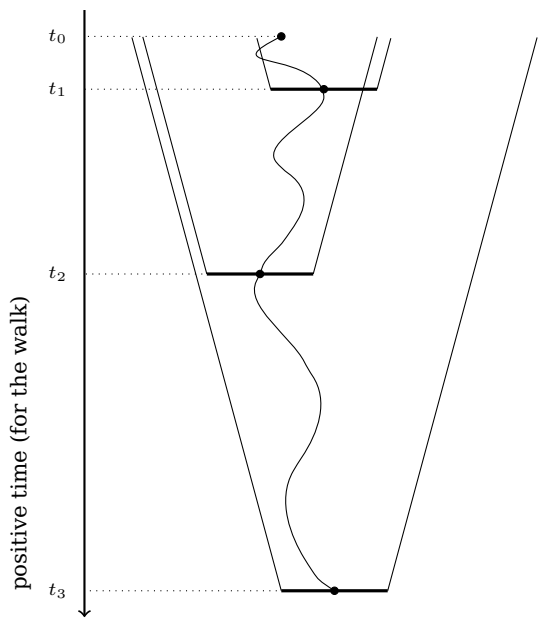

\section{A more abstract set-up}
\label{sect:abstr-setup}

The goal of this section is to present an abstract set-up where a
renewal construction similar to the one of the previous section can be
implemented. Our main motivation of this set-up is to study the
dynamics of ancestral lineages in spatial populations, but it can be
likely applied for other types of directed random walks in random
environment.

In Sections~\ref{subsect:abstr-assumpt}
and~\ref{subsect:abstr-assumpt-walk}, we present certain abstract
assumptions on the random environment and the associated random walk.
These assumptions allow to control the behaviour of the random walk
using a regeneration construction that is very similar to the one from
Section~\ref{sect:model1}. In particular, they allow to link the model
with oriented percolation, using a coarse-graining technique.

We would like to stress that coarse-graining does not convert the
presented model to the one of the previous section. In particular, the
nature of regenerations is somewhat different. We will see that the
sequence of regeneration times and associated displacements,
$(T_{i+1}-T_i, X_{T_{i+1}}-X_{T_i})_{i\ge 2}$ is not i.i.d.~but can be
generated as a certain function of an irreducible, finite-state Markov
chain and additional randomness. By ergodic properties of such chains,
this will lead to the same results as previously.

\begin{theorem}
  \label{thm:abstr-regen}
  Let the random environment $\eta$ and the random walk $X$ satisfy
  the assumptions of Sections~\ref{subsect:abstr-assumpt} and
  \ref{subsect:abstr-assumpt-walk} below with sufficiently small
  parameter $\varepsilon_U$. Then the random walk $X$ satisfies the
  strong law of large numbers with speed $0$ and the annealed central
  limit theorem with non-trivial covariance matrix. A corresponding
  functional central limit theorem holds as well.
\end{theorem}

A concrete example satisfying the abstract assumptions of
Sections~\ref{subsect:abstr-assumpt}
and~\ref{subsect:abstr-assumpt-walk} will be given in
Section~\ref{sect:model.real}. They can also be verified for the
oriented random walk on the backbone of the oriented percolation
cluster which was treated in \cite{BCDG13} using simpler, but related,
methods.

\subsection{Assumptions for the environment}
\label{subsect:abstr-assumpt}
We now formulate two assumptions on the random environment. The first
assumption requires that the environment is Markovian (in the positive
time direction), and that there is a `flow construction' for this
Markov process, coupling the processes with different starting
conditions. The second assumption then allows to use the
coarse-graining techniques and the links with oriented percolation.

Formally, let
\begin{align*}
  U\coloneqq \{U(x,n): x\in\Z^d, n\in\Z\}
\end{align*}
be an i.i.d.\ random field, $U(0,0)$ taking values in some Polish
space $E_U$ ($E_U$ could be $\{-1,+1\}$, $[0,1]$, a path space, etc.).
Furthermore for $R_{ \eta} \in \N$ let
$B_{R_{ \eta}} = B_{R_{ \eta}}(0) \subset \Z^d$ be the ball of radius
$R_{ \eta}$ around $0$ with respect to $\sup$-norm. Let
\begin{align*}
  \varphi : \Z_+^{B_{R_{  \eta}}} \times E_U^{B_{R_{\eta}}} \to \Z_+
\end{align*}
be a measurable function.

\begin{assumption}[Markovian, local dynamics, flow construction]
\label{ass:markovdyn}
We assume that $\eta \coloneqq (\eta_n)_{n \in \Z}$ is a Markov
chain with values in $\Z_+^{\Z^d}$ 
whose evolution is local in the sense that $\eta_{n+1}(x)$ depends
only on $\eta_n(y)$ for $y$ in a finite ball around $x$. In particular
we assume that $\eta$ can be realised using the `driving noise'
$U$ as
\begin{equation}
  \label{eq:etadynabstr}
  \eta_{n+1}(x) = \varphi\big( \restr{\theta^x \eta_n}{B_{R_{\eta}}},
  \restr{\theta^x U(\,\cdot\,,n+1)}{B_{R_{\eta}}} \big), \quad
  x\in\Z^d, \; n \in \Z.
\end{equation}
Here $\theta^x$ denotes the spatial shift by $x$, i.e.,
$\theta^x \eta_n(\,\cdot\,) = \eta_n(\,\cdot\,+x)$ and
$\theta^x U(\,\cdot\,, n+1) = U(\,\cdot\,+x,n+1)$. Furthermore
$\restr{\theta^x \eta_n}{B_{R_{ \eta}}}$ and
$\restr{\theta^x U(\,\cdot\,,n+1)}{B_{R_{\eta}}}$ are the
corresponding restrictions to the ball $B_{R_{\eta}}$.
\end{assumption}

Note that \eqref{eq:etadynabstr} defines a flow, in the sense that
given a realisation of $U$ we can construct $\eta$ simultaneously
for all starting configurations. In most situations we have in mind
the constant zero configuration $\underline{0} \in \Z_+^{\Z^d}$ is an
equilibrium for $\eta$, that is,
\begin{align*}
  \varphi\big(\restr{\underline{0}}{B_{R_{\eta}}}, \, \cdot\, \big)
  \equiv 0,
\end{align*}
and there is another non-trivial equilibrium. It will be a consequence
of our assumptions that the latter is in fact the unique non-trivial
ergodic equilibrium.

\smallskip

The second assumption, inspired by \cite{BD07}, allows for comparison
of $\eta$ with a supercritical oriented percolation on a suitable
space-time grid. Loosely speaking, this assumption states that if we
have a good configuration on the bottom of a (suitably big) block and
the driving noise inside the blocks is good, too, then the
configuration on the top of the block is also good and the good region
grows with high probability. Furthermore if we input two good
configurations at the bottom of the block then good noise inside the
block produces a coupled region at the top of the block.

Formally, let $L_{\mathrm{t}}, L_{\mathrm{s}} \in \N$. We use
space-time boxes whose `bottom parts' are centred at points in the
coarse-grained grid $L_{\mathrm{s}} \Z^d \times L_{\mathrm{t}}\Z$.
They will be partly overlapping in the spatial direction but not in
the temporal direction, and we typically think of
$L_{\mathrm{t}} > L_{\mathrm{s}} \gg R_\eta$.

For $(\wt{x}, \wt{n}) \in \Z^d \times \Z$ we set
\begin{equation}
  \mathsf{block}_m(\wt{x},\wt{n}) \coloneqq \big\{ (y,k)
  \in \Z^d \times \Z \, : \,
  \norm{y-L_{\mathrm{s}} \wt{x}} \le
  m L_{\mathrm{s}}, \wt{n} L_{\mathrm{t}} < k \le
  (\wt{n}+1) L_{\mathrm{t}}\big\} ,
\end{equation}
and
$\mathsf{block}(\wt{x},\wt{n}) \coloneqq
\mathsf{block}_1(\wt{x},\wt{n})$;
see Figure~\ref{fig:boxes}. For a set $A \subset \Z^d \times \Z$,
slightly abusing the notation, we denote by $\restr{U}{A}$ the
restriction of the random field $U$ to $A$. In particular,
$\restr{U}{\mathsf{block}_4(\wt{x},\wt{n})}$ is the restriction of $U$
to $\mathsf{block}_4(\wt{x},\wt{n})$ and can be viewed as element of
$E_U^{B_{4 L_{\mathrm{s}}}(0) \times \{1,2,\dots,L_{\mathrm{t}}\}}$.

\smallskip

\begin{assumption}[`Good' noise configurations and propagation of
  coupling]
  There \label{ass:coupling} exist a finite set of `good' local
  configurations $G_{\eta} \subset \Z_+^{B_{2 L_{\mathrm{s}}}(0)}$ and a
  set of `good' local realisations of the driving noise
  $G_U \subset E_U^{B_{4 L_{\mathrm{s}}}(0) \times
    \{1,2,\dots,L_{\mathrm{t}}\}}$ with the following properties:
  \begin{itemize}
  \item For a suitably small $\varepsilon_U$,
    \begin{align}
      \label{eq:goodblockprob}
      \Pr\big( \restr{U}{\mathsf{block}_4(0,0)} \in G_U \big)
      \ge 1-\varepsilon_U
    \end{align}
  \item For any $(\wt{x}, \wt{n}) \in \Z^d \times \Z$ and any
    configurations
    $\eta_{\wt{n} L_{\mathrm{t}}}, \eta'_{\wt{n} L_{\mathrm{t}}} \in
    \Z_+^{\Z^d}$ at time $\wt{n} L_{\mathrm{t}}$,
    \begin{align}
      \label{eq:contraction}
      \begin{split}
        & \restr{\eta_{\wt{n} L_{\mathrm{t}}}}{B_{2 L_{\mathrm{s}}}
          (L_{\mathrm{s}} \wt{x})}, \, \restr{\eta'_{\wt{n}
            L_{\mathrm{t}}}}{B_{2 L_{\mathrm{s}}} (L_{\mathrm{s}}
          \wt{x})} \in G_{\eta} \quad \text{and} \quad
        \restr{U}{\mathsf{block}_4(\wt{x},\wt{n})} \in G_U  \\[+0.5ex]
        & \Rightarrow \;\; \eta_{(\wt{n}+1) L_{\mathrm{t}}}(y) =
        \eta'_{(\wt{n}+1) L_{\mathrm{t}}}(y) \quad \text{for all $y$
          with} \; \norm{y-L_{\mathrm{s}}
          \wt{x}} \le 3 L_{\mathrm{s}}  \\
        & \qquad \text{and} \quad \restr{\eta_{(\wt{n}+1)
            L_{\mathrm{t}}}}{B_{2 L_{\mathrm{s}}} (L_{\mathrm{s}}
          (\wt{x}+\wt{e}))} \in G_{\eta} \; \text{for all $\wt{e}$
          with} \; \norm{\wt{e}} \le 1,
      \end{split}
      \\ \intertext{and}
      \label{eq:propagation.coupling}
        & \restr{\eta_{\wt{n} L_{\mathrm{t}}}}{B_{2 L_{\mathrm{s}}}(L_{\mathrm{s}} \wt{x})}
          =  \restr{\eta'_{\wt{n} L_{\mathrm{t}}}}{B_{2 L_{\mathrm{s}}}(L_{\mathrm{s}} \wt{x})}
          \quad \Rightarrow \quad
          \eta_k(y) = \eta'_k(y) \;\; \text{for all} \; (y,k) \in
          \mathsf{block}(\wt{x},\wt{n}),
    \end{align}
    where $\eta=(\eta_n)$ and $\eta'=(\eta'_n)$ are given by
    \eqref{eq:etadynabstr} with the same $U$ but possibly different
    initial conditions.
  \item There is a fixed (e.g., $L_{\mathrm{s}}$-periodic or even
    constant in space) reference configuration
    $\eta^{\mathrm{ref}} \in \Z_+^{\Z^d}$ such that
    $\restr{\eta^{\mathrm{ref}}}{B_{2 L_{\mathrm{s}}}(L_{\mathrm{s}}
      \wt{x})} \in G_{\eta}$ for all $\wt{x} \in \Z^d$.
  \end{itemize}
\end{assumption}

Note that if the event in
\eqref{eq:contraction}--\eqref{eq:propagation.coupling} occurs then a
coupling of $\eta$ and $\eta'$ on $B_{2 L_{\mathrm{s}}}(L_{\mathrm{s}}
\wt{x}) \times \{\wt{n} L_{\mathrm{t}} \}$ has propagated to
$B_{2 L_{\mathrm{s}}}(L_{\mathrm{s}}(\wt{x}+\wt{e})) \times \{(\wt{n}+1) L_{\mathrm{t}} \}$
for $\norm{\wt{e}} \le 1$ and also the fact that the local
configuration is `good' has propagated. The event in
\eqref{eq:contraction} enforces propagation of goodness and can
also be viewed as a contractivity property of the local dynamics.
In other words the flow tends to merge local configurations once they
are in the `good set'.

\begin{figure}
  \centering
  \begin{tikzpicture}
    \draw[->,thick,] (-6,-1) -- (6,-1);
    \draw (6.2,-1.3) node {\footnotesize $\Z^d$};

    \draw[dotted] (0,2)-- (0,-1)
                  (0.8,2)-- (0.8,-1)
                  (1.6,2)-- (1.6,-1)
                  (2.4,2)-- (2.4,-1)
                  (3.2,2)-- (3.2,-1)
                  (4,2)-- (4,-1)
                  (-0.8,2)-- (-0.8,-1)
                  (-1.6,2)-- (-1.6,-1)
                  (-2.4,2)-- (-2.4,-1)
                  (-3.2,2)-- (-3.2,-1)
                  (-4,2)-- (-4,-1);

    \draw (0,-1) -- (0,-1.2) (3.2,-1) -- (3.2,-1.2) (-3.2,-1) -- (-3.2,-1.2);
    \draw (0,-1.5) node {\footnotesize $\wt x L_{\mathrm{s}}$};
    \draw (3.2,-1.5) node {\footnotesize $(\wt x+K_\eta) L_{\mathrm{s}}$};
    \draw (-3.2,-1.5) node {\footnotesize $(\wt x -K_\eta) L_{\mathrm{s}}$};

    \draw[->,thick,] (-6,-1) -- (-6,2.7);
    \draw (-6.3,2.6) node {\footnotesize $\Z$};

    \draw[dotted] (-6,0) -- (-4,0)
                  (-6,2) -- (4,2);

    \draw (-6.5,0) node {\footnotesize $\wt n L_{\mathrm{t}}$};
    \draw (-6.75,2) node {\footnotesize $(\wt n +1) L_{\mathrm{t}}$};

    \draw[very thick] (-0.8,0.1) -- (-0.8,2) -- (0.8,2) -- (0.8,0.1) -- cycle;
    \draw[very thick,dashed] (-3.2,0) -- (-1.3,2) -- (1.3,2)-- (3.2,0)--cycle;
    \begin{scope}
      \path [fill=black,opacity=0.15] (-4,0) -- (-4,2) -- (4,2) -- (4,0) -- cycle;
    \end{scope}
 \end{tikzpicture}
 \caption{Locality of the construction of $(\eta_n)$ on the block
   level for $d=1$. If $U$ is known in the grey region and
   $\eta_{\wt n L_{\mathrm{t}}}$ is known on the bottom of the dashed
   trapezium then the configurations $\eta_{k}$ are completely
   determined inside $\mathsf{block} (\wt{x},\wt{n})$ drawn in solid
   lines.}
  \label{fig:boxes}
\end{figure}
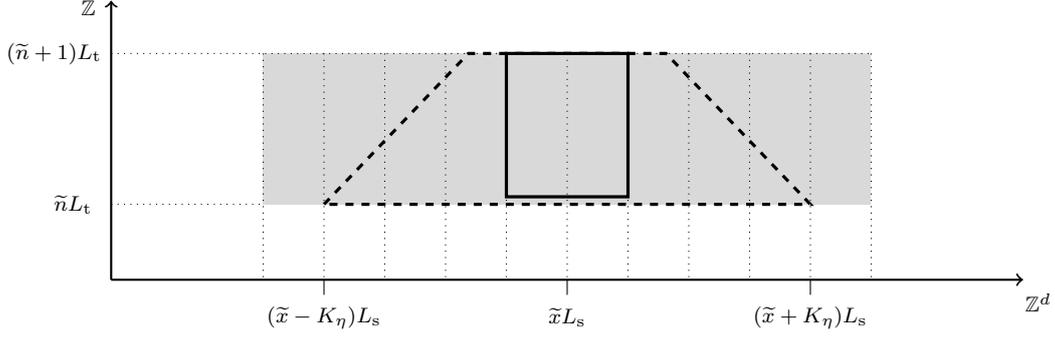

\begin{remark}[Locality on the block level]
  Put \label{rem:localityforblocks}
  \begin{equation}
    K_{\eta} \coloneqq R_{\eta} \big( \lceil
    \tfrac{L_{\mathrm{t}}}{L_{\mathrm{s}}} \rceil + 1 \big) .
  \end{equation}
  From the local construction of $\eta$ given in
  \eqref{eq:etadynabstr} it follows easily (see
  Figure~\ref{fig:boxes}) that for fixed
  $(\wt{x},\wt{n}) \in \Z^d \times \Z$ the values
  $\eta_n(x)$ for
  $(x,n) \in \mathsf{block}(\wt{x},\wt{n})$ are
  completely determined by $\eta_{\wt{n} L_{\mathrm{t}}}$
  restricted to
  $B_{K_{ \eta} L_{\mathrm{s}}}(\wt{x} L_{\mathrm{s}})$ and $U$
  restricted to
  $\cup_{\norm{\wt{y}} \le K_{\eta}}
  \mathsf{block}(\wt{x}+\wt{y},\wt{n})$.
\end{remark}


Using the above assumptions, it is fairly standard to couple $\eta$ to
an oriented percolation cluster. Recall the notation in
Section~\ref{sec:defin-model-results} and in particular the definition
of the stationary discrete time contact process in \eqref{eq:CP}.

\begin{lemma}[Coupling with oriented percolation]
  \label{lem:abstr-OCcoupl}
  Put
  \begin{align}
    \wt{U}(\wt{x}, \wt{n}) \coloneqq
    \ind{\restr{U}{\textstyle
    \mathsf{block}_4(\wt{x},\wt{n})} \in G_U},
    \quad (\wt{x}, \wt{n}) \in \Z^d \times \Z.
  \end{align}
  If $\varepsilon_U$ is sufficiently small,
  we can couple $\wt U(\wt x, \wt n)$ to an
  i.i.d.\ Bernoulli random field
  $\wt{\omega}(\wt{x}, \wt{n})$ with
  $\Pr(\wt{\omega}(\wt{x}, \wt{n})=1) \ge 1 -
  \varepsilon_{\wt{\omega}}$
  such that $\wt{U} \ge \wt{\omega}$,
  and $\varepsilon_{\wt{\omega}}$ can be chosen small
  (how small depends on $\varepsilon_U$, of course).

  Moreover, the process $\eta$ then has a unique non-trivial
  ergodic equilibrium and one can couple a stationary process
  $\eta = (\eta_n)_{n\in\Z}$ with $\eta_0$ distributed according to
  that equilibrium with $\wt \omega $ so that
  \begin{align}
    \label{e:tildeG}
    \wt{G}(\wt{x}, \wt{n}) \coloneqq
    \wt U(\wt x,\wt n)
    \ind{\restr{\eta_{\wt{n} L_{\mathrm{t}}}}{B_{2
    L_{\mathrm{s}}}(L_{\mathrm{s}} \wt{x})}\in G_\eta}
    \ge \wt{\xi}(\wt{x}, \wt{n}),
    \quad (\wt{x}, \wt{n}) \in \Z^d \times \Z
  \end{align}
  where
  $\wt{\xi}\coloneqq \{\wt{\xi}(\wt{x}, \wt{n}):
  \wt x \in \Z^d, \wt n \in \Z\}$ is the discrete time contact process
  defined by
  \begin{align}
    \label{eq:CP-coarse-grained}
    \wt{\xi}(\wt{x}, \wt{n}) \coloneqq  \ind{\Z^d
    \times \{-\infty\} \to^{\wt \omega} (x,n)}.
  \end{align}
\end{lemma}
\begin{proof}
  The first part is standard: Note that the $\wt{U}(\wt{x}, \wt{n})$'s
  are i.i.d.\ in the $\wt{n}$-coordinate, with finite range dependence
  in the $\wt{x}$-coordinate. Using \eqref{eq:goodblockprob},
  \eqref{eq:contraction} and \eqref{eq:propagation.coupling}, we can
  employ e.g.\ the Liggett-Schonman-Stacey device (\cite{LSS97} or
  \cite[Thm.~B26]{Lig99}).

\smallskip

  For the second part consider for each $k \in \N$ the process
  $\eta^{(k)} = (\eta^{(k)}_n)_{n \ge -k L_{\mathrm{t}}}$ which starts
  from $\eta^{(k)}_{-k L_{\mathrm{t}}} = \eta^{\mathrm{ref}}$ and
  evolves according to \eqref{eq:etadynabstr} for $n \ge -k
  L_{\mathrm{t}}$, using given $\wt{U}$'s which are coupled to
  $\wt\omega$'s as above so that $\wt{U} \ge
  \wt{\omega}$ holds. We see from the coupling properties
  guaranteed by Assumption~\ref{ass:coupling} and Lemma~\ref{lem:DCheight}
  below that the law of $\eta^{(k)}$ restricted to any finite
  space-time window converges.  By a diagonal argument we can take a
  subsequence $k_m \nearrow \infty$ such that $\eta_n(x) \coloneqq
  \lim_{m\to\infty} \eta^{(k_m)}_n(x)$ exists a.s.\ for all $(x,n) \in
  \Z^d \times \Z$, then \eqref{e:tildeG} and
  \eqref{eq:CP-coarse-grained} hold by construction.

  The fact that the law of limit is the unique non-trivial ergodic
  equilibrium can be proved analogously to \cite[Cor.~4]{BD07}.
\end{proof}

\begin{remark}[Clarification about the relation between $\wt{\xi}$ and
  $\eta$]
  The contact process $\wt{ \xi}$ is defined here with respect to
  $\wt \omega$ analogously to the definition of the discrete time
  contact process $\eta$ with respect to $\omega$ in \eqref{eq:CP}.
  The rationale behind this change of notation is that throughout the
  paper $\eta$ is a stationary population process (contact process
  in Section~\ref{sect:model1} and logistic BRW in Section
  \ref{sect:model.real}) and the random walk $X$ is interpreted as
  an ancestral lineage of an individual from that population. The
  coarse-grained contact process $\wt{ \xi}$ plays a different
  role. In particular, the knowledge of $\wt{ \xi}$ alone does not
  determine the dynamics of $X$; cf.\ definition of $X$ in \eqref{eq:abstr-walk}.
\end{remark}

Finally, we need the following technical assumption which is
sufficiently strong for our purposes but can be relaxed presumably.

\begin{assumption}[Irreducibility on $G_{\eta}$]
  On \label{ass:irred} $\wt{G}(\wt{x}, \wt{n})$, conditioned on seeing
  a particular local configuration $\chi \in \Z_+^{B_{2
      L_{\mathrm{s}}}(L_{\mathrm{s}} \wt{x})} \cap G_{\eta}$ at the
  bottom of the space-time box [time coordinate $\wt{n}
  L_{\mathrm{t}}$], every configuration $\chi' \in G_{\eta}$ has a
  uniformly positive chance of appearing at the top of the space-time
  box [time coordinate $(\wt{n}+1) L_{\mathrm{t}}$].
\end{assumption}

\begin{remark}
  For the discrete time contact process the above assumptions can be
  checked easily in the case $d=1$ when $p$ is sufficiently close to
  $1$. For $G_{\eta}$ we could for instance take configurations
  $\eta \in \{0,1\}^{\Z^d}$ with
  \begin{align*}
   \# \{ \norm{x} \le L_s/2, \eta (x) = 1\} \ge \frac23L_s.
  \end{align*}
  For $G_U$ we could take configurations of $\omega$'s for which this
  property propagates to the top of the block and its neighbours
  irrespective of the positions of the $1$'s in the initial
  configuration (cf.\ construction in the proof of
  Lemma~\ref{lem:CPconeconv}). For $d \ge 2$ one can reduce the
  argument to the one-dimensional case.
\end{remark}

\subsection{Assumptions for random walk}
\label{subsect:abstr-assumpt-walk}
We now state the assumptions for the random walk $X=(X_k)_{k
  =0,1,\dots}$ in the random environment generated by $\eta$. To this
end let $\widehat U \coloneqq (\widehat{U}(x,k) : x \in \Z^d, k \in
\Z_+)$ be an independent space-time i.i.d.\ field of random variables
uniformly distributed on $(0,1)$. Furthermore let
\begin{align}
  \label{eq:phix}
  \varphi_{ X} : \Z_+^{B_{R_{ X}}} \times \Z_+^{B_{R_{ X}}}
  \times [0,1] \to B_{R_{ X}}
\end{align}
a measurable function, where $R_{ X} \in \N$ is an upper bound on
the jump size as well as on the dependence range. Given $\eta$, let
$X_0=0$ and put
\begin{equation}
  \label{eq:abstr-walk}
  X_{k+1} \coloneqq X_k + \varphi_{ X}\big(\restr{\theta^{X_k}
    \eta_{-k}}{B_{R_{ X}}},
  \restr{\theta^{X_k} \eta_{-k-1}}{B_{R_{ X}}} , \widehat{U}(X_k,k) \big),
  \quad k =0,1,\dots.
\end{equation}
Note that, as usual, forwards time direction for $X$ is backwards
time direction for $\eta$.


\begin{assumption}[Closeness to SRW while on $\wt{G}=1$]
  A \label{ass:cl-srw} walker with dynamics \eqref{eq:abstr-walk}
  starting from the middle half of the top of a box with $\wt{G}(\wt
  x,\wt n)=1$ stays inside the box with high probability:
\begin{align}
  \label{eq:abstr-sb}
  \min_{z \, : \, \norm{z-\wt x} \le L_{\mathrm{s}}/2} \Pr\Big(
  \max_{(n-1) L_{\mathrm{t}} < k \le n L_{\mathrm{t}}} \norm{X_k-z}
  \le \frac{L_{\mathrm{s}}}{4} \, \Big| \,
  X_{(n-1) L_{\mathrm{t}}}=z, \wt{G}(\wt x,\wt n)=1,  \eta \Big)
  \ge 1 - \varepsilon.
\end{align}
\end{assumption}

\begin{remark}
  \begin{enumerate}[(a)]
  \item Note that \eqref{eq:abstr-sb} translates into the upper bound
    $\varepsilon R_{ X} + L_{\mathrm{s}}/(4L_{\mathrm{t}})$ for the speed
    of the walk $X$ on a block satisfying $\wt{G}(\wt x,\wt n)=1$. The
    factor $\frac14$ in $\frac{L_{\mathrm{s}}}{4}$ inside
    \eqref{eq:abstr-sb} is somewhat arbitrary. Depending on
    $\varepsilon_U$ and the ratio of $L_{\mathrm{s}}$ to
    $L_{\mathrm{t}}$ one could use a different factor.
  \item The simple Assumption~\ref{ass:cl-srw} allows to obtain a
    rough a priori bound on the speed of the walk and suffices for our
    purposes here, a more elaborate version would require successful
    couplings of the coordinates of $X$ with true random walks
    with a small drift while on the box, similar to the proof of
    Lemma~\ref{lem:aprioribd}.
  \end{enumerate}
\end{remark}

\begin{assumption}[Symmetry of $\varphi_{ X}$ w.r.t.\ point reflection]
  Let $\varrho$ be the (spatial) point reflection operator acting on
  $\eta$, i.e., $\varrho \eta_k(x) = \eta_k(-x)$ for any $k \in \Z$ and
  $x \in \Z^d$. We assume
  \begin{align}
    \label{eq:abstr-symm}
    \varphi_{ X}\bigl(\restr{\varrho
    \eta_0}{B_{R_{ X}}},\restr{\varrho \eta_{-1}}{B_{R_{ X}}},
    \widehat{U}(0,0)\bigr) = -
    \varphi_{ X} \bigl(\restr{\eta_0}{B_{R_{
    X}}},\restr{\eta_{-1}}{B_{R_{ X}}},
    \widehat{U}(0,0)\bigr).
  \end{align}
\end{assumption}


Note that \eqref{eq:abstr-symm} guarantees that the averaged speed of
$X$ will be $0$.

\subsection{The determining cluster of a block}

We now explain how Theorem~\ref{thm:abstr-regen} can be proved using
similar ideas as in Section~\ref{sect:model1}. In order to avoid
repetitions and to keep the length of the paper acceptable, we only
explain the major differences to the proof of Theorem~\ref{thm:LLNuCLTmodel1}.

The main change that should be dealt with is the fact that the
construction of the random walk $X$ requires not only the
knowledge of the coarse-grained oriented percolation $\wt { \xi}$,
but also of the underlying random environment $\eta $. This
additional dependence on $\eta $ should be controlled at
regeneration times. To tackle this problem,
Assumption~\ref{ass:coupling} and Lemma~\ref{lem:abstr-OCcoupl} play
the key role. By this lemma, the value of $\eta(x,n) $ can be
reconstructed by looking only at the driving noise $U$ in certain
finite set `below' $(x,n)$.

\smallskip

Formally, for $(\wt{x},\wt{n}) \in \Z^d \times \Z$ we define its
determining cluster $\mathsf{DC}(\wt{x},\wt{n})$ by the following
recursive algorithm:
\begin{enumerate}
\item Initially, put $\wt{k}\coloneqq\wt{n}$,
  $\mathsf{DC}(\wt{x},\wt{n})\coloneqq \{
  (\wt{x},\wt{n}) \}$.
\item If $\wt{\xi}(\wt{y},\wt{k})=1$ for all
  $(\wt{y},\wt{k}) \in
  \mathsf{DC}(\wt{x},\wt{n})$ : Stop.
\item Otherwise, for all blocks where this condition fails, add every
  block one time layer below that could have influenced it
  (cf.~Remark~\ref{rem:localityforblocks}), that is replace
  $ \mathsf{DC}(\wt{x},\wt{n})$ by
  \begin{align}
    \mathsf{DC}(\wt{x},\wt{n}) \cup
    \big\{ (\wt{z},\wt{k}-1)
    : \norm{\wt{z}-\wt{y}} \le K_\eta \;
    \text{for some $\wt{y}$ with $\wt{\xi}(\wt{y},\wt{k})=0$}\big\},
  \end{align}
  put $\wt{k}\coloneqq\wt{k}-1$ and go back to Step~2.
\end{enumerate}
\begin{lemma}
  \label{lem:DCheight}
  For $\varepsilon_U$ small enough, the height (and the diameter)
  of $\mathsf{DC}(\wt{x},\wt{n})$, defined as
\begin{align}
  \mathrm{height}(\mathsf{DC}(\wt{x},\wt{n})) \coloneqq
  \max\big\{\wt{n}- \wt{k} :
  (\wt{y},\wt{k})
  \in \mathsf{DC}(\wt{x},\wt{n}) \},
\end{align}
is finite a.s.\ with
exponential tail bounds.
\end{lemma}
\begin{proof}
  This can be shown as in the proof of Lemma~\ref{lem:drysitesbd}, see
  alternatively Lemma~7 in~\cite{Dur92}, or proof of Lemma~14
  in~\cite{BD07}.
\end{proof}

\begin{remark}
  On \label{rem:wblr} $\{ \wt{\xi}(\wt{x},\wt{n})=1\}$,
  $\restr{\eta}{\mathsf{block}(\wt{x},\wt{n})}$ is a function of
  local randomness. In fact it is then determined by
  $\restr{U}{\mathsf{block}_5(\wt{x},\wt{n}) \cup
    \mathsf{block}_5(\wt{x},\wt{n}-1)}$.
  Thus, $\eta$ on $\mathsf{block}(\wt{x},\wt{n})$ is determined by
  the `wet boundary' plus local randomness in a slightly `thickened'
  version of $\mathsf{DC}(\wt{x},\wt{n})$ which is the analogue of
  the `determining triangle' $D(x,n)$ from \eqref{eq:Dxn} in this
  coarse-grained context.

  To see this, consider the system $\eta'\coloneqq(\eta'_n :
  (\wt{n}-1)L_{\mathrm{t}} \le n \le (\wt{n}+1)L_{\mathrm{t}})$
  which starts from $\eta'_{(\wt{n}-1)L_{\mathrm{t}}} =
  \eta^{\mathrm{ref}}$ and uses the fixed boundary condition
  $\eta'_n(y)=\eta^{\mathrm{ref}}(y)$ for
  $\norm{y-L_{\mathrm{s}} \wt{x}} > 5 L_{\mathrm{s}}$ and
  $(\wt{n}-1)L_{\mathrm{t}} < n \le (\wt{n}+1)L_{\mathrm{t}}$. For
  $(y,n) \in \mathsf{block}_5(\wt{x},\wt{n}) \cup
  \mathsf{block}_5(\wt{x},\wt{n}-1)$ the values
  $\eta'_n(y)$ are computed using \eqref{eq:etadynabstr} with the same
  realisations of $U$ as the true system $\eta$.

  Note that $\wt{\xi}(\wt{x},\wt{n})=1$ implies that
  $\restr{U}{\mathsf{block}_4(\wt{x},\wt{n})} \in G_U$ and
  \begin{align*}
  \max_{\norm{\wt{e}} \le 1} \indset{G_U}\big(
  \restr{U}{\mathsf{block}_4(\wt{x}+\wt{e},\wt{n}-1)}\big)
  \indset{G_\eta}\big( \restr{\eta_{(\wt{n}-1)
  L_{\mathrm{t}}}}{B_{2 L_{\mathrm{s}}}(L_{\mathrm{s}}
  (\wt{x}+\wt{e}))} \big) = 1
  \end{align*}
  Now use \eqref{eq:contraction} to see that
  $\eta'_{\wt{n}L_{\mathrm{t}}}$ and $\eta_{\wt{n}L_{\mathrm{t}}}$
  agree on $B_{2 L_{\mathrm{s}}}(L_{\mathrm{s}} \wt{x})$, then use
  that and \eqref{eq:contraction}--\eqref{eq:propagation.coupling} to
  verify that $\eta'$ and $\eta$ agree on
  $\mathsf{block}(\wt{x},\wt{n})$.
\end{remark}

\subsection{A regeneration structure}
\label{subsect:regen}

In this section we construct regeneration times similar to those
constructed in Section~\ref{subs:regen-model1}. First we need to
introduce the analogue of the `tube around the path' and its
`decoration with determining triangles'; cf.\ equations
\eqref{eq:tuben}, \eqref{eq:Dxn} and \eqref{eq:dtuben}. We set
\begin{align}
  \label{eq:tuben-slice}
  \wt V_{\wt m}
  & \coloneqq \{ \wt x: \exists k, (\wt m-1)
    L_{\mathrm{t}} \le k \le \wt m L_{\mathrm{t}},
    \norm{X_k - \wt x L_{\mathrm{t}}} \le
    L_{\mathrm{s}} + R_{ X}  \}, \\
  \label{eq:tuben-cg}
  \mathsf{Tube}_{\wt n}
  & \coloneqq  \bigcup_{\wt m \le \wt n} \wt
    V_{\wt m} \times \{\wt m\} ,\\
  \label{eq:dtuben-cg}
  \mathsf{DTube}_{\wt n}
  & \coloneqq \bigcup_{(\wt x, \wt \jmath) \in
    \mathsf{Tube}_{\wt n}} \mathsf{DC}(\wt{x},\wt
    \jmath).
\end{align}
We define the coarse-graining function $\wt \pi: \Z^d \to \Z^d$
by
\begin{align}
  \label{eq:pi-cg}
  \wt \pi(x) = \wt \pi(x_1,\dots,x_d) = (\wt x_1,
  \dots, \wt x_d) \coloneqq \Bigl( \Big\lceil
  \frac{x_1}{L_{\mathrm{s}}} - \frac12 \Big\rceil, \dots,   \Big\lceil
  \frac{x_d}{L_{\mathrm{s}}} - \frac12 \Big\rceil \Bigr),
\end{align}
and denote by $\wt \rho(x)$ the relative position of $x$ inside
the block centred at $\wt x L_{\mathrm{s}}$, i.e.\ we set
\begin{align}
  \label{eq:pho-cg}
  \wt \rho(x) \coloneqq x - \wt x L_{\mathrm{s}}.
\end{align}
We define the \emph{coarse-grained random walk}
$\wt{X} = (\wt X_{\wt n})_{\wt n=0,1,\dots}$ and the \emph{relative
  positions} $\wt{ Y} = (\wt Y_{\wt n})_{\wt n=0,1,\dots}$ by
\begin{align}
  \label{eq:tildXY}
  \wt X_{\wt n} \coloneqq \wt \pi( X_{\wt
  n L_{\mathrm{t}}}) \quad \text{and} \quad
  \wt Y_{\wt n} \coloneqq \wt \rho( X_{\wt
  n L_{\mathrm{t}}}).
\end{align}
We need to keep track of the relative positions to preserve the
Markovian structure. Note that between the original random walk and
the coarse-grained components just defined we have the following
relation:
\begin{align*}
 X_{\wt n L_{\mathrm{t}}} = \wt X_{\wt n}
  L_{\mathrm{s}} + \wt Y_{\wt n}.
\end{align*}
We define the filtration $\wt{\mathcal F}\coloneqq (
\wt{\mathcal{F}}_{\wt n})_{\wt n = 0,1,\dots}$ by
\begin{align}
  \label{eq:tilde-filtration}
  \wt{\mathcal{F}}_{\wt n} \coloneqq \sigma\Bigl((\wt
  X_{\wt \jmath} , \wt Y_{\wt \jmath}) : 0 \le  \wt \jmath
  \le \wt n \bigr) \vee \sigma \bigl( \wt \omega
  (\wt y, \wt \jmath),  \wt \xi
  (\wt y, \wt \jmath), \restr{U}{\textstyle
  \mathsf{block}_4(\wt{y},\wt \jmath )} : (\wt y, \wt \jmath) \in
  \mathsf{DTube}_{\wt n} \Bigr).
\end{align}
To mimic the proofs of Section~\ref{sect:model1} for the model
considered here we need the following ingredients:
\begin{enumerate}
\item As in Lemma~\ref{lem:aprioribd} there exist
  $\wt s_{\mathrm{max}}$ (that is close to $\frac14$ under our
  assumptions) and positive constants $C, c$ such that
  \begin{align}
    \label{eq:apr-bound-cg}
    \Pr\big( \norm{\wt X_{\wt n}} > \wt
    s_{\mathrm{max}} \wt n \big) \leq C e^{-c \wt n}.
  \end{align}
\item For stopping times (analogous to $\sigma$'s in
  \eqref{eq:sigmas}) we set
  \begin{align}
    \label{eq:Dtilde}
    \wt D_{\wt n}
    & \coloneqq \wt n + \max\Big\{
      \mathrm{height}(\mathsf{DC}(\wt{x},\wt{n})):
      \wt x \in \wt V_{\wt n}\Big\}\\ \intertext{and define}
    \label{eq:sigmas-tild}
    \wt \sigma_0 \coloneqq 0, \quad \wt \sigma_{i}
    & \coloneqq \min\Big\{\wt m>\wt \sigma_{i-1} : \max_{\wt
      \sigma_{i-1} \leq \wt n \leq \wt m} \wt D_{\wt n}
      \leq \wt m \Big\}, \; i \ge 1.
  \end{align}
\end{enumerate}

\begin{lemma}
  When \label{lem:regen1-tild} $1-\varepsilon_{\wt \omega}$ is
  sufficiently close to $1$ there exist finite positive constants $c$
  and $C$ so that
  \begin{align}
    \label{eq:sigma-inc-tilde}
    \Pr\big( \wt \sigma_{i+1}-\wt \sigma_i > \wt n \, \big| \,
    \wt{\mathcal{F}}_{\wt \sigma_i} \big) \leq C e^{-c
    \wt n} \quad \text{for all }
    \wt n =1,2,\dots, \;  i =0,1,\dots \; \text{a.s.},
  \end{align}
  in particular, all $\wt \sigma_i$ are a.s.\ finite.
  Furthermore,
  \begin{align}
    \label{eq:omegatauidomination-tilde}
    \mathscr{L}\big( (\wt \omega(\cdot, -\wt
    \jmath -\wt \sigma_i)_{\wt \jmath =0,1,\dots } \, \big| \,
    \wt{\mathcal{F}}_{\wt \sigma_i}
    \big) \succcurlyeq \mathscr{L}\big( (\wt \omega(\cdot,
    -\wt \jmath)_{\wt \jmath =0,1,\dots} \big) \quad
    \text{for every $i=0,1,\dots$ a.s.},
  \end{align}
  where `$\succcurlyeq$' denotes stochastic domination.
\end{lemma}
\begin{proof}
  Analogous to the proof of Lemma~\ref{lem:regen1} (see also Lemma~\ref{lem:DCheight}).
\end{proof}

Similarly to the definition in \eqref{eq:decconetime} we say that
$\wt n$ is a $(b,s)$-cone time point for the decorated path beyond
$\wt m$ (with $\wt m < \wt n$) if
\begin{align}
  \label{eq:decconetime-tilde}
  \begin{split}
    &\mathsf{DTube}_{\wt n} \cap \big( \Z^d \times \{-\wt n,-\wt
    n+1,\dots,- \wt m\} \big) \\
    & \hspace{6em}
    \subset \big\{ (\wt x,-\wt \jmath) : \wt m \le
    \wt \jmath \le \wt n, \norm{\wt x-\wt
      X_{\wt n}} \le b + s(\wt n-\wt \jmath ) \big\}.
  \end{split}
\end{align}
In words, as in Section~\ref{subs:regen-model1} $\wt n$ is a
cone time point for the decorated path beyond $\wt m$ if the
space-time path
$(\wt X_{\wt \jmath},-\wt \jmath)_{\wt \jmath= \wt
  m,\dots, \wt n}$
together with its `tilde'-decorations is contained in the cone
with base radius $b$, slope $s$ and base point
$(\wt X_{\wt n},- \wt n)$.

\begin{lemma}
  \label{lem:xconetimes-tilde}
  There exist suitable $b$ and $s >\wt s_{\mathrm{max}}$
   such that for all finite $\wt{\mathcal{F}}$-stopping times
  $\wt T$ with
  $\wt T \in \{\wt \sigma_i : i \in \N\}$ a.s.\ (i.e.,
  $\wt T=\wt \sigma_J$ for a suitable random index $J$)
  and all $\wt k \in \N$, with
  $\wt T'\coloneqq\inf\{\wt \sigma_i : \wt \sigma_i \ge
  \wt k\}$
  \begin{align}
    \label{eq:xconetimeprob-tilde}
    \Pr\big( \text{$\wt T'$ is a $(b,s)$-cone time point for
    the decorated path beyond $\wt T$}
    \, \big| \, \wt{\mathcal{F}}_{\wt T} \big) \geq 1-\varepsilon \quad
  \end{align}
  a.s.\ on $\{\wt T<\wt k\}$. Furthermore $0<s-\wt
  s_{\mathrm{max}} \ll 1$ can be chosen small.
\end{lemma}
\begin{proof}
  Analogous to the proof of Lemma~\ref{lem:xconetimes}. Intermediate
  results, that is Lemma~\ref{lem:tautildetails} and
  Lemma~\ref{lem:xincrbd1}, can be adapted to the present
  situation.
\end{proof}

We now define `good configurations' of $\wt \omega$'s
(analogous to \eqref{def:goodconshells}). Recall the definition of a
cone shell in \eqref{eq:coneshell}. Let
$\wt G(b_\inn , b_\out , s_\inn , s_\out ,h) \subset
\{0,1\}^{\mathsf{cs}(b_\inn , b_\out , s_\inn , s_\out ,h)}$
be the set of possible $\wt \omega$-configurations in
$\mathsf{cs}(b_\inn , b_\out , s_\inn , s_\out ,h)$ with the property
\begin{align}
  \label{def:goodconshells-tilde}
  & \forall \,
    \wt \xi (\cdot,0) , \wt \xi' (\cdot,0) \in \{0,1\}^{\Z^d} \; \text{with} \;
    \wt \xi (\cdot,0)|_{\displaystyle B_{b_\out }(0)} = \wt
    \xi' (\cdot,0)|_{\displaystyle B_{b_\out }(0)}
    \equiv 1 \quad \text{and} \notag \\
  & \hspace{1em}\wt \omega \in \{0,1\}^{\Z^d \times \{1,\dots,h\}} \; \text{with} \;
    \wt \omega|_{\mathsf{cs}(b_\inn , b_\out , s_\inn , s_\out ,h)}
    \in \wt G(b_\inn , b_\out , s_\inn , s_\out ,h) \, : \\[1ex]
  & \hspace{3em} \wt \xi(\wt x,\wt n) =
    \wt \xi'(\wt x,\wt n)  \; \text{for all} \;
    (\wt x, \wt n) \in \mathsf{cone}(b_\inn , s_\inn , h) \notag
\end{align}
where $\wt { \xi}$ and $\wt{ \xi}'$ are both constructed from time $0$
using the same $\wt \omega$'s, i.e.\ when $A$ and $A'$ are subsets of
$\Z^d$ with $\indset{A} = \wt \xi(\cdot,0)$ and $\indset{A'} =  \wt
\xi'(\cdot,0)$ then (cf.\ \eqref{eq:CPA})
\begin{align*}
  \wt \xi (\cdot, n) = \indset{\{ \wt x \in \Z^d: A\times \{0 \} \to^{\wt
  \omega} (\wt x, \wt n)\}} \quad \text{and} \quad
  \wt \xi' (\cdot, n) = \indset{\{ \wt x \in \Z^d: A'\times \{0 \} \to^{\wt
  \omega} (\wt x, \wt n)\}}.
\end{align*}
Note that if $\wt \xi(\wt x,0) =1$ in the ball $B_{b_\out }(0)$ and
\[\wt \omega|_{\mathsf{cs}(b_\inn, b_\out, s_\inn, s_\out ,h)} \in \wt
G(b_\inn, b_\out, s_\inn, s_\out, h)\] then
\begin{align*}
  \{\eta_n(x) : (x,n) \in \mathsf{block}
  (\wt{x},\wt{n}),  \;  (\wt{x},\wt{n}) \in
  \mathsf{cone}(b_\inn , s_\inn , h)\}
\end{align*}
is a function of $\eta_0(y)$, $\norm{y} \le b_\out L_{\mathrm{s}}$ and
$\restr{U}{\mathsf{block}_4(\wt x,\wt n)}$,
$(\wt x,\wt n) \in \mathsf{cone}(b_\inn , s_\inn , h)$.
In particular, if we start with different $\eta_0'$ and $U'$ with
$\eta_0'(y) = \eta_0(y)$, $\norm{y} \le b_\out L_{\mathrm{s}}$ and
$\restr{U'}{\mathsf{block}_4(\wt x,\wt n)} =
\restr{U}{\mathsf{block}_4(\wt x,\wt n)}$,
$(\wt x,\wt n) \in \mathsf{cone}(b_\inn , s_\inn , h)$
then
\begin{align*}
  \eta_n(x) = \eta'_n(x) \; \text{ for all } \;  (x,n) \in
  \mathsf{block}
  (\wt{x},\wt{n}),  (\wt{x}, \; \wt{n})
  \in
  \mathsf{cone}(b_\inn , s_\inn , h).
\end{align*}

\begin{proof}[Proof sketch for Theorem~\ref{thm:abstr-regen}]
  We now have all the ingredients for the regeneration construction,
  to imitate the proof of Theorem~\ref{thm:LLNuCLTmodel1}. We again
  choose to keep the arguments more verbal and descriptive, hoping to
  strike a sensible balance between notational precision and
  readability.

 \smallskip

  First we choose a sequence $t_0,t_1,\dots$ with $t_\ell \uparrow
  \infty$ such that \eqref{eq:conenesting2} is satisfied with $\wt
  s_{\max}$ replacing $s_{\max}$ and parameters $b_\out$, $s_\out$,
  $b_\inn$ and $s_\inn$ adapted from Lemma~\ref{lem:xconetimes-tilde}.
  Recall from Remark~\ref{rem:wblr} that on the event $\{
  \wt{\xi}(\wt{x},\wt{n})=1\}$,
  $\restr{\eta}{\mathsf{block}(\wt{x},\wt{n})}$ is determined by
  $\restr{U}{\mathsf{block}_5(\wt{x},\wt{n}) \cup
    \mathsf{block}_5(\wt{x},\wt{n}-1)}$.

  \begin{enumerate}
  \item Go to the first $\wt \sigma_i$ after $t_1$, check if in the
    $b_\out $-neighbourhood of $(\wt X_{\wt \sigma_i},-\wt \sigma_i)$
    we have $\wt { \xi} \equiv 1$, the path (together with its tube
    and decorations) has stayed inside the interior of the
    corresponding conical shell based at the current space-time
    position and the $\wt \omega$'s in that conical shell are in the
    good set as defined in \eqref{def:goodconshells-tilde}. This has
    positive (in fact, very high) probability
    (cf.~Lemma~\ref{lem:xconetimes-tilde}) and if it occurs, we have
    found the `regeneration time'.

  \item If the event fails, we must try again. We successively check
    at times $t_2$, $t_3$, etc.: If not previously successful, at the
    $\ell$-th step let $\wt{\sigma}_{J(\ell)}$ be the first $\wt
    \sigma_i$ after $t_\ell$, check if $\wt{\sigma}_{J(\ell)}$ is a
    cone point for the decorated path beyond $t_{\ell-1}$ with
    $\norm{\wt X_{\wt{\sigma}_{J(\ell)}}} \le \wt s_{\mathrm{max}}
    \wt{\sigma}_{J(\ell)}$, the $\eta$'s in the $b_\out$-neighbourhood
    of $(X_{\tilde{\sigma}_\ell},-\tilde{\sigma}_\ell)$ are $\equiv
    1$, $\wt \omega$'s in the corresponding conical shell are in the
    good set as defined in \eqref{def:goodconshells-tilde} and the
    path (with tube and decorations) up to time $t_{\ell-1}$ is
    contained in the box of diameter $s_\out t_{\ell-1} + b_\out $ and
    height $t_{\ell-1}$. If this all holds, we have found the
    regeneration time.

    (We may assume that $\wt{\sigma}_{J(\ell-1)}$ is suitably close to
    $t_{\ell-1}$, this has very high probability by an adaptation of
    Lemma~\ref{lem:tautildetails}.)

  \item The path containment property holds from some finite $\ell_0$
    on. Given the construction and all the information obtained from
    it up to the $(\ell-1)$-th step, the probability that the other
    requirements occur is uniformly high: For the cone time property
    use Lemma~\ref{lem:xconetimes-tilde} with $\wt k=t_\ell$; use
    \eqref{eq:omegatauidomination-tilde} to verify that the
    probability to see $\wt{ \xi} \equiv 1$ in a box around $(\wt
    X_{\wt{\sigma}_{J(\ell)}},-\wt{\sigma}_{J(\ell)})$ is high; use (a
    notational adaptation of) Lemma~\ref{lem:goodconshells} to check
    that conditional on the construction so far the probability that
    the $\wt \omega$'s in the corresponding conical shell are in the
    good set $\wt G(b_\inn , b_\out , s_\inn , s_\out ,t_\ell)$ is
    high. Note that these $\wt \omega$'s have not yet been looked at.

  \item We thus construct a random time $\wt R_1$ with the following
    properties:
    \begin{enumerate}[(i)]
    \item
      $\wt \xi(\wt X_{\wt R_1}+\wt
      y,\wt R_1) = 1$ for all $\norm{\wt y} \le b_\out$;
    \item the decorated path up to time $\wt R_1$ is in
      $\mathsf{cone}(b_\inn , s_\inn , \wt R_1)$ centred at
      $(\wt X_{\wt R_1},\wt R_1)$;
    \item after centring the cone at base point
      $(\wt X_{\wt R_1},\wt R_1)$,
      $\wt \omega |_{\mathsf{cs}(b_\inn , b_\out , s_\inn ,
        s_\out ,\wt R_1)}$
      lies in the good set
      $\wt G(b_\inn , b_\out , s_\inn , s_\out ,\wt
      R_1)$.
    \end{enumerate}
    We will thus at most require a geometric number of $t_\ell$'s to
    construct the $\wt R_1$. As in step \ref{step:tlgeom} in
    the proof of Theorem~\ref{thm:LLNuCLTmodel1} we obtain
    \begin{align*}
      \Pr( \wt R_1 \ge \wt n) \le
      \Pr(\text{more than $\log \wt n/\log c$ steps needed}) \le
      \delta^{ \log \wt n/\log c} = \wt n^{-a},
    \end{align*}
    where again $a$ can be chosen large when $p$ is close to $1$.
  \item Set
    \begin{align*}
      \widehat \eta_1 & \coloneqq (\eta_{-L_{\mathrm{t}}\wt R_1} (x+
                        L_{\mathrm{s}} \wt X_{\wt R_1}):
                        \norm{x} \le b_\out L_{\mathrm{s}}),\\
      \widehat Y_1 & \coloneqq \wt{Y}_{\wt R_1}, \text{the displacement of }
                     X_{L_{\mathrm{t}}\wt R_1} \text{ relative to the
                     centre} \\
                      & \qquad \text{of the $L_{\mathrm{s}}$-box in
                        which it is contained}.
    \end{align*}
    Now we shift the space-time origin to $(\wt X_{\wt R_1},\wt R_1)$
    (on coarse-grained level). Then we start afresh conditioned on
    seeing
    \begin{enumerate}[(i)]
    \item configuration $\wt \xi \equiv 1$ in the $b_\out $-box around
      $0$ (on the coarse-grained level);
    \item $\widehat\eta_1$ on the $b_\out L_{\mathrm{s}}$ box (on the
      `fine' level);
    \item Displacement of the walker on the fine level relative to the
      centre of the corresponding coarse-graining box given by
      $\widehat Y_1$.
    \end{enumerate}
  \item We iterate the above construction to obtain a sequence of random
    times $\wt R_i$, positions $\wt X_{\wt R_i}$, relative displacements
    $\widehat Y_i$ and local configurations $\widehat \eta_i$. By
    construction
    \begin{align*}
      \bigl( \wt X_{\wt
      R_i} - \wt X_{\wt R_{i-1}},  \wt R_i
      -\wt R_{i-1}, \widehat Y_i, \widehat\eta_i
      \bigr)_{i \in \N}
    \end{align*}
    is a Markov chain. Furthermore,
    $\bigl(\widehat Y_i, \widehat\eta_i \bigr)_{i \in \N}$ is itself a
    finite state space Markov chain and the increments
    $\bigl(\wt X_{\wt R_{i+1}} - \wt X_{\wt R_{i}}, \wt R_{i+1} -\wt
    R_{i}\bigr)$
    depend only on $\bigl(\widehat Y_i, \widehat\eta_i \bigr)$.
\end{enumerate}

Along the random times $L_{\mathrm{t}} \wt R_{n}$,
\begin{align*}
  X_{L_{\mathrm{t}} \wt R_{n}} = \widehat Y_n + \sum_{i=1}^n
  L_{\mathrm{s}}\big( \wt X_{\wt
    R_i} - \wt X_{\wt R_{i-1}} \big)
\end{align*}
is an additive functional of a well-behaved Markov chain
(with exponential mixing properties) and
\begin{align*}
  \E\big[(\wt R_{i+1} -\wt R_{i})^a \, \big| \, \widehat
  Y_i, \widehat\eta_i\big] < \infty,
\quad
\E\big[\norm{\wt X_{\wt R_{i+1}} - \wt
  X_{\wt R_{i}}}^a \, \big| \, \widehat Y_i,
  \widehat\eta_i\big] < \infty
\end{align*}
for some $a>2$ uniformly in $\widehat Y_i, \widehat\eta_i$
(cf.~Step~4). From this representation the (functional) central limit
theorem can be deduced; see e.g.\ Chapter~1 in
\cite{KomorowskiLandimOlla2012} or Theorem~2 in
\cite{Rassoul-AghaSeppalainen2008}.

 Note that the speed of the random
walk must be $0$ by the symmetry assumption; see
\eqref{eq:abstr-symm}.
\end{proof}

\section{Example: an ancestral lineage of logistic
branching random walks}
\label{sect:model.real}

In this section we consider a concrete stochastic model for a locally
regulated, spatially distributed population that was introduced and
studied in \cite{BD07} and we refer the reader to that paper for a
more detailed description, interpretation, context and properties. We
call this logistic branching random walk because the function $f$ in
\eqref{def:f}, which describes the dynamics of the local mean
offspring numbers, is a `spatial relative' of the classical logistic
function $x \mapsto x (1-x)$ which appears in many (deterministic)
models for population growth under limited resources.
See also Remark~\ref{rem:furthermodels} below for a discussion
of related models and possible extensions.

After defining the model we recall and slightly improve some relevant
results from \cite{BD07}. Then in Proposition~\ref{prop:flowcoupling}
we show that in a high-density regime (see Assumption~\ref{ass:m-lambda})
assumptions from Section~\ref{sect:abstr-setup} are fulfilled by the
logistic branching random walk and the corresponding ancestral random
walk.

\subsection{Ancestral lineages in a locally regulated model}
\label{sec:ancestr-line-locally}

Let $p=(p_{xy})_{x,y \in \Z^d} =(p_{y-x})_{x,y \in \Z^d}$ be a
symmetric aperiodic stochastic kernel with finite range $R_p \ge 1$.
Furthermore let $\lambda=(\lambda_{xy})_{x,y \in \Z^d}$ be a
non-negative symmetric kernel satisfying
$0 \le \lambda_{xy} =\lambda_{0,y-x}$ and having finite range
$R_\lambda$. We set $\lambda_0 \coloneqq \lambda_{00}$ and for a
configuration $\zeta \in \R_+^{\Z^d}$ and $x \in \Z^d$ we define
\begin{align}
  \label{def:f}
  f(x; \zeta) \coloneqq \zeta(x) \big( m - \lambda_0 \zeta(x) -
  \sum_{z \ne x} \lambda_{xz} \zeta(z)  \big)^+.
\end{align}
We consider a population process $\eta\coloneqq (\eta_n)_{n \in \Z}$
with values in $\Z_+^{\Z^d}$, where as in the previous sections
$\eta_n(x)$ is the number of individuals at time $n \in \Z$ at site $x
\in \Z^d$. Before giving a formal definition of $\eta$ let us describe
the dynamics informally: Given the configuration $\eta_n$ in
generation $n$, each individual at $x$ (if any at all present) has a
Poisson distributed number of offspring with mean
$f(x;\eta_n)/\eta_n(x)$, independent of everything else. Offspring
then take an independent random walk step according to the kernel $p$
from the location of their mother. Then the offspring of all
individuals together form the next generations configuration
$\eta_{n+1}$. For obvious reasons $p$ and $\lambda$ are referred to as
\emph{migration} and \emph{competition} kernels respectively. Note
that in the case $\lambda \equiv 0$ the process $\eta$ is literally a
branching random walk.

We now give a formal construction of $\eta$. Let
\begin{align}
  \label{eq:Nxy}
 U \coloneqq \{U^{(y,x)}_n: n \in \Z,\; x, y \in \Z^d, \norm{x-y} \leq R_p\}
\end{align}
be a collection of independent Poisson processes on $[0,\infty)$ with
intensity measures of $U^{(y,x)}_n$ given by $p_{yx}\, dt$. The
natural state space for each $U^{(y,x)}_n$ is
\begin{equation}
  \wt{\mathcal{D}}\coloneqq \big\{ \psi : [0,\infty) \to \Z_+ \colon \,
  \mbox{$\psi$ c\`adl\`ag, piece-wise constant, only jumps of size $1$}\big\},
\end{equation}
which is a Polish space as a closed subset of the (usual) Skorokhod
space $\mathcal{D}$. For given $\eta_n \in \Z_+^{\Z^d}$, define
$\eta_{n+1} \in \Z_+^{\Z^d}$ via
\begin{equation}
  \label{eq:flowdef}
  \eta_{n+1}(x) \coloneqq \sum_{y \, : \, \norm{x-y} \leq R_p}
  U^{(y,x)}_n\big(f(y;\eta_n)\big),  \quad x\in \Z^d.
\end{equation}
Note that for each $x$, the right-hand side of \eqref{eq:flowdef} is a
finite sum of (conditionally) Poisson random variables with finite
means bounded by $\norm{f}_\infty$. Thus, \eqref{eq:flowdef} is well
defined for any initial condition -- in this discrete time scenario,
no growth condition at infinity, etc.\ is necessary. Furthermore we
note that by well known properties of Poisson processes
$\eta_{n+1}$, given $\eta_n$, is a family of conditionally
independent random variables with
\begin{align}
  \label{eq:eta-rm}
  \eta_{n+1} (x)  \sim  \mathrm{Pois}\Bigl( \sum_{y \in \Z^d}
  p_{yx} f(y;\eta_{n}) \Bigr), \quad x \in \Z^d.
\end{align}

For $-\infty < m < n$ set
\begin{align}
  \label{eq:Gmn}
  \mathcal{G}_{m,n} \coloneqq \sigma(U_k^{(x,y)} : m \leq k < n, \, x,
  y \in \Z^d).
\end{align}
By iterating \eqref{eq:flowdef}, we can define a random family of
$\mathcal{G}_{m,n}$-measurable mappings
\begin{align}
  \label{eq:Phimn}
 \Phi_{m,n} : \Z_+^{\Z^d} \to \Z_+^{\Z^d}, \; -\infty < m<n \quad
  \text{such that} \quad \eta_n = \Phi_{m,n}(\eta_m).
\end{align}
To this end define $\Phi_{m,m+1}$ as in
\eqref{eq:flowdef} via
\begin{align}
  \label{eq:flowdef2}
  (\Phi_{m,m+1}(\zeta))(x)
  & \coloneqq \sum_{y \, : \, \norm{x-y} \leq R_p} \hspace{-0.8em}
    U^{(y,x)}_m\big(f(y;\zeta)\big) \quad \mbox{for $y \in \Z^d$ and
    $\zeta\in \Z_+^{\Z^d}$} \\ \intertext{and then put}
  \label{eq:flowdef2a}
  \Phi_{m,n} & \coloneqq \Phi_{n-1,n} \circ \cdots \circ \Phi_{m,m+1}.
\end{align}
Using these mappings we can define the dynamics of
$(\eta_n)_{n=m,m+1,\dots}$ \emph{simultaneously} for all initial
conditions $\eta_m \in \Z_+^{\Z^d}$ for any $m \in \Z$.

Let us for a moment consider the process
$\eta=(\eta_n)_{n=0,1,\dots}$. Obviously, the configuration
$\underline 0 \in \Z_+^{\Z^d}$ is an absorbing state for $\eta$.
Thus, the Dirac measure in this configuration is a trivial invariant
distribution of $\eta$. In \cite{BD07} it is shown that for certain
parameter regions, in particular $m \in (1,4)$ and suitable $\lambda$,
the population survives with positive probability. For $m \in (1,3)$
(and again suitable $\lambda$) the existence and uniqueness of
non-trivial invariant distribution is proven. We recall the relevant
results for $m \in (1,3)$.

\begin{proposition}[Survival and complete convergence, \cite{BD07}]
  Assume \label{prop:surv} $m \in (1,3)$ and let $p$ and $\lambda$ be
  as above.
  \begin{enumerate}[(i)]
  \item There are $\lambda_0^{*}=\lambda_0^{*}(m,p) > 0$ and
    $a^{*}=a^{*}(m,p)> 0$ such that if $\lambda_0 \le \lambda_0^{*}$
    and $\sum_{x\neq 0} \lambda_{0x}\le a^{*}\lambda_0$ then the
    process $(\eta_n)_{n=0,1,\dots}$ survives with positive
    probability (if survival for one step has positive probability)
    and has a unique non-trivial invariant extremal distribution $\bar
    \nu$.
  \item Conditioned on non-extinction, $\eta_n$ converges in
    distribution in the vague topology to $\bar\nu$.
  \end{enumerate}
\end{proposition}

Since we are only interested in the regime when the corresponding
deterministic system, cf.~\eqref{eq:coupl-map} below, is well
controlled and in particular, Proposition~\ref{prop:surv} guarantees
that a non-trivial invariant extremal distribution $\bar \nu$ exists,
we make the following general assumption.
\begin{assumption}
  \label{ass:m-lambda}
  \begin{enumerate}
  \item
  With the notation from
  Proposition~\ref{prop:surv} we assume $m \in (1,3)$ and $\sum_{x\neq
    0} \lambda_{0x}\le a^{*}\lambda_0$.
  \item $\gamma \coloneqq \sum_x \lambda_{0x}$ is sufficiently small.
  \end{enumerate}
\end{assumption}

Note that $a^*$ is determined by the dimension $d$, the parameters
$m$, $p$ and a renormalised $\wt\lambda$ by the requirement that the
left-hand side of \eqref{eq:fgradsmall} at $\zeta \equiv m^*$ must be
strictly smaller than $1$, see Section~\ref{sec:determdyn} below.

Under this assumption we can (and do so from now on) consider the
stationary process $\eta=(\eta_n)_{n \in \Z}$ with $\eta_n$
distributed according to $\bar\nu$. From the informal description
(after \eqref{def:f}) above and the formal definition
\eqref{eq:flowdef} it is clear that the model can be easily enriched
with genealogical information; see e.g.\ Chapter~4 in \cite{Dep08}.
Put
\begin{equation}
  \label{eq:anclinquedyn}
  p_{ \eta}(k; x,y) \coloneqq
  \frac{p_{yx}f(y;\eta_{-k-1})}{\sum_z
    p_{zx}f(z;\eta_{-k-1})}, \quad x,y \in
  \Z^d, \, k \in \Z_+
\end{equation}
with some arbitrary convention if the denominator is $0$. For a given
$\eta$, conditioned on $\eta_0(0)>0$, let
$X \coloneqq (X_k)_{k=0,1,2,\dots}$ be a time-inhomogeneous Markov
chain with
\begin{align}
  \label{eq:defX-real}
  X_0 = 0, \quad \text{and} \quad \Pr(X_{k+1}=y \, | \, X_k=x,
   \eta) = p_{ \eta}(k; x,y).
\end{align}
This is the dynamics of the space-time embedding of the ancestral
lineage of an individual sampled at random from the (space-time)
origin at stationarity, conditioned on the (full) space-time
configuration $\eta$. Note that given $\eta$, we see from
\eqref{eq:flowdef} that the number of offspring coming from $y$ in
generation $-k-1$ that moved to $x$ is given by
$U^{(y,x)}_{-k}\big(f(y;\eta_{-k-1})\big)$ which is
$\text{Pois}\big(p_{yx}f(y;\eta_{-k-1}) \big)$-distributed conditional on
the sum over all $y$ in the neighbourhood of $x$ being equal to
$\eta_{-k}(x)$. Since a vector of independent Poisson random variables
conditioned on its total sum has a multinomial distribution we see
that the dynamics of the ancestral lineage are indeed given by
\eqref{eq:defX-real}.

\smallskip

Our main result in this section is the following theorem.

\begin{theorem}[LLN and averaged CLT]
  Assume \label{thm:lln-clt-rm} $d\ge 1$,
  let the Assumption~\ref{ass:m-lambda} be satisfied
  and let $\eta=(\eta_n)_{n \in \Z}$ be the stationary process
  conditioned on $\eta_0(0)>0$. For the
  random walk $(X_k)_{k=0,1,\dots}$ defined in \eqref{eq:defX-real} we
  have
\begin{align}
    \label{eq:LLN}
    P_\eta\Big( \frac1k X_k \to 0 \,\Big) = 1
    \quad \text{for} \quad \text{$\Pr \bigl(\,\cdot\, | \,
  \eta_0(0)>0 \bigr)$-a.a.\ $\eta$},
  \end{align}
  and for any $g \in C_b(\R^d)$
  \begin{align}
    \label{eq:annealedCLT}
    \E \Big[ g\left( X_k/\sqrt{k}\, \right) \, \Big| \, \eta_0(0)>0 \, \Big]
    \xrightarrow{n \to \infty}   \Phi(g),
  \end{align}
  where $\Phi$ is a non-trivial $d$-dimensional normal law and
  $\Phi(g) \coloneqq \int g(x)\, \Phi(dx)$.
\end{theorem}
\begin{proof}
  The assertions of the theorem follow from a combination of
  Proposition~\ref{prop:flowcoupling} and Theorem~\ref{thm:abstr-regen}.
\end{proof}

\subsection{Deterministic dynamics}
\label{sec:determdyn}

For comparison, we consider the dynamical system (such systems are
also called a coupled map lattices)
$\zeta \coloneqq (\zeta_n)_{n =0,1,\dots}$ on $[0,\infty)^{\Z^d}$
defined by
\begin{align}
  \label{eq:coupl-map}
  \zeta_n(x) \coloneqq \sum_{y \in \Z^d}
  p_{yx}f(y;\zeta_{n-1}), \quad x \in \Z^d, \: n \in \N
\end{align}
with $f$ from \eqref{def:f} and arbitrary initial condition
$\zeta_0 \in [0,\infty)^{\Z^d}$ (cf.~\cite[Eq.~(5)]{BD07}). It is
easily seen from \eqref{def:f} that with
\begin{align}
  \label{eq:mstar}
  m^* = m^*(\lambda) = \frac{m-1}{\sum_z \lambda_{0,z}},
\end{align}
$\zeta^*(\cdot) \equiv m^*$ is an equilibrium of the dynamical system
$\zeta$. Furthermore, setting
\begin{align*}
  \gamma \coloneqq \sum_z \lambda_{0,z}, \quad \wt\lambda_{xy} \coloneqq
  \lambda_{xy}/\gamma \quad \text{and} \quad \wt{\zeta}_n(x) \coloneqq
  \gamma \zeta_n(x)
\end{align*}
we see from \eqref{def:f} that $(\wt{\zeta}_n)_{n=0,1,\dots}$ solves
\begin{align}
  \label{def:ftilde}
  \wt\zeta_n(x) & = \sum_{y \in \Z^d } p_{yx} \wt\zeta_{n-1}(y) \big( m -
                  \wt\lambda_0 \wt\zeta_{n-1}(y) - \sum_{z \ne
                  x} \wt\lambda_{yz} \wt\zeta_{n-1}(z)  \big)^+,
\end{align}
i.e., \eqref{eq:coupl-map} with $\lambda$ in the function $f$ replaced
by $\wt\lambda$. Thus, we can and shall assume $\gamma = 1$ for the
rest of this subsection.

\begin{lemma}
  There \label{lem:deterdyn} exist $\alpha_0<\alpha <m^*< \beta$,
  $\varepsilon=\varepsilon(m, \lambda)>0$, $R_0$, $k_0$, $N_0$ and
  $s_0$ such that for all $R \ge R_0$ the following assertions hold:
  \begin{enumerate}[(i)]
  \item If
    $\zeta_0(y) \in [\alpha, \beta] \text{ for all } y \in B_{R}(x)$
    then
    \begin{align}
      \label{eq:0to1}
      \zeta_n(y) & \in \big[(1+\varepsilon)\alpha,
                   \beta/(1+\varepsilon)\big] \text{ for all } n \ge N_0, \:
                   \norm{y-x} \le R + s_0(n-N_0), \\ \intertext{and}
      \label{eq:0to11}
      \zeta_n(y) & \ge \alpha_0 \text{ for all } n \ge 1, \: \norm{y-x}
                   \le R-k_0 + s_0n.
    \end{align}
  \item For $(\zeta(y))_{y \in B_{R_\lambda}(x)} \in
    [\alpha,\beta]^{B_{R_\lambda}(x)}$ we have
    \begin{align}
      \label{eq:fgradsmall}
      \sum_{y \in B_{R_\lambda}(x)} \left| \frac{\partial}{\partial
          \zeta(y)} f(x; \zeta) \right| < 1-\frac{\varepsilon}2.
    \end{align}
  \item One can choose $\beta-\alpha>0$ arbitrarily close to $0$.
  \end{enumerate}

\end{lemma}
\begin{proof}
  Assertions (i) and (iii) follow from Lemma~11 and 12 (and arguments in their
  proofs) in \cite{BD07}. For (ii) see the proof of Lemma~13 and in
  particular Eq.~(40) in \cite{BD07}.
\end{proof}

\begin{remark}[Interpretation of Lemma~\ref{lem:deterdyn}]
  Assertion (i) in the above lemma means that if $\zeta_0$ in the
  neighbourhood of $x$ is in the interval $[\alpha,\beta]$ around
  $m^*$ then the regions around $x$ where $\zeta_n$ is bounded away
  from $0$ and where it is close to $m^*$ grow at positive speed
  (after a finite number of steps). Assertion (ii) means that the
  equilibrium $\zeta^*(\cdot) \equiv m^*$ is attracting.
\end{remark}

\subsection{Coupling reloaded}

\begin{remark}[Initial/boundary conditions on certain space-time regions]
  Note that for any $n \in \N$, $\Phi_{0,n}$ as defined in
  (\ref{eq:flowdef2}) can be viewed as a function of
  $(U^{(x,y)}_m : 0 \leq m < n, \, x,y \in \Z^d)$.

  Let $L \in \N$, $R_p$ the range of $p$, put
  \begin{equation}
    \mathsf{cone}(L,R_p) \coloneqq \big\{ (x,n) \in \Z^d \times \Z_+
    \, \colon \, \norm{x} \leq L + R_p n \big\}
  \end{equation}
  (recalling \eqref{def:cone}, we have
  $\mathsf{cone}(L,R_p) = \cup_{h>0} \mathsf{cone}(L,R_p,h)$). For
  given values of $\eta_k(x)$,
  $(x,k) \in \big((\Z^d \times \Z_+) \setminus \mathsf{cone}(L,R_p)
  \big) \cup ([-L,L]^d \times \{0\})$
  (we can view the latter set as a `space-time boundary' of
  $\mathsf{cone}(L,R_p)$), we can define $\eta_n$ consistently inside
  $\mathsf{cone}(L,R_p)$ through \eqref{eq:flowdef}.

  In fact, we can think of constructing the space-time field $\eta$ in
  a two-step procedure: First, generate the values outside
  $\mathsf{cone}(L,R_p)$ (in any way consistent with the model), then,
  conditionally on their outcome, use \eqref{eq:flowdef} inside.
\end{remark}

\begin{proposition}
  \label{prop:flowcoupling}
  Let Assumption~\ref{ass:m-lambda}~1.\ be fulfilled.
  For any $\varepsilon>0$ we can find $\gamma^*$ and such that
  if  $\gamma \coloneqq \sum_x
  \lambda_{0x} \le \gamma^*$ there exists a spatial scale
  $L_{\mathrm{s}}$ and a temporal scale $L_{\mathrm{t}}$, a set of
  good configurations $G_\eta$ and a set of good Poisson process
  realisations $G_U
  \subset \wt{\mathcal{D}}^{B_{4 L_{\mathrm{s}}}(0) \times
    \{1,2,\dots,L_{\mathrm{t}}\}}$ with $\Pr\big(
  \restr{U}{\mathsf{block}_4(0,0)} \in G_U \big) \ge 1-\varepsilon$
  such that the contraction and coupling conditions
  \eqref{eq:contraction}, \eqref{eq:propagation.coupling} from
  Section~\ref{sect:abstr-setup} are fulfilled. Furthermore the random
  walk defined in \eqref{eq:defX-real} satisfies \eqref{eq:abstr-sb}
  in Section~\ref{sect:abstr-setup}.
\end{proposition}

\begin{proof}
  The crucial idea is that using the flow version \eqref{eq:flowdef}
  we can augment the coupling argument in Lemma~13 in \cite{BD07} to
  work with a set of (good) initial conditions
  \begin{align}
    \label{eq:defI}
    \{\eta_0^{(i)}: i \in I\} = \big\{ \eta \in \Z_+^{\Z^d} :
    \alpha/\gamma \le \eta(x) \le \beta/\gamma \text{
    for } x \in B_{2L_{\mathrm{s}}}(0) \big\}
  \end{align}
  with $\alpha, \beta$ from \eqref{eq:0to1} and the (uncountable)
  index set $I$ being defined implicitly here.

  \smallskip

 The proof consists of 6 steps.
  For parameters $K'_{\mathrm{t}} \gg K_{\mathrm{s}} \gg
  K''_{\mathrm{t}}$ to be suitably tuned below, we set
  \begin{align*}
    L_{\mathrm{s}} & = \lceil K_{\mathrm{s}} \log(1/\gamma) \rceil \\
    L_{\mathrm{t}} & = L'_{\mathrm{t}} + L''_{\mathrm{t}} \quad
                     \text{with} \quad L'_{\mathrm{t}} = \lceil
                     K'_{\mathrm{t}} \log(1/\gamma) \rceil,
                     \; L''_{\mathrm{t}} = \lceil K''_{\mathrm{t}}
                     \log(1/\gamma) \rceil.
  \end{align*}
  In the first step, we use the propagation properties of the deterministic
  system as described in Lemma~\ref{lem:deterdyn} together with the
  fact that for small $\gamma$, the relative fluctuations of the
  driving Poisson processes are typically small to ensure that
  after time $L'_{\mathrm{t}}$, the `good region' has increased sufficiently.

\smallskip

  In the second step we use the flow version \eqref{eq:flowdef} and
  its contraction properties to ensure that in a subregion, after
  $L''_{\mathrm{t}}$ steps, coupling has occurred with high probability.

\smallskip

  Several copies of such subregions are then glued together in Steps~3
  and~4. In Step~5 we use the fact that in a good region, the relative
  fluctuations of $\eta$ are small so that $p_{\eta}(k; x,y)$ is
  close to the deterministic kernel $p_{xy}$; this ensures
  \eqref{eq:abstr-sb}. Finally, in the last step we collect the
  requirements on the various constants that occurred before and verify
  that they can be fulfilled consistently.

  \medskip

  \noindent \textit{Step~1}.\ Let
  \begin{align*}
    \mathcal{X}_1 \coloneqq \Big\{ \max_{\norm{x}, \norm{y} \le 5
      L_{\mathrm{s}}, p_{xy}>0, 0 < n \le L'_{\mathrm{t}}} \sup_{u \ge
      \alpha_0/\gamma} \Big| \frac{U^{(x,y)}_n(u)}{p_{xy} u} - 1 \Big|
    \le \delta \Big\}
  \end{align*}
  with $\alpha_0$ from Lemma~\ref{lem:deterdyn}.
  By standard large deviation estimates for Poisson processes,
  we have
  \begin{align}
    \label{eq:X1}
    \Pr(\mathcal{X}_1) \ge 1 - (10 L_{\mathrm{s}} R_p)^d L'_{\mathrm{t}}
    \exp(- c \alpha_0/\gamma)
  \end{align}
  (for some fixed constant $c>0$) which can be made arbitrarily close
  to $1$ by choosing $\gamma$ small.

 \smallskip

  By iterating \eqref{eq:0to1} in combination with \eqref{eq:0to11} we
  see that
  \begin{align}
    \label{eq:goodgrow}
    \mathcal{X}_1 \cap \big\{ \eta_0(x) \in [\alpha/\gamma,
    \beta/\gamma] \text{ for } x \in B_{2L_{\mathrm{s}}}(0) \big\}
    \subset \big\{ \eta_{L'_{\mathrm{t}}}(y) \in [\alpha/\gamma,
    \beta/\gamma] \;\; \text{for $y \in B_{5L_{\mathrm{s}}}(0)$}
    \big\}
  \end{align}
  if the ratio $L'_{\mathrm{t}}/L_{\mathrm{s}}$ is chosen sufficiently large.
  To verify this note that we can consider $\eta$ as a perturbation
  of the deterministic system $\zeta$ from \eqref{eq:coupl-map} and
  on $\mathcal{X}_1$ the relative size of the perturbation is small
  when $\gamma$ is small (cf.~\cite[Eq.~(13) and the proof of Lemma~7]{BD07}).

  \medskip

  \noindent \textit{Step~2}.\ Let
  $G_0 \subset \wt{\mathcal{D}}^{B_{3 L_{\mathrm{s}}}(0) \times \{1,\dots,L''_{\mathrm{t}}\}}$
  be the set of Poisson process path configurations in the space-time box
  $B_{3 L_{\mathrm{s}}}(0) \times \{1,\dots,L''_{\mathrm{t}}\}$ with the property
  \begin{multline}
    \restr{\eta_0}{B_{2 L_{\mathrm{s}}}(0)} \in [\alpha/\gamma,
    \beta/\gamma]^{B_{2 L_{\mathrm{s}}}(0)}, \:
    \restr{U}{B_{3 L_{\mathrm{s}}}(0) \times \{1,\dots,L''_{\mathrm{t}}\}} \in G_0 \\
    \quad \Longrightarrow \quad
    \big(\Phi_{1,L''_{\mathrm{t}}}(\eta_0)\big)(x) =
    \big(\Phi_{1,L''_{\mathrm{t}}}(\eta^{\mathrm{ref}})\big)(x) \;
    \text{for } \norm{x} \le L_{\mathrm{s}}
  \end{multline}
  with $\Phi_{1,L''_{\mathrm{t}}}$ as in \eqref{eq:flowdef2a} and
  $\eta^{\mathrm{ref}} \equiv \lceil m^* \rceil$.

  Observe that for $x \in \Z^d$, $n \in \Z_+$
  \begin{align}
    \begin{split}
    \sup_{i\in I} \eta^{(i)}_n(x) - \inf_{i\in I} \eta^{(i)}_n(x)
    & = \sup_{i\in I} \sum_y U^{(y,x)}_{n-1}\big(f(y;\eta^{(i)}_{n-1})\big) - \inf_{i\in
      I} \sum_y U^{(y,x)}_{n-1}\big(f(y;\eta^{(i)}_{n-1})\big) \\
    & \le \sum_y \left( \sup_{i\in I}
      U^{(y,x)}_{n-1}\big(f(y;\eta^{(i)}_{n-1})\big)- \inf_{i\in
      I} U^{(y,x)}_{n-1}\big(f(y;\eta^{(i)}_{n-1})\big)\right) \\
    & = \sum_y U^{(y,x)}_{n-1}\Big( \sup_{i\in I}
      f(y;\eta^{(i)}_{n-1})\Big) - U^{(y,x)}_{n-1}\Big( \inf_{i\in I}
      f(y;\eta^{(i)}_{n-1}) \Big).
    \end{split}
  \end{align}
  Thus,
  \begin{align}
    \label{eq:econdrange}
    \E\Big[ \sup_{i\in I} \eta^{(i)}_n(x) - \inf_{i\in I}
    \eta^{(i)}_n(x) \Big| \mathcal{F}_{n-1} \Big] \le \sum_y p_{yx}
    \Big( \sup_{i\in I} f(y;\eta^{(i)}_{n-1}) - \inf_{i\in I}
    f(y;\eta^{(i)}_{n-1}) \Big)
  \end{align}
  and we can now use contraction properties of $f$ near its fixed
  point, analogous to the proof of (44), (45) in Lemma~13 in
  \cite{BD07}. Note that if all $\eta^{(i)}_{n-1}(y)$, $i \in I$ are
  (locally around $x$) in the neighbourhood $[\alpha,\beta]$ of $m^*$
  (as required for \eqref{eq:fgradsmall}), there are
  $\tilde{\eta}_{y,z} \in [\alpha,\beta]^{B_{R_\lambda}(y)}$ such that
  \begin{align}
    \label{eq:econdrange2}
    \sup_{i\in I} f(y;\eta^{(i)}_{n-1}) - \inf_{i\in I}
    f(y;\eta^{(i)}_{n-1}) \le \sum_{z \in B_{R_\lambda}(y)}
    \big| {\textstyle \frac{\partial}{\partial \eta(z)}} f(y; \tilde{\eta}_{y,z}) \big|
    \Big( \sup_{i\in I} \eta^{(i)}_{n-1}(z) - \inf_{i\in I}
    \eta^{(i)}_{n-1}(z) \Big).
  \end{align}

  Put
  \begin{align}
    \psi_R(\eta) \coloneqq \ind{ \eta(x) \in [\alpha/\gamma,
    \beta/\gamma] \text{ for } \norm{x} \le R} .
  \end{align}
  We have on $\{ \inf_{i \in I} \psi_{R+R_p+R_\lambda}(\eta^{(i)}_{n-1})=1 \}$
{\allowdisplaybreaks
 \begin{align}
 & \frac{1}{\abs{B_R(0)}} \sum_{x\in B_R(0)} \E\Bigl[ \Bigl(\sup_{i\in I} \psi_R(\eta^{(i)}_n) \eta^{(i)}_n(x) - \inf_{i\in I} \psi_R(\eta^{(i)}_n) \eta^{(i)}_n(x) \Bigr) \big| \FF_{n-1}\Bigr]                                                      \notag\\
 & \quad \le \frac{1}{|B_R(0)|} \sum_{x\in B_R(0)} \E[ \sup_{i\in I} \eta^{(i)}_n(x) - \inf_{i\in I} \eta^{(i)}_n(x) \big| \FF_{n-1}]                                                                                                                 \notag\\
 & \quad \le \frac{1}{\abs{B_R(0)}}\sum_{x\in B_R(0)} \sum_{y \in B_{R_p}(x)} p_{yx} \sum_{z \in B_{R_\lambda}(y)} \abs{\nabla_z f(y;\tilde\eta_z)} \Bigl(\sup_{i\in I} \eta^{(i)}_{n-1}(z) - \inf_{i\in I} \eta^{(i)}_{n-1}(z) \Bigr)                \notag\\
 & \quad \le \sum_{z \in B_{R+R_p+R_\lambda}(0)} \Bigl(\sup_{i\in I} \eta^{(i)}_{n-1}(z) - \inf_{i\in I} \eta^{(i)}_{n-1}(z) \Bigr) \frac1{\abs{B_R(0)}} \sum_{y \in B_{R_\lambda}(z)} \abs{\nabla_z f(y;\tilde\eta_z )} \sum_{x \in B_R(0)} p_{xy}   \notag\\
 & \quad \le \frac{\abs{B_{R+R_p+R_\lambda}(0)}}{\abs{B_R(0)}} \Bigl(1-\frac{\e}2\Bigr) \frac{1}{\abs{B_{R+R_p+R_\lambda}(0)}} \sum_{z \in B_{R+R_p+R_\lambda}(0)} \Bigl(\sup_{i\in I} \eta^{(i)}_{n-1}(z) - \inf_{i\in I} \eta^{(i)}_{n-1}(z) \Bigr) \notag\\
 & \quad\le c(\e) \frac{1}{\abs{B_{R+R_p+R_\lambda}(0)}} \sum_{z \in B_{R+R_p+R_\lambda}(0)} \Bigl(\sup_{i\in I} \eta^{(i)}_{n-1}(z) - \inf_{i\in I} \eta^{(i)}_{n-1}(z) \Bigr),\label{eq:contrit}
\end{align}}%
where we used \eqref{eq:econdrange} and \eqref{eq:econdrange2} in the second inequality
and assume that $R$ is so large that
\begin{align}
  \label{eq:est.1.term.b}
  \frac{\abs{B_{R+R_p+R_\lambda}(0)}}{\abs{B_R(0)}}
    \Bigl(1-\frac{\e}2\Bigr) \le c(\e) < 1 .
\end{align}
Note that the factor $(1-\frac{\varepsilon}2)$ comes from \eqref{eq:fgradsmall}.

We can iterate \eqref{eq:contrit} for
$n=L''_{\mathrm{t}}, L''_{\mathrm{t}}-1,\dots, 1$ to obtain on
$\mathcal{X}_1$ (which in particular implies
$\psi_{L_{\mathrm{s}}+k(R_p+R_\lambda)}(\eta_{n-k})=1$ for $k=1,2,\dots,n-1$) that
\begin{align}
  \label{eq:est.1.term}
  \begin{split}
    \frac1{\abs{B_{L_{\mathrm{s}}}(0)}}
    & \sum_{x \in B_{L_{\mathrm{s}}}(0)} \E\Bigl[ \Bigl(\sup_{i\in I}
    \psi_{L_{\mathrm{s}}}(\eta^{(i)}_{L''_{\mathrm{t}}})
    \eta^{(i)}_{L''_{\mathrm{t}}}(x) - \inf_{i\in I}
    \psi_{L_{\mathrm{s}}}(\eta^{(i)}_{L''_{\mathrm{t}}})\eta^{(i)}_{L''_{\mathrm{t}}}(x)
    \Bigr) \big| \FF_{0}\Bigr] \\
    & \le c(\e)^{L''_{\mathrm{t}}}
    \frac{1}{\abs{B_{L_{\mathrm{s}}+L''_{\mathrm{t}}(R_p+R_\lambda)}(0)}}
    \sum_{z \in B_{L_{\mathrm{s}}+L''_{\mathrm{t}}(R_p+R_\lambda)}(0)}
    \Bigl(\sup_{i\in I} \eta^{(i)}_0(z) - \inf_{i\in I}
    \eta^{(i)}_0(z) \Bigr) \\
    & \le c(\e)^{L''_{\mathrm{t}}} \frac{\beta-\alpha}{\gamma}.
  \end{split}
\end{align}
On the event
$\{ \inf_{i \in I}
\psi_{R+L''_{\mathrm{t}}(R_p+R_\lambda)}(\eta^{(i)}_0)=1 \}$ via
Markov inequality \eqref{eq:est.1.term} yields
\begin{align}
  \Pr\Big( \max_{\norm{x} \le L_{\mathrm{s}}}
  \big( \sup_{i\in I} \psi_{L_{\mathrm{s}}}(\eta^{(i)}_{L''_{\mathrm{t}}}) \eta^{(i)}_{L''_{\mathrm{t}}}(x)
  - \inf_{i\in I} \psi_{L_{\mathrm{s}}}(\eta^{(i)}_{L''_{\mathrm{t}}})\eta^{(i)}_{L''_{\mathrm{t}}}(x) \big)
  \ge 1 \, \Big| \, \mathcal{F}_0 \Big)
  \le \abs{B_{L_{\mathrm{s}}}(0)} c(\e)^{L''_{\mathrm{t}}} \frac{\beta-\alpha}{\gamma}
\end{align}
and the right hand side can be made as small as we like by choosing
$\gamma$ small.

Hence for
$\mathcal{X}_2 := \{ \restr{U}{B_{3 L_{\mathrm{s}}}(0) \times
  \{1,\dots,L''_{\mathrm{t}}\}} \in G_0 \}$ we have
\begin{align}
  \label{eq:X2}
  \Pr(\mathcal{X}_2) \ge 1- \abs{B_{L_{\mathrm{s}}}(0)}
  c(\e)^{L''_{\mathrm{t}}} \frac{\beta-\alpha}{\gamma}.
\end{align}

 \medskip

 \noindent
 \textit{Step~3}.\ Let $\mathcal{X}_2(y,k)$ be the event that
 $\mathcal{X}_2$ occurs in the space-time box whose `bottom' is
 centred at $(y,k)$, i.e.\
 $\mathcal{X}_2(y,k) = \{ \restr{U}{B_{3 L_{\mathrm{s}}}(y) \times
   \{k+1,\dots,k+L''_{\mathrm{t}}\}} \in G_0 \}$. By construction, on
 \begin{align}
   \mathcal{X}_3 \coloneqq \mathcal{X}_1 \cap
   \bigcap_{\substack{j \in \{-2,-1,\dots,2\}, \\ k=1,\dots,d}}
   \mathcal{X}_2(j L_{\mathrm{s}}e_k , L'_{\mathrm{t}})
 \end{align}
 we have
 \begin{align}
   \eta_0(x) \in [\alpha/\gamma, \beta/\gamma] \text{ for } x \in B_{2L_{\mathrm{s}}}(0)
   \quad \Longrightarrow \quad
   \eta_{L_{\mathrm{t}}}(y) =
   \big(\Phi_{1,L_{\mathrm{t}}}(\eta^{\mathrm{ref}})\big)(y) \;
   \text{for } \norm{y} \le 3 L_{\mathrm{s}},
 \end{align}
 i.e.\ \eqref{eq:contraction} holds. In particular
 \begin{align}
   \label{eq:X3}
   \Pr(\mathcal{X}_3) \ge
   \Pr(\mathcal{X}_1) - 5^d(1- \Pr (\mathcal{X}_2))
 \end{align}

  \medskip

  \noindent \textit{Step~4}.\ On
  \begin{align}
    \mathcal{X}_4 \coloneqq \mathcal{X}_3 \cap
    \bigcap_{j=0,\dots,\lceil L'_{\mathrm{t}}/L''_{\mathrm{t}} \rceil}
    \mathcal{X}_2(0, j
    L''_{\mathrm{t}}),
  \end{align}
  \eqref{eq:propagation.coupling} holds as well, and we have
  \begin{align}
    \label{eq:X4}
    \Pr(\mathcal{X}_4) \ge  \Pr(\mathcal{X}_3) -
  \lceil L'_{\mathrm{t}}/L''_{\mathrm{t}} \rceil (1- \Pr
  (\mathcal{X}_2))
  \end{align}

  \medskip

  \noindent
  \textit{Step~5}.\ Note that on $\mathcal{X}_4$ we have
  \begin{align*}
    \eta_0(x) \in [\alpha/\gamma, \beta/\gamma] \text{ for } x \in
    B_{2L_{\mathrm{s}}}(0) \quad  \Longrightarrow \quad
    \eta_{n}(y)  \in [\alpha/\gamma, \beta/\gamma] \; \text{for }
    \norm{y} \le 2 L_{\mathrm{s}}, n =1, \dots, L_{\mathrm{t}}.
  \end{align*}
  Then \eqref{eq:anclinquedyn} implies
  \begin{align*}
    \frac{\alpha}{\beta} p_{xy} \le p_{\eta}(k; x,y) \le
    \frac{\beta}{\alpha} p_{xy} \quad
    \text{ for } x,y \in
    B_{2L_{\mathrm{s}}}(0), \; k=1,\dots, L_{\mathrm{t}},
  \end{align*}
  hence the total variation distance between
  $p_{\eta}(k; x, \cdot)$ and $p_{x,\cdot}$ is at most
  $(1-\frac{\alpha}{\beta}) \vee (\frac{\beta}{\alpha}-1)$ uniformly
  inside this space-time block. We use
  Lemma~\ref{lem:deterdyn},~\emph{(iii)} to make this so small that
  coupling arguments as in the proof of Lemma~\ref{lem:aprioribd}
  (with a comparison random walk that has a deterministic drift
  $d_\mathrm{max} \ll
  L_{\mathrm{s}}/(L'_{\mathrm{t}}+L''_{\mathrm{t}}$)
  show \eqref{eq:abstr-sb}.

  \medskip

  \noindent
  \textit{Step~6}.\ Finally, we verify that the constants
  $K_{\mathrm{s}}, K'_{\mathrm{t}}, K''_{\mathrm{t}}$ can be chosen
  consistently so that all intermediate requirements are fulfilled.

  \begin{enumerate}
  \item The right-hand side of \eqref{eq:X1} can be chosen arbitrarily
    close to $1$ for any choice of
    $K_{\mathrm{s}}, K'_{\mathrm{t}}, K''_{\mathrm{t}}$ by making
    $\gamma$ small.
  \item \eqref{eq:goodgrow} requires that
    $K'_{\mathrm{t}} > \frac{3}{s_0} K_{\mathrm{s}}$ with $s_0$ from
    \eqref{eq:0to1}, \eqref{eq:0to11}.
    \item \eqref{eq:est.1.term}, which uses \eqref{eq:est.1.term.b}
      $L''_{\mathrm{t}}$ times, requires that
      $L_{\mathrm{s}} - (R_p+R_\lambda) L''_{\mathrm{t}}$ is large.
      This is achieved when
      $K_{\mathrm{s}} \gg (R_p+R_\lambda) K''_{\mathrm{t}}$ (and
      $\gamma$ is small).
    \item The right-hand side of \eqref{eq:X2} can be made close to
      $1$ if $K''_{\mathrm{t}} > (-\log c(\varepsilon))^{-1}$ (and
      $\gamma$ is small). This also implies that the right-hand side
      of \eqref{eq:X3} can be chosen arbitrarily close to $1$.
    \item For \eqref{eq:X4} note that
      $\frac{L'_{\mathrm{t}}}{L''_{\mathrm{t}}} \approx
      \frac{K'_{\mathrm{t}}}{K''_{\mathrm{t}}}$
      is a fixed ratio when $\gamma$ is small, and
      $1-\Pr(\mathcal{X}_2)$ can be made small by choosing $\gamma$
      small.
  \end{enumerate}
  We see that for $\gamma \le \gamma^*$ for some $\gamma^*>0$, all
  requirements can be fulfilled e.g.\ by choosing
  $K''_{\mathrm{t}} \coloneqq 2/(-\log c(\varepsilon))$,
  $K_{\mathrm{s}} \coloneqq C (R_p+R_\lambda) K''_{\mathrm{t}}$ with
  some large $C$ and
  $K'_{\mathrm{t}} \coloneqq \frac{6}{s_0} K_{\mathrm{s}}$.
\end{proof}

\section{Discussion of further classes of population models}
\label{sec:disc-furth-class}

\label{rem:furthermodels}
\rm While we analysed in this article only one explicit spatial
population model, namely logistic branching random walks (LBRW)
defined in \eqref{eq:flowdef} with the dynamics of the space-time
embedding of an ancestral lineage given by \eqref{eq:defX-real}, we do
believe that the same program can be carried out for many related
population models and that LBRW is in this sense prototypical; see
also \cite[Remark~5]{BD07}. We list and discuss some of these in the
following five paragraphs. Note that implementing the details to
verify the conditions of Theorem~\ref{thm:abstr-regen} for these
models will still require quite some technical work and we defer this
to future research.

\paragraph{More general `regulation functions'}
The logistic function $x \mapsto x(m-\lambda x)$ (with $\lambda$ small
enough) whose `spatial version' is used in the definition
\eqref{def:f} can be replaced by some function $\phi: \R_+ \to [0,a]$,
$a \in (0,\infty) \cup \{\infty\}$ with $\phi'(0)>1$, $\lim_{x \to a}
\phi(x) =0$ that possesses a unique attracting fixed point $x^* =
\phi(x^*)>0$.  This will ensure that a result analogous to
Lemma~\ref{lem:deterdyn} holds, and an analogue of
Proposition~\ref{prop:flowcoupling} can be obtained when a suitable
small parameter is introduced.

For example, in the ecology literature, in addition to the logistic
model, also the \emph{Ricker model} corresponding to $\phi(x) = x
\exp(r-\lambda x)$ and the \emph{Hassel model} corresponding to
$\phi(x) = m x/(1 + \lambda x)^b$ are used to describe population
dynamics under limited resources ($r>0$, resp., $m>1$ and $b>0$ are
parameters). Note that in all these cases,
$1/\lambda$ is related to a carrying capacity, so assuming $\lambda$
small means weak competition.

\paragraph{More general families of offspring distributions}
As described in the informal discussion above \eqref{eq:Nxy},
\eqref{eq:eta-rm} can be interpreted as stipulating that each
individual at $y$ in generation $n$ has a Poisson number of offspring
(with mean $f(y;\eta_n)/\eta_n(y)$) which then independently take a
random walk step. One could replace the Poisson distribution by some
another family of distributions $\mathcal{L}(X(\nu))$ on $\N_0$ that
is parametrised by the mean $\E[X(\nu)]=\nu \in [0,\bar{\nu}]$ where
$\bar{\nu} \ge \sup_{y, \eta \not\equiv 0} f(y;\eta)/\eta(y)$ and then
define the model accordingly. If the family of offspring laws
satisfies a suitably quantitative version of the law of large numbers
(cf.\ Step~1 of the proof of Proposition~\ref{prop:flowcoupling}), one can
derive an analogue of Proposition~\ref{prop:flowcoupling}.

For example, one could take an $\N_0$-valued random variable $X$ with
mean $\E[X]=\bar{\nu}$ and $\E[e^{a X}] < \infty$ for some $a>0$ and
then define $X(\nu)$ via independent thinning, i.e.\ $X(\nu)
\stackrel{d}{=} \sum_{i=1}^X \ind{U_i \le \nu/\bar{\nu}}$ where
$U_1,U_2,\dots$ are i.i.d.\ uniform($[0,1]$).

\paragraph{`Moderately' small competition parameters}
As it stands, Theorem~\ref{thm:lln-clt-rm} requires sufficiently (in
fact, very) small competition parameters (cf.\
Assumption~\ref{ass:m-lambda}). This is owed to the fact that our
abstract `work-horse' Theorem~\ref{thm:abstr-regen} requires (very)
small $\varepsilon_U$ in Assumption~\ref{ass:coupling} and
$\varepsilon$ in Assumption~\ref{ass:cl-srw} (we in fact did not spell
out explicit bounds). In simulations of LBRW one observes also for
moderately small competition parameters $\lambda_{xy}$ apparent
stabilisation to a non-trivial `equilibrium' as required by
Assumption~\ref{ass:coupling}. We note that the assumptions of
Theorem~\ref{thm:abstr-regen} are `effective' in the sense that they
only require controlling the system $(\eta_n)$ and the walk in certain
finite space-time boxes. Thus, a suitably quantified version of
Theorem~\ref{thm:abstr-regen} allows at least in principle to
ascertain via simulations that for a given choice of parameters $m$,
$(p_{xy})$ and $(\lambda_{xy})$ the system $(\eta_n)$ has a unique
non-trivial ergodic equilibrium and that the conclusions of
Theorem~\ref{thm:lln-clt-rm}~hold.\looseness=1

\paragraph{Continuous-time and continuous-mass models}
An infinite system of interacting diffusions that can be obtained as a
time- and mass-rescaling of LBRW is considered in
\cite{Etheridge:2004}, see Definition~1.3 there; one can in principle
define an `ancestral lineage' in such a model which will be a certain
continuous-time random walk in (the time-reversal of this) random
environment. It is conceivable that a coarse-graining construction
similar to the one discussed here can be implemented and that in fact
an analogue of Theorem~\ref{thm:lln-clt-rm} can be proved at least for
suitably small interaction parameters.

\paragraph{Reversible Markov systems}
Obviously, we tailored Theorem~\ref{thm:abstr-regen} and its
assumptions to a random walk that moves in the time-reversal of a
non-reversible Markov system $\eta$ which possesses two distinct
ergodic equilibria, our prime example being the (discrete time)
contact process.

If we instead assume that $\eta$ is a reversible Markov system with
local dynamics and `good' mixing properties possessing a unique
equilibrium (for example, the stochastic Ising model at high
temperature), the assumptions from Section~\ref{subsect:abstr-assumpt}
will be fulfilled as well -- this case is in fact easier since `good
blocks' in $\eta_{(\wt{n}+1) L_{\mathrm{t}}}$ as required in
\eqref{eq:contraction} of Assumption~\ref{ass:coupling} will have
uniformly high probability anyway, irrespective of
$\eta_{\wt{n} L_{\mathrm{t}}}$. Assume in addition that we can verify
Assumption~\ref{ass:cl-srw} for the walk. For example, this can be
done by requiring that the walk is a (sufficiently small) perturbation
of a fixed symmetric random walk or by assuming a (small) a priori
bound on the drift. Then we don't need to require the symmetry
assumption~\eqref{eq:abstr-symm} (in fact, the resulting walk can have
non-zero speed). This re-reading of Theorem~\ref{thm:abstr-regen} and
its proof allows to recover a special case of
\cite[Thm.~3.6]{RedigVoellering:2013} where, using entirely different
methods, a CLT is obtained for random walks in dynamic environments
that satisfy sufficiently strong coupling and mixing properties.

\appendix
\section{An auxiliary result}
The following result should be standard, we give here a brief argument
for completeness' sake and for lack of a precise point reference.

\begin{lemma}
  \label{lem:equalstoppingtimes}
  Let $\mathcal{F}=(\mathcal{F}_n)_{n =0,1,\dots}$ be a filtration,
  $T$, $T'$ finite $\mathcal F$-stopping times, and $Y$ a bounded
  random variable. We have
  \begin{equation}
    \label{eq:condexpTeqTp}
    \E\left[ Y \mid \mathcal{F}_T \right] \ind{T=T'}
    = \E\left[ Y \mid \mathcal{F}_ {T'} \right] \ind{T=T'} \quad \text{a.s.}
  \end{equation}
\end{lemma}
\begin{proof}
  Note that $\{T=T'\} \in \mathcal{F}_T \cap \mathcal{F}_{T'}$ because
  \begin{align*}
    \{T=T'\} \cap \{T=n\} = \{T=T'\} \cap \{T'=n\} = \{T=n\} \cap
  \{T'=n\} \in \mathcal{F}_n, \quad n =0,1,\dots.
  \end{align*}
  Furthermore we have
  \begin{align*}
    A \in \mathcal{F}_T \cup \mathcal{F}_ {T'} \quad \Rightarrow \quad
    A \cap \{T=T'\} \in \mathcal{F}_T \cap \mathcal{F}_ {T'}.
  \end{align*}
  To see this note that for $A \in \mathcal{F}_T$
  \begin{align*}
   A \cap \{T=T'\} \cap \{T'=n\}= (A \cap \{T=n\}) \cap \{T'=n\} \in
  \mathcal{F}_n, \quad n =0,1,\dots.
  \end{align*}
  Thus, we obtain $A \cap \{T=T'\} \in \mathcal{F}_ {T'}$ and a
  similar argument for the other case shows the assertion. By
  approximation arguments we find that
  \begin{align*}
    Z\; \text{is $\mathcal{F}_T$-measurable} \quad \Rightarrow \quad Z
    \ind{T=T'} \; \text{is ($\mathcal{F}_T \cap \mathcal{F}_
    {T'}$)-measurable}.
  \end{align*}
  Let $Z$ be a version of
  $\E\left[ Y\ind{T=T'} \mid \mathcal{F}_T \right] = \ind{T=T'}
  \E\left[ Y \mid \mathcal{F}_T \right]$,
  i.e., $Z$ is $\mathcal{F}_T$-measurable,
  $\E\left[ Z \indset{A}\right] = \E\left[ Y \ind{T=T'}
    \indset{A}\right]$
  for all $A \in \mathcal{F}_T$. We may assume that $Z=Z\ind{T=T'}$.
  Then, $Z$ is also a version of
  $\E\left[ Y\ind{T=T'} \mid \mathcal{F}_{T'} \right] = \ind{T=T'}
  \E\left[ Y \mid \mathcal{F}_{T'} \right]$. Furthermore
  $Z=Z\ind{T=T'}$ is also $\mathcal{F}_ {T'}$-measurable and for
  $A'\in\mathcal{F}_{T'}$,
  \begin{align*}
    \E\left[ Z \indset{A'}\right] = \E\left[ Z \indset{A' \cap
        \{T=T'\}}\right] \E\left[ Y \ind{T=T'} \indset{A' \cap
        \{T=T'\}}\right] = \E\left[ Y \ind{T=T'} \indset{A'}\right].
  \end{align*}
  This concludes the proof of the lemma.
\end{proof}


\providecommand{\MR}{\relax\ifhmode\unskip\space\fi MR }
\providecommand{\MRhref}[2]{%
  \href{http://www.ams.org/mathscinet-getitem?mr=#1}{#2}
}
\providecommand{\href}[2]{#2}

\end{document}